\documentclass[10pt]{article}
\usepackage{amsmath,amssymb,amsfonts,latexsym,epsfig,natbib}
\bibpunct{(}{)}{,}{a}{}{}

\makeatletter

\newif\iftwodates
\twodatesfalse

\renewcommand\maketitle{\begin{titlepage}%
  \pagenumbering{Alph}  
  \let\footnotesize\small
  \let\footnoterule\relax
  \let\footnote\thanks
  \null\vfil
  \vskip 30\p@
  \begin{center}%
    {\LARGE \bf \@title \par}%
    \vskip 3em%
    {\large
     \lineskip .75em%
     \begin{tabular}[t]{c}%
       \@author
     \end{tabular}\par}%
     \vskip 1.5em%
  \end{center}\par
  \vfill
  \begin{center}
    \raisebox{1.5cm}{\includegraphics[width=0.58\textwidth]%
      {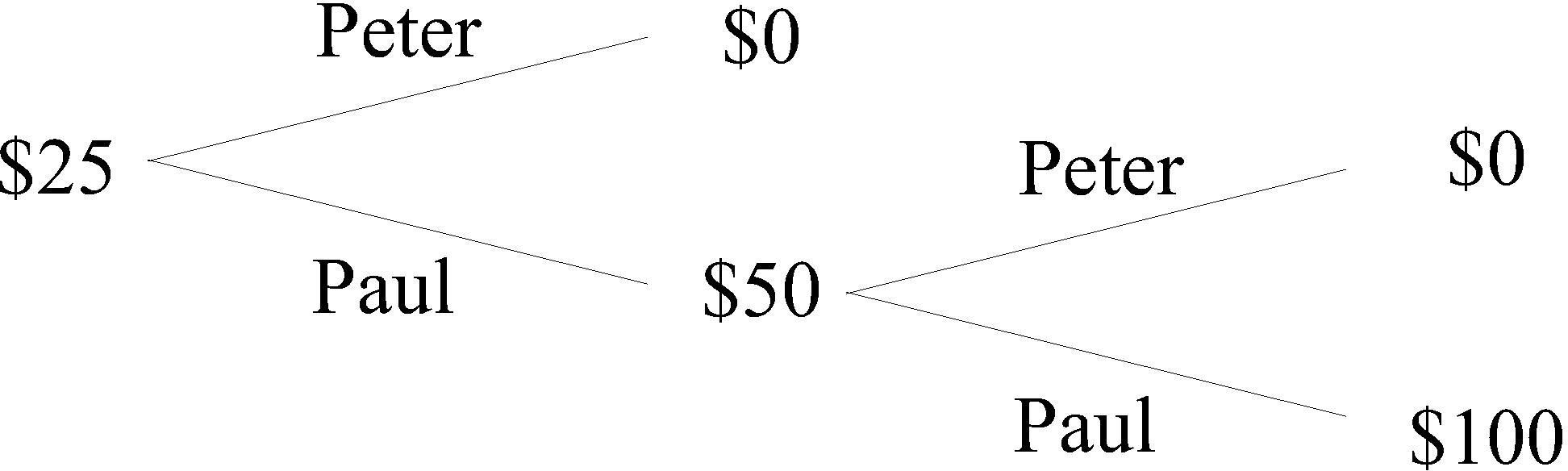}}%
    \hskip 3em%
    \includegraphics[width=0.29\textwidth]%
      {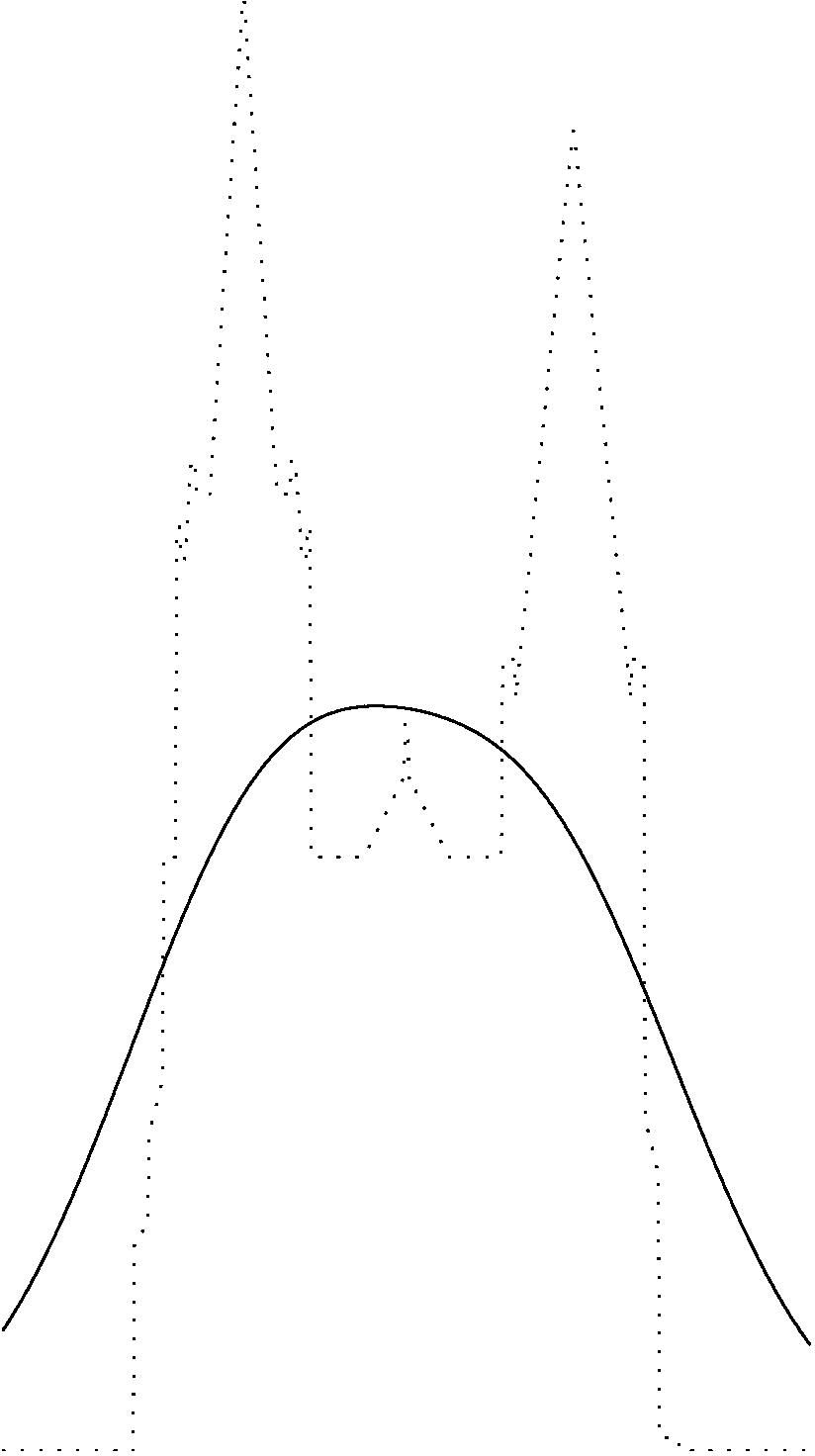}%
  \end{center}
  \@thanks
  \vfill
  \begin{center}
    {\large \bf The Game-Theoretic Probability and Finance Project}
  \end{center}
  \begin{center}
    {\large Working Paper \#\No}
  \end{center}
  \begin{center}
    {\iftwodates\large First posted \firstposted.
    Last revised \@date.\else\large\@date\fi}
  \end{center}
  \begin{center}
    Project web site:\\
    http://www.probabilityandfinance.com
  \end{center}
  \end{titlepage}%
  \setcounter{footnote}{0}%
  \global\let\thanks\relax
  \global\let\maketitle\relax
  \global\let\@thanks\@empty
  \global\let\@author\@empty
  \global\let\@date\@empty
  \global\let\@title\@empty
  \global\let\title\relax
  \global\let\author\relax
  \global\let\date\relax
  \global\let\and\relax
}

\renewenvironment{abstract}{%
  \titlepage\pagenumbering{roman}  
  \null\vfil
  \@beginparpenalty\@lowpenalty
  \begin{center}%
    \Large \bfseries \abstractname
    \@endparpenalty\@M
  \end{center}}%
  {\par\vfill\tableofcontents\thispagestyle{empty}\endtitlepage
  \pagenumbering{arabic}}  

\renewenvironment{thebibliography}[1]
  {\section*{\refname}%
  \addcontentsline{toc}{section}{\refname}
  \@mkboth{\MakeUppercase\refname}{\MakeUppercase\refname}%
  \list{\@biblabel{\@arabic\c@enumiv}}%
    {\settowidth\labelwidth{\@biblabel{#1}}%
    \leftmargin\labelwidth
    \advance\leftmargin\labelsep
    \@openbib@code
    \usecounter{enumiv}%
    \let\p@enumiv\@empty
    \renewcommand\theenumiv{\@arabic\c@enumiv}}%
    \sloppy
    \clubpenalty4000
    \@clubpenalty \clubpenalty
    \widowpenalty4000%
    \sfcode`\.\@m}
    {\def\@noitemerr
    {\@latex@warning{Empty `thebibliography' environment}}%
  \endlist}

\makeatother

\newcommand{\st}{\mathop{|}}
\newcommand{\given}{\mathrel{|}}
\newcommand{\amp}{\mathrel{\&}}

\newcommand{\FFF}{\mathfrak{F}}
\newcommand{\KKK}{\mathfrak{K}}
\newcommand{\GGG}{\mathfrak{G}}

\newcommand{\Prob}{\mathsf{P}}
\newcommand{\Expect}{\mathsf{E}}
\newcommand{\bbbr}{R}		

\newenvironment{Equation*}
  {$$\begin{array}{c}\displaystyle}
  {\end{array}$$}

\newcommand{\BigSkip}{\\[1.6mm]\displaystyle}
\newcommand{\BiggSkip}{\\[3.6mm]\displaystyle}

\newtheorem{lemma}{Lemma}
\newtheorem{Remark}{Remark}

\newtheorem{Example}{Example}

\newenvironment{proof}
  {\trivlist\item[\hskip\labelsep\textbf{Proof}]}
  {\endtrivlist}
\newcommand{\boxforqed}{\rule{.3em}{1.5ex}}
\newcommand{\qedtext}{\unskip\nobreak\hfil
  \penalty50\hskip1em\null\nobreak\hfil\boxforqed
  \parfillskip=0pt\finalhyphendemerits=0\endgraf}

\usepackage[T2A]{fontenc}
\usepackage[utf8]{inputenc}

\title{The origins and legacy of Kolmogorov's \emph{Grundbegriffe}}

\author{Glenn Shafer\\
Rutgers School of Business\\
\texttt{gshafer{\rm @}andromeda.rutgers.edu}
\and
Vladimir Vovk\\
Royal Holloway, University of London\\
\texttt{v.vovk{\rm @}rhul.ac.uk}}

\newcommand{\No}{4}
\twodatestrue
\newcommand{\firstposted}{February 8, 2003}

\begin{document}
\maketitle
\begin{abstract}
April 25, 2003, marked the 100th anniversary of the birth
of Andrei Nikolaevich Kolmogorov, the twentieth century's 
foremost contributor to the mathematical and philosophical
foundations of probability.
The year 2003 was also the 70th anniversary of the publication of 
Kolmogorov's \emph{Grundbegriffe der Wahrscheinlichkeitsrechnung}.

Kolmogorov's \emph{Grundbegriffe} put probability's modern mathematical formalism in place.
It also provided a philosophy of probability---an explanation of how the formalism
can be connected to the world of experience.
In this article, we examine the sources of these two aspects
of the \emph{Grundbegriffe}---the work of the earlier scholars
whose ideas Kolmogorov synthesized.
\end{abstract}


\section{Introduction}

Andrei Kolmogorov's \emph{Grundbegriffe der Wahrscheinlichkeitsrechnung}\nocite{kolmogorov:1933},
which set out the axiomatic basis for modern probability theory, appeared in 1933.  
Four years later, in his opening address to an international colloquium
at the University of Geneva, Maurice Fr\'echet praised
Kolmogorov for organizing and expositing a theory that
\'Emile Borel had created by adding countable additivity to classical probability.  
Fr\'echet put the matter this way
in the written version of his address
(1938b\nocite{frechet:1938b}, p.~54):
\begin{quotation}
  It was at the moment when Mr.~Borel introduced this new kind of additivity
  into the calculus of probability---in 1909, that is to say---that 
  all the elements needed to formulate explicitly 
  the whole body of axioms
  of (modernized classical) probability theory came together.

  It is not enough to have all the ideas in mind, to recall them 
  now and then; one must make sure that their totality is sufficient,
  bring them together explicitly, 
  and take responsibility for saying that nothing further
  is needed in order to construct the theory.

  This is what Mr.~Kolmogorov did.
  This is his achievement.
  (And we do not believe he wanted to claim any others,
  so far as the axiomatic theory is concerned.)
\end{quotation}
\noindent
Perhaps not everyone in Fr\'echet's audience agreed
that Borel had put everything on the table.
But surely many saw the \emph{Grundbegriffe} as a work of synthesis.
In Kolmogorov's axioms, and in his way of relating his axioms to the world of experience,
they must have seen traces of the work of many others---the work of Borel, yes,
but also the work of Fr\'echet himself, and that of Cantelli, Chuprov,
L\'evy, Steinhaus, Ulam, and von Mises.

Today, what Fr\'echet and his contemporaries knew is no longer known.
We know Kolmogorov and what came after;
we have mostly forgotten what came before.
This is the nature of intellectual progress.
But it has left many modern students
with the impression that Kolmogorov's axiomatization was born
full grown---a sudden brilliant triumph over confusion and chaos.

In order to see both the innovation and the synthesis in the \emph{Grundbegriffe},
we need a broad view of the foundations of probability and the advance of measure theory
from 1900 to 1930.
We need to understand how measure theory became more abstract during those decades,
and we need to recall what others were saying about axioms for probability, 
about Cournot's principle, and
about the relation of probability with measure and with frequency.
Our review of these topics draws mainly on work by 
authors listed by Kolmogorov in the \emph{Grundbegriffe}'s bibliography, 
especially Sergei Bernstein, \'Emile Borel, Francesco Cantelli, 
Maurice Fr\'echet, Paul L\'evy, Antoni {\L}omnicki, 
Evgeny Slutsky, Hugo Steinhaus, and Richard von Mises. 
Others enter into our story along the way, but for the most part 
we do not review
the contributions of authors whose foundational work does not seem
to have influenced Kolmogorov.  We say relatively little, for example, about
Harold Jeffreys, John Maynard Keynes, 
Jan {\L}ukasiewicz\nocite{lukasiewicz:1913}, 
Paolo Medolaghi\nocite{medolaghi:1907}, and Frank P. Ramsey\nocite{ramsey:1931}.
For further information about foundational and mathematical
work during the early twentieth century, see 
Urban (1923)\nocite{urban:1923},
Freudenthal and Steiner (1966)\nocite{freudenthal/steiner:1966,behnke/bertram/sauer:1966},
Regazzini (1987a,b)\nocite{regazzini:1987a,regazzini:1987b},
Krengel (1990)\nocite{krengel:1990},
von Plato (1994)\nocite{vonplato:1994},
Benzi (1995)\nocite{benzi:1995},
Holgate (1997)\nocite{holgate:1997}, Hochkirchen (1999)\nocite{hochkirchen:1999},
Bingham (2000)\nocite{bingham:2000}, 
and Bru (2003a)\nocite{bru:2003a}.

We are interested not only in Kolmogorov's mathematical
formalism, but also in his philosophy
of probability---how he proposed to relate the mathematical
formalism to the real world.
In a 1939 letter to Fr\'echet\nocite{kolmogorov:1939}, 
which we reproduce in \S\ref{app:kolfr},
Kolmogorov wrote, ``You are also 
right in attributing to me the opinion that the formal axiomatization
should be accompanied by an analysis of its real meaning.''
Kolmogorov devoted only two pages of the \emph{Grundbegriffe}
to such an analysis.  But the question was more important to him than 
this brevity might suggest.  We can study any mathematical
formalism we like, but we have the right to call it probability
only if we can explain how it relates to the empirical phenomena
classically treated by probability theory.

Kolmogorov's philosophy was frequentist.  
One way of understanding
his frequentism would be to place it in a larger social and cultural 
context, emphasizing perhaps Kolmogorov's role 
as the leading new Soviet mathematician.  We will not 
ignore this context, but we are more interested
in using the thinking of Kolmogorov and his predecessors 
to inform our own understanding of probability.  
In 1963\nocite{kolmogorov:1963},
Kolmogorov complained that his axioms had been so successful on the purely
mathematical side that
many mathematicians had lost interest in understanding how probability theory
can be applied.  This situation persists today.  Now, more than ever,
we need fresh thinking about how probability theory relates
to the world, and there is no better starting point for this thinking
than the works of Kolmogorov and his predecessors early in 
the twentieth century.

We begin by looking at the classical foundation
that Kolmogorov's measure-theoretic foundation replaced:
equally likely cases.
In \S\ref{sec:classical}, we review how probability was defined in terms of 
equally likely cases, how the rules of the calculus of probability were
derived from this definition, and
how this calculus was related to the real world by Cournot's principle.
We also look at some paradoxes discussed at the time.

In \S\ref{sec:measuretheory}, we sketch the 
development of measure theory and its increasing entanglement with
probability during the first three decades of the twentieth century.
This story centers on Borel,
who introduced countable additivity into pure mathematics
in the 1890s and then brought it to the center of probability theory,
as Fr\'echet noted, in 1909, when he first stated and more or less
proved the strong law of large numbers for coin tossing.
But it also features Lebesgue, Radon, Fr\'echet, Daniell, Wiener, Steinhaus,
and Kolmogorov himself.

Inspired partly by Borel and partly by the challenge issued by Hilbert
in 1900, a whole series of mathematicians proposed abstract frameworks
for probability during the three decades we are emphasizing.
In~\S\ref{sec:hilbert}, we look at some of these, beginning with the
doctoral dissertations by Rudolf Laemmel and Ugo Broggi in the first
decade of the century and including an early contribution
by Kolmogorov himself, written in 1927,
five years before he started work on the \emph{Grundbegriffe}\nocite{kolmogorov:1929}.

In~\S\ref{sec:grundbegriffe}, we finally turn
to the \emph{Grundbegriffe} itself.  Our review of it will confirm what
Fr\'echet said in 1937 and what Kolmogorov himself says in the preface:  
it was a synthesis and a manual, not a report on new research.  
Like any textbook, its mathematics was novel for most of its readers.
But its real originality was rhetorical and philosophical.

In \S\ref{sec:reception} we discuss how
the \emph{Grundbegriffe} was received---at the time and in the following decades.
Its mathematical framework for probability, as we know,
came to be regarded as fundamental,
but its philosophical grounding for probability was largely ignored.

\section{The classical foundation}
\label{sec:classical}

The classical foundation of probability theory, which began with the 
notion of equally likely cases,
held sway for two hundred years.  Its elements were 
put in place by Jacob Bernoulli\nocite{bernoulli:1713} and 
Abraham De Moivre\nocite{demoivre:1718} early in the eighteenth century, 
and they remained in place in the early twentieth
century.  Even today the classical foundation is used in teaching probability.

Although twentieth-century proponents of new approaches were fond of deriding
the classical foundation as naive or circular, it can be defended.
Its basic mathematics can be explained in a few words,
and it can be related to the real world by Cournot's principle,
the principle that an event with small or zero probability will not occur.
This principle was advocated in France and Russia
in the early years of the twentieth century but disputed in Germany.
Kolmogorov adopted it in the \emph{Grundbegriffe}.

In this section we review the mathematics of equally likely cases 
and recount the discussion of Cournot's principle, contrasting the support
for it in France with German efforts to replace it with other ways of relating
equally likely cases to the real world.  
We also discuss two paradoxes contrived at the end of the nineteenth century 
by Joseph Bertrand, which illustrate the care that must be taken 
with the concept of relative probability.  The lack of 
consensus on how to make philosophical sense of equally likely cases 
and the confusion engendered by Bertrand's paradoxes were two sources
of dissatisfaction with the classical theory.

\subsection{The classical calculus}
\label{subsec:equal}

The classical definition of probability
was formulated by Jacob Bernoulli in \emph{Ars Conjectandi} 
(1713)\nocite{bernoulli:1713}
and Abraham De Moivre in 
\emph{The Doctrine of Chances} (1718)\nocite{demoivre:1718}:
the probability of an event is the ratio of the 
number of equally likely cases that favor it
to the total number of equally likely cases possible under the circumstances.

From this definition, De Moivre derived two rules for probability.
The \emph{theorem of total probability}, or the \emph{addition theorem},
says that if $A$ and $B$ cannot both happen, then
\begin{Equation*}
  \mbox{probability of $A$ or $B$ happening}      \hspace*{5cm}
\BigSkip
   = \frac{\mbox{\# of cases favoring $A$ or $B$}}{\mbox{total \# of cases}}
   \hspace*{1.8cm}
\BiggSkip
     \hspace*{1.2cm}
   = \frac{\mbox{\# of cases favoring $A$}}{\mbox{total \# of cases}}
   + \frac{\mbox{\# of cases favoring $B$}}{\mbox{total \# of cases}}
\BiggSkip
   = (\mbox{probability of $A$}) + (\mbox{probability of $B$}).
\end{Equation*}%
The \emph{theorem of compound probability},
or the \emph{multiplication theorem}, says 
\begin{Equation*}
  \mbox{probability of both $A$ and $B$ happening}  \hspace*{4cm}
\BigSkip
   = \frac{\mbox{\# of cases favoring both $A$ and $B$}}{\mbox{total \# of cases}}
   \hspace*{3cm}
\BiggSkip
     \hspace*{1cm}
   = \frac{\mbox{\# of cases favoring $A$}}{\mbox{total \# of cases}}
   \times \frac{\mbox{\# of cases favoring both $A$ and $B$}}{\mbox{\# of cases favoring $A$}}
\BiggSkip
   = (\mbox{probability of $A$}) \times  
     (\mbox{probability of $B$ if $A$ happens}).
\end{Equation*}%
These arguments were still standard fare in probability textbooks
at the beginning of the twentieth century,
including the great treatises by Henri Poincar\'e (1896)\nocite{poincare:1896} in France,
Andrei Markov (1900)\nocite{markov:1900} in Russia,
and Emanuel Czuber (1903)\nocite{czuber:1903} in Germany.
Some years later we find them in Guido Castelnuovo's Italian textbook
(1919)\nocite{castelnuovo:1919},
which has been held out as the acme of the genre
(Onicescu 1967)\nocite{onicescu:1967}.

Only the British held themselves aloof from the classical theory,
which they attributed to Laplace\nocite{laplace:1812,laplace:1814} 
and found excessively apriorist.  
The British style of introducing probability, going back to
Augustus De Morgan (1838, 1847)\nocite{demorgan:1838,demorgan:1847},
emphasized combinatorics without dwelling on 
formalities such as the rules of total and compound 
probability, and the British statisticians preferred to
pass over the combinatorics as quickly as possible, so as to get
on with the study of ``errors of observations'' 
as in Airy (1861)\nocite{airy:1861}.  
According to his son Egon (1990\nocite{pearson:1990}, pp.~13, 71),
Karl Pearson recommended to his students 
nothing more theoretical than books by De Morgan,
William Whitworth (1878)\nocite{whitworth:1878}, and the Danish statistician
Harald Westergaard (1890, in German)\nocite{westergaard:1890}.

The classical theory, which had begun in England with De Moivre,
returned to the English language in an influential book
published in New York by the actuary Arne Fisher (1915)\nocite{fisher:1915},
who had immigrated to the United States from Denmark at the age of 16.
Fisher's book played an 
important role in bringing the methods of the great Scandinavian mathematical
statisticians, Jorgen Gram, Thorvald Thiele, Harald Westergaard, and Carl Charlier,
into the English-speaking world 
(Molina 1944\nocite{molina:1944}, Lauritzen 2002\nocite{lauritzen:2002}).  
The Scandinavians stood between the very empirical British statisticians 
on the one hand and the
French, German, and Russian probabilists on the other; they were serious
about advancing mathematical statistics beyond where Laplace had left it,
but they valued the classical foundation for probability.

After Fisher's book appeared, Americans adopted the classical rules but 
looked for ways to avoid the classical arguments based on equally likely cases; 
the two most notable American probability textbooks of the 1920s,
by Julian Lowell Coolidge (1925)\nocite{coolidge:1925} at Harvard and 
Thornton C. Fry (1928)\nocite{fry:1928}
at Bell Telephone Laboratories, replaced the classical arguments
for the rules of probability with arguments based on the assumption
that probabilities are limiting frequencies.
But this approach did not endure, and the only American probability textbook 
from before the second world war that remained in print 
in the second half of the century was the more classical one
by the Petersburg-educated Stanford professor
James~V. Uspensky (1937)\nocite{uspensky:1937}.

\subsubsection{Geometric probability}

Geometric probability was incorporated into the classical theory
in the early nineteenth century.
Instead of counting equally likely cases,
one measures their geometric extension---their area or volume.
But probability is still a ratio, and the rules of total and 
compound probability are still theorems. 
This is explained clearly by Antoine-Augustin Cournot in his 
influential treatise on probability and statistics, 
\emph{Exposition de la th\'eorie des
chances et des probabilit\'es}, published in 
1843\nocite{cournot:1843} (p.~29).  In his commentary in Volume XI of
Cournot's \emph{{\OE}uvres compl\`etes},
Bernard Bru traces the idea back to the mathematician Joseph Fourier and 
the naturalist George-Louis Leclerc de Buffon.

This understanding of geometric probability did not change
in the early twentieth century, when Borel and Lebesgue expanded
the class of sets for which we can define geometric extension.
We may now have more events with which to work, but we define and study
geometric probabilities in the same way as before.  Cournot would have seen nothing
novel in Felix Hausdorff's definition of probability in the chapter on 
measure theory in his 1914 treatise on set theory\nocite{hausdorff:1914}
(pp.~416--417).

\subsubsection{Relative probability}

The classical calculus was enriched at the beginning of the twentieth
century by a formal and universal notation for relative probabilities.
In 1901\nocite{hausdorff:1901}, Hausdorff introduced 
the symbol $p_F(E)$ for what he called the \emph{relative
Wahrscheinlichkeit von $E$, posito $F$} 
(relative probability of $E$ given $F$).   
Hausdorff explained that this notation can be used for any two events $E$
and $F$, no matter what their temporal or logical relationship,
and that it allows one to streamline Poincar\'e's proofs
of the addition and multiplication theorems.
At least two other authors, Charles Saunders Peirce (1867\nocite{peirce:1867},
1878\nocite{peirce:1878})
and Hugh MacColl (1880\nocite{maccoll:1880}, 1897\nocite{maccoll:1897}),
had previously proposed universal notations for the 
probability of one event given another.  But Hausdorff's
notation was adopted by the influential textbook
author Emanuel Czuber (1903\nocite{czuber:1903}).  Kolmogorov used it
in the~\emph{Grundbegriffe}, 
and it persisted, especially in the German literature,
until the middle of the twentieth century,
when it was displaced by the more flexible $P(E \given F)$,
which Harold Jeffreys had introduced
in his \emph{Scientific Inference} (1931\nocite{jeffreys:1931}).%
\footnote{See the historical discussion in Jeffreys's 
          \emph{Theory of Probability} (1939\nocite{jeffreys:1939}),
          on p.~25 of the first or third editions or p.~26 of the second edition.
          Among the early adopters of Jeffreys's vertical stroke
          were Jerzy Neyman, who used $\mathrm{P}\{\mathrm{A}\given\mathrm{B}\}$
          in 1937\nocite{neyman:1937}, 
          and Valery Glivenko, who used $\boldsymbol{P}(A/B)$
          (with only a slight tilt)
          in 1939\nocite{glivenko:1939}.}

Although Hausdorff's relative probability resembles
today's conditional probability, other classical authors used 
``relative'' in other ways.
For Sylvestre-Fran\c{c}ois Lacroix (1822, p.~20)\nocite{lacroix:1822}
and Jean-Baptiste-Joseph Liagre (1879, p.~45)\nocite{liagre:1879}, the 
probability of $E$ relative to an 
incompatible event $F$ was $P(E)/(P(E)+P(F))$.
For Borel (1914\nocite{borel:1914}, pp.~58--59), the relative probability of $E$ 
was $P(E)/P(\mathrm{not }E)$.  Classical authors could use the phrase however they 
liked, because it did not mark a sharp distinction like the modern distinction
between absolute and conditional probability.  Nowadays some authors write
the rule of compound probability in the form
$$
    P(A \amp B) = P(A) P(B \given A)
$$
and regard the conditional probability $P(B \given A)$ as fundamentally
different in kind from the absolute probabilities
$P(A \amp B)$ and $P(A)$.  But for the classical
authors, every probability
was evaluated in a particular situation, with respect to the equally
likely cases in that situation.  When these authors
wrote about the probability of $B$ ``after $A$ has
happened'' or ``when $A$ is known'', these
phrases merely identified
the situation; they did not indicate that a different kind of probability
was being considered.

Before the \emph{Grundbegriffe}, it was unusual 
to call a probability or expected value ``conditional'' rather than ``relative'',
but the term does appear.  George Boole may have been the first to
use it, though only casually.
In his \emph{Laws of Thought}\nocite{boole:1854}, in 1854,
Boole calls an event considered under a certain 
condition a conditional event, and he discusses the
probabilities of conditional events.  Once (p.~261),
and perhaps only once, he abbreviates this to ``conditional probabilities''.
In 1887, in his \emph{Metretike}\nocite{edgeworth:1887,edgeworth:1996,mirowski:1994},
Francis Edgeworth, citing Boole, systematically called
the probability of an effect given a cause a ``conditional probability'' 
(Mirowski 1994\nocite{mirowski:1994}, p.~82).  
The Petersburg statistician Aleksandr Aleksandrovich Chuprov,
who was familiar with Edgeworth's work, used the Russian equivalent
(условная вероятность)
in his 1910 book\nocite{chuprov:1910} (p.~151).
A few years later, in 1917, the German equivalent of ``conditional expectation''
(bedingte mathematische Erwartung) appeared in a book by Chuprov's friend
Ladislaus von Bortkiewicz\nocite{vonbortkiewicz:1917}, professor of statistics in Berlin.
We see ``conditional probability'' again in English in 1928,
in Fry's textbook\nocite{fry:1928} (p.~43).

We should also note that different authors used the term 
``compound probability'' (``probabilit\'e compos\'ee'' in French) in different ways.
Some authors (e.g., Poincar\'e 1912\nocite{poincare:1912}, p.~39)
seem to have reserved it for the case where the two events are independent;
others (e.g., Bertrand 1889\nocite{bertrand:1889}, p.~3)
used it in the general case as well.

\subsection{Cournot's principle}
\label{subsec:cournot}

An event with very small probability is \emph{morally impossible};
it will not happen.
Equivalently, an event with very high probability is \emph{morally certain};
it will happen.
This principle was first formulated within mathematical probability by Jacob Bernoulli.
In his \emph{Ars Conjectandi}\nocite{bernoulli:1713}, published in 1713,
Bernoulli proved a celebrated theorem:
in a sufficiently long sequence of independent trials of an event, 
there is a very high probability that the frequency with which the event happens
will be close to its probability.
Bernoulli explained that we can treat the very high probability as moral certainty
and so use the frequency of the event as an estimate of its probability.
This conclusion was later called the law of large numbers.

Probabilistic moral certainty was widely discussed in the eighteenth century.
In the 1760s, the French savant Jean d'Alembert muddled matters
by questioning whether the prototypical event of very small probability,
a long run of many happenings of an event as likely to fail as happen on each trial,
is possible at all.
A run of a hundred may be metaphysically possible, he felt, but it is physically impossible.
It has never happened and never will happen
(d'Alembert 1761, 1767\nocite{dalembert:1761,dalembert:1767}; Daston 1979\nocite{daston:1979}).
In 1777\nocite{buffon:1777}, Buffon argued that the distinction
between moral and physical certainty was one of degree.
An event with probability $9999/10000$ is morally certain;.
an event with much greater probability, such as 
the rising of the sun, is physically certain~(Loveland 2001)\nocite{loveland:2001}.

Cournot, a mathematician now remembered as an economist and 
a philosopher of science (Martin 1996, 1998\nocite{martin:1996,martin:1998}), 
gave the discussion a nineteenth-century cast in his 1843 treatise.
Being equipped with the idea of geometric probability, Cournot could talk about 
probabilities that are vanishingly small.
He brought physics to the foreground.  It may be mathematically
possible, he argued, for a heavy cone to stand in equilibrium
on its vertex, but it is physically impossible.
The event's probability is vanishingly small.  Similarly, it 
is physically impossible for the frequency of an event in a 
long sequence of trials to differ substantially
from the event's probability (1843\nocite{cournot:1843}, pp.~57, 106).

In the second half of the nineteenth century, the principle that an 
event with a vanishingly small probability will not happen
took on a real role in physics, most saliently in Ludwig Boltzmann's statistical
understanding of the second law of thermodynamics.  As Boltzmann explained 
in the 1870s, dissipative processes
are irreversible because the probability of a state with entropy far from 
the maximum is vanishingly small 
(von Plato 1994\nocite{vonplato:1994}, p.~80; 
Seneta 1997\nocite{johnson/kotz:1997,seneta:1997}).
Also notable was Henri Poincar\'e's use of the principle
in the three-body problem\nocite{poincare:1890,vonplato:1994}.
Poincar\'e's recurrence theorem,
published in 1890\nocite{poincare:1890},
says that an isolated mechanical system
confined to a bounded region of its phase space
will eventually return arbitrarily close to its initial state,
provided only that this initial state is not exceptional.
Within any region of finite volume,
the states for which the recurrence does not hold are exceptional
inasmuch as they are contained in subregions whose total volume is arbitrarily small.

Saying that an event of very small or vanishingly small probability 
will not happen is one thing.  Saying that probability
theory gains empirical meaning only by ruling out the happening of such events
is another.  Cournot seems to have been the first to say explicitly that 
probability theory does gain empirical meaning only by declaring
events of vanishingly small probability to be impossible:  
    \begin{quotation}
       \noindent
       \dots \textit{The physically impossible event is therefore the one that
       has infinitely small probability}, and only this remark gives 
       substance---objective and phenomenal value---to the theory 
       of mathematical probability (1843\nocite{cournot:1843} p.~78).%
\footnote{The phrase ``objective and phenomenal'' refers to
          Kant's distinction between the noumenon, or thing-in-itself, 
          and the phenomenon, or object of experience
          (Daston 1994\nocite{daston:1994}).} 
    \end{quotation}
After the second world war (see \S\ref{subsubsec:cournotafter}), some authors began to use
``Cournot's principle'' for the principle that an event of 
very small or zero probability singled out in advance will not happen,
especially when this principle is advanced as the means by which a 
probability model is given empirical meaning.

\subsubsection{The viewpoint of the French probabilists}\label{subsubsec:french}

In the early decades of the twentieth century, probability theory
was beginning to be understood as pure mathematics.  What does 
this pure mathematics have to do with the real world?  The 
mathematicians who revived research in probability theory in France
during these decades, \'Emile Borel, Jacques Hadamard, Maurice Fr\'echet, and Paul L\'evy, 
made the connection by
treating events of small or zero probability as impossible.

Borel explained this repeatedly,
often in a style more literary than mathematical or philosophical
(Borel 1906, 1909b, 1914, 1930\nocite{borel:1906,borel:1909b,borel:1914,borel:1930}).
According to Borel, a result of the probability calculus deserves to be called objective
when its probability becomes so great as to be practically the same as certainty.
His many discussions of the considerations that go into assessing
the boundaries of practical certainty culminated in a classification
more refined than Buffon's.
A probability of $10^{-6}$, he decided, is negligible at the human scale,
a probability of $10^{-15}$ at the terrestrial scale,
and a probability of $10^{-50}$ at the cosmic scale
(Borel 1939, pp.~6--7)\nocite{borel:1939}.  

Hadamard, the preeminent analyst who did pathbreaking
work on Markov chains in the 1920s~(Bru 2003a)\nocite{bru:2003a},
made the point in a different way.  
Probability theory, he said, is based on two basic notions:  
the notion of perfectly equivalent (equally likely) events and the 
notion of a very unlikely event (Hadamard 1922, p.~289)\nocite{hadamard:1922}.
Perfect equivalence is a mathematical assumption, which cannot be verified.
In practice, equivalence
is not perfect---one of the grains in a cup of sand may be more likely than 
another to hit the ground first when they are thrown out of the cup.
But this need not prevent us from applying the principle of the very unlikely event.
Even if the grains are not exactly the same,
the probability of any particular one hitting the ground first is negligibly small.
Hadamard cited Poincar\'e's work on the three-body problem in this connection,
because Poincar\'e's conclusion is insensitive
to how one defines the probabilities for the initial state.
Hadamard was the teacher of both Fr\'echet and L\'evy.

It was L\'evy, perhaps, who had the strongest sense of probability's being pure mathematics
(he devoted most of his career as a mathematician to probability),
and it was he who expressed most clearly in the 1920s
the thesis that Cournot's principle is probability's only bridge to reality.
In his \emph{Calcul des probabilit\'es}\nocite{levy:1925}
L\'evy emphasized the different roles of Hadamard's two basic notions.
The notion of equally likely events, L\'evy explained, suffices as a foundation
for the mathematics of probability, but so long as we base our reasoning 
only on this notion, our probabilities are merely subjective.  
It is the notion of a very unlikely event
that permits the results of the mathematical theory to take on practical
significance (\citealt{levy:1925}, pp.~21, 34; see also \citealt{levy:1937}, p.~3).
Combining the notion of a very unlikely event
with Bernoulli's theorem, we obtain the notion of the objective probability of an event,
a physical constant that is measured by relative frequency.
Objective probability, in L\'evy's view,
is entirely analogous to length and weight,
other physical constants whose empirical meaning is also defined
by methods established for measuring them to a reasonable approximation
(\citealt{levy:1925}, pp.~29--30).

By the time he undertook to write the \emph{Grundbegriffe},
Kolmogorov must have been very familiar with L\'evy's views.  
He had cited L\'evy's 1925 book in his 1931 article on Markov processes
and subsequently, during his visit to France,
had spent a great deal of time talking with L\'evy about probability.
But he would also have learned about Cournot's principle from the Russian literature.
The champion of the principle in Russia had been Chuprov,
who became professor of statistics in Petersburg in 1910.
Like the Scandinavians, Chuprov wanted to bridge the gap between
the British statisticians and the continental mathematicians
(Sheynin 1996\nocite{sheynin:1996},
Seneta 2001\nocite{seneta:2001}).  
He put Cournot's principle---which he called ``Cournot's lemma''---at the heart
of this project; it was, he said, a basic principle of the logic of 
the probable (Chuprov 1910\nocite{chuprov:1910}, Sheynin 1996\nocite{sheynin:1996}, pp.~95--96).  
Markov, Chuprov's neighbor in Petersburg, learned about the burgeoning field
of mathematical statistics from Chuprov (Ondar 1981\nocite{ondar:1981}),
and we see an echo of Cournot's principle in Markov's textbook
(1912\nocite{markov:1912}, p.~12 of the German edition):
\begin{quotation}
   The closer the probability of an event is to one,
   the more reason we have to expect the event to happen
   and not to expect its opposite to happen.

   In practical questions, we are forced to regard as certain events
   whose probability comes more or less close to one,
   and to regard as impossible events whose probability is small.

   Consequently, one of the most important tasks of probability theory
   is to identify those events whose probabilities come close to one
   or zero.
\end{quotation}
The Russian statistician Evgeny Slutsky discussed Chuprov's views
in his influential article on limit theorems, published in 
German in 1925\nocite{slutsky:1925}.
Kolmogorov included L\'evy's book and Slutsky's article in 
his bibliography, but not Chuprov's book.
An opponent of the Bolsheviks,
Chuprov was abroad when they seized power, and he never returned home.
He remained active in Sweden and Germany,
but his health soon failed, and he died in 1926, at the age of 52.

\subsubsection{Strong and weak forms of Cournot's principle}\label{subsubsec:twoforms}

Cournot's principle has many variations.
Like probability, moral certainty can be subjective or objective.  Some authors
make moral certainty sound truly equivalent to absolute certainty; others emphasize its
pragmatic meaning.

For our story, it is important to distinguish between the strong and weak
forms of the principle (Fr\'echet 1951, p.~6\nocite{frechet:1951};
Martin 2003\nocite{martin:2003}). 
The strong form refers to an event of small or zero probability that we 
single out in advance of a single trial:  it says the 
event will not happen on that trial.  
The weak form says that an event with very small probability will
happen very rarely in repeated trials.  

Borel, L\'evy, and Kolmogorov all enunciated Cournot's principle in its strong 
form.  In this form, the principle combines with Bernoulli's theorem to produce
the unequivocal conclusion that an event's probability will be approximated
by its frequency in a particular sufficiently long sequence of independent trials.
It also provides a direct foundation for statistical testing.  If the empirical meaning of
probability resides precisely in the non-happening of small-probability events 
singled out in advance,
then we need no additional principles to justify rejecting
a hypothesis that gives small probability to an event we single out in advance
and then observe to happen (Bru 1999\nocite{bru:1999}).

Other authors, including Chuprov, enunciated Cournot's principle in its weak form,
and this can lead in a different direction.  The weak principle
combines with Bernoulli's theorem to produce
the conclusion that an event's probability will \emph{usually} be approximated
by its frequency in a sufficiently long sequence of independent trials,
a general principle that has the weak principle as a special case\label{p:specialcase}.
This was pointed out by Castelnuovo in his 1919 textbook
(p.~108)\nocite{castelnuovo:1919}.
Castelnuovo called the general principle the \emph{empirical law of chance}
(la legge empirica del caso):
\begin{quotation}
   In a series of trials repeated a large number of times under identical conditions,
   each of the possible events happens with a (relative) frequency
   that gradually equals its probability.
   The approximation usually improves
   with the number of trials. (Castelnuovo 1919, p.~3)\nocite{castelnuovo:1919}
\end{quotation}
Although the special case where the probability is close to one is sufficient to 
imply the general principle, Castelnuovo preferred to begin his introduction to 
the meaning of probability by enunciating the general principle, and so he can be
considered a frequentist.  His approach was influential at the time.
Maurice Fr\'echet and Maurice Halbwachs adopted it in their textbook 
in 1924\nocite{frechet/halbwachs:1924}.
It brought Fr\'echet to the same understanding of objective probability as L\'evy:
it is a physical constant that is measured by relative frequency
(1938a\nocite{frechet:1938a}, p.~5;
1938b\nocite{frechet:1938b}, pp.~45--46).

The weak point of Castelnuovo and Fr\'echet's position lies
in the modesty of their conclusion:
they conclude only that an event's probability is \emph{usually}
approximated by its frequency.
When we estimate a probability from an observed frequency,
we are taking a further step:
we are assuming that what usually happens has happened in the particular case.
This step requires the strong form of Cournot's principle.
According to Kolmogorov (1956\nocite{kolmogorov:1956},
p.~240 of the 1965 English edition)\label{p:weak},
it is a reasonable step only if ``we have some reason for assuming''
that the position of the particular case among other potential ones
``is a regular one, that is, that it has no special features''.

\subsubsection{British indifference and German skepticism}\label{subsubsec:britger}

The mathematicians who worked on probability in France in the early
twentieth century were unusual in the extent to which they
delved into the philosophical side of their subject.
Poincar\'e had made a mark in the philosophy of science as well as in mathematics,
and Borel, Fr\'echet, and L\'evy tried to emulate him.
The situation in Britain and Germany was different.

In Britain there was little mathematical work in probability proper in this period.
In the nineteenth century, British interest in probability
had been practical and philosophical, not mathematical
(Porter 1986, p.~74ff)\nocite{porter:1986}.
British empiricists such as Robert Leslie Ellis (1849\nocite{ellis:1849}) 
and John Venn (1888\nocite{venn:1888}) accepted the usefulness of probability but
insisted on defining it directly in terms of frequency,
leaving little meaning or role for the law of large numbers
and Cournot's principle (Daston 1994\nocite{daston:1994}).   
These attitudes, 
as we noted in \S\ref{subsec:equal},
persisted even after Pearson and Fisher had brought Britain into a leadership role
in mathematical statistics.  The British statisticians had little interest in
mathematical probability theory and hence no puzzle to solve 
concerning how to link it to the real world.
They were interested in reasoning directly about frequencies.

In contrast with Britain, Germany did see a substantial amount of mathematical
work in probability
during the first decades of the twentieth century, much of it published 
in German by Scandinavians and eastern Europeans.
But few German mathematicians of the first rank fancied themselves philosophers.
The Germans were already pioneering the division of labor
to which we are now accustomed, between mathematicians who prove theorems about
probability and philosophers, logicians, statisticians, and
scientists who analyze the meaning of probability.
Many German statisticians believed that one must decide
what level of probability will count as practical certainty in order to apply
probability theory (von Bortkiewicz 1901\nocite{vonbortkiewicz:1901},
p.~825; Bohlmann 1901\nocite{bohlmann:1901}, p.~861),
but German philosophers did not give 
Cournot's principle a central role.

The most cogent and influential of the German philosophers who discussed 
probability in the late nineteenth century was Johannes von Kries, whose 
\emph{Principien der Wahrscheinlichkeitsrechnung} first appeared in 1886.
Von Kries rejected what he called the orthodox philosophy
of Laplace and the mathematicians who followed him.  As von Kries's saw it,
these mathematicians began with a subjective concept of probability
but then claimed to establish the existence
of objective probabilities
by means of a so-called law of large numbers,
which they erroneously derived by combining
Bernoulli's theorem with the belief
that small probabilities can be neglected.
Having both subjective and objective probabilities at their disposal,
these mathematicians then used Bayes's theorem
to reason about objective probabilities for almost any question
where many observations are available.
All this, von Kries believed, was nonsense.
The notion that an event with very small probability is impossible was,
in von Kries's eyes, simply d'Alembert's mistake.

Von Kries believed that objective probabilities sometimes exist, but only under
conditions where equally likely cases can legitimately be identified.
Two conditions, he thought, are needed:  
\begin{itemize} 
   \item
      Each case is produced by equally many of the possible arrangements of 
      the circumstances, and this remains true when we look back in time to earlier 
      circumstances that led to the current ones.
      In this sense, the relative sizes of the cases are \emph{natural}.
   \item
      Nothing besides these circumstances affects our expectation about the
      cases.  In this sense, the Spielr\"aume%
\footnote{In German, Spiel means ``game'' or ``play'', and Raum (plural R\"aume) means
    ``room'' or ``space''.  In most contexts, Spielraum 
    can be translated as ``leeway'' or ``room for maneuver''.
    For von Kries, the Spielraum for each case was the set of all arrangements
    of the circumstances that produce it.}
      are \emph{insensitive}.
\end{itemize}
Von Kries's \emph{principle of the Spielr\"aume}
was that objective probabilities can be calculated from equally likely cases when these
conditions are satisfied.  He considered this principle analogous to
Kant's principle that everything that exists has a cause.  Kant thought that 
we cannot reason at all without the principle of cause and effect.  Von Kries thought that we 
cannot reason about objective probabilities without the principle of the Spielr\"aume.

Even when an event has an objective probability, von Kries saw no legitimacy in the law
of large numbers.  Bernoulli's theorem is valid, he thought, but it 
tells us only that a large deviation of an event's frequency from its probability is just
as unlikely as some other unlikely event, say a long run of successes.
What will actually happen is another matter.
This disagreement between Cournot and von Kries 
can be seen as a quibble about words.  Do we say that an event
will not happen (Cournot), or do we say merely that it is as unlikely
as some other event we do not expect to happen (von Kries)?  Either way, we proceed
as if it will not happen.  But the quibbling has its reasons.  Cournot
wanted to make a definite prediction, because this provides
a bridge from probability theory to the world of phenomena---the real world,
as those who have not studied Kant would say.  
Von Kries thought he had a different
way of connecting probability theory with phenomena.

Von Kries's critique of moral certainty and the law of large numbers
was widely accepted in Germany
(Kamlah, 1983\nocite{kamlah:1983}).
Czuber, in the influential textbook we have already mentioned, named 
Bernoulli, d'Alembert, Buffon, and De Morgan as advocates of 
moral certainty and declared them all wrong;
the concept of moral certainty, he said, violates
the fundamental insight that an event of ever so small a probability can still happen
(Czuber 1903\nocite{cournot:1843}, p.~15; see also Meinong 1915\nocite{meinong:1915}, p.~591).

This wariness about ruling out the happening of events whose
probability is merely very small does not seem to have prevented 
acceptance of the idea that zero probability represents impossibility.
Beginning with Wiman's work on continued fractions in 1900, mathematicians writing
in German had worked on showing that 
various sets have measure zero, and everyone understood that the point was to show 
that these sets are impossible
(see Felix Bernstein 1912, p.~419\nocite{bernstein:1912}).
This suggests a great gulf between zero probability and merely small probability.
One does not sense such a gulf in the writings of Borel and his French colleagues;
as we have seen, the vanishingly small, for them, was merely
an idealization of the very small.

Von Kries's principle of the Spielr\"aume did not endure,
for no one knew how to use it.  But his project of providing 
a Kantian justification for the uniform distribution of probabilities 
remained alive in German philosophy in the first decades of the twentieth century
(Meinong 1915\nocite{meinong:1915}; Reichenbach 1916\nocite{reichenbach:1916}).  
John Maynard Keynes (1921)\nocite{keynes:1921} brought it into the English literature,
where it continues to echo, 
to the extent that today's probabilists, when asked about the philosophical grounding of
the classical theory of probability, are more likely to think about arguments for
a uniform distribution of probabilities than about Cournot's principle.

\subsection{Bertrand's paradoxes}
\label{subsec:bertrand}

How do we know cases are 
equally likely, and when something happens, do the cases that remain
possible remain equally likely?  In the decades before the \emph{Grundbegriffe},
these questions were frequently discussed in the context of paradoxes 
formulated by Joseph Bertrand, an influential French mathematician, in 
a textbook that he published in 1889\nocite{bertrand:1889} after teaching
probability for many decades (Bru and Jongmans 2001\nocite{bru/jongmans:2001}).

We now look at discussions by other authors of two of Bertrand's paradoxes:
Poincar\'e's discussion of the paradox of the three jewelry boxes,
and Borel's discussion of the paradox of the great circle.%
\footnote{In the literature of the period, ``Bertrand's paradox'' usually 
   referred to a third paradox, concerning two possible interpretations of the 
   idea of choosing a random chord on a circle.  Determining a chord by choosing 
   two random points on the circumference is not the same as determining it
   by choosing a random distance from the center and then a random orientation.}
The latter 
was also discussed by Kolmogorov and is now sometimes called the ``Borel-Kolmogorov 
paradox''.

\subsubsection{The paradox of the three jewelry boxes}

This paradox, laid out by Bertrand on pp.~2--3 of his textbook, 
involves three identical jewelry boxes, each with two drawers.  
Box A has gold medals in both drawers, Box B has silver medals in both, and 
Box C has a gold medal in one and a silver medal in the other.
Suppose we choose a box at random.  It will be Box C with probability $1/3$.
Now suppose we open at random one of the drawers in the box we have chosen.  
There are two possibilities for what we find:
\begin{itemize}
   \item
      We find a gold medal.  In this case, 
      only two possibilities remain:  the other drawer has a gold medal (we have chosen
      Box A), or the other drawer has a silver medal (we have chosen Box C).
   \item
      We find a silver medal.  Here also, 
      only two possibilities remain:  the other drawer has a gold medal (we have chosen
      Box C), or the other drawer has a silver medal (we have chosen Box B).
\end{itemize}
Either way, it seems, there are now two cases, one of which is that we have chosen
Box C.  So the probability that we have chosen Box C is now $1/2$.

Bertrand himself did not accept the conclusion that opening the drawer would change 
the probability of having Box C from $1/3$ to $1/2$, and Poincar\'e 
gave an explanation (1912\nocite{poincare:1912}, pp.~26--27).  Suppose the drawers in each 
box are labeled (where we cannot see) $\alpha$ and $\beta$, and suppose the gold medal
in Box C is in drawer $\alpha$.  Then there are six equally likely cases for the 
drawer we open:
\begin{enumerate}
  \item
     Box A, Drawer $\alpha$:  gold medal.
  \item
     Box A, Drawer $\beta$:  gold medal.
  \item
     Box B, Drawer $\alpha$:  silver medal.
  \item
     Box B, Drawer $\beta$:  silver medal.
  \item
     Box C, Drawer $\alpha$:  gold medal.
  \item
     Box C, Drawer $\beta$:  silver medal.
\end{enumerate}		
When we find a gold medal, say, in the drawer we have opened, three of these
cases remain possible:  case 1, case 2, and case 5.  Of the three, only one 
favors our having our hands on Box C.  So the probability for Box C is still $1/3$.

\subsubsection{The paradox of the great circle}

This paradox, on pp.~6--7 of Bertrand's textbook, 
begins with a simple question:  if we 
choose at random two points on the surface of a sphere,
what is the probability
that the distance between them is less than $10^{\prime}$?

By symmetry, we can suppose that the first point is known.  So one way of 
answering the question is to calculate the proportion of a sphere's
surface that lies within $10^{\prime}$ of a given point.  This is
$2.1 \times 10^{-6}$.%
\footnote{The formula Bertrand gives is correct, and it evaluates to this
          number.  Unfortunately, he then gives a numerical value that is twice as 
          large, as if the denominator of the ratio being calculated were the 
          area of a hemisphere rather than the area of the entire sphere.
          (Later in the book, on p.~169, he considers a version of 
          the problem where the point is drawn at random from a 
          hemisphere rather than from a sphere.)  Bertrand composed his book
          by drawing together notes from decades of teaching, and the
          carelessness with which he did this may have enhanced the sense
          of confusion that his paradoxes engendered.}
Bertrand also found a different answer using a different method.
After fixing the first point, he said, we can also assume that we 
know the great circle that connects the two points, because 
the possible chances are the same on great circles through the first point.
There are $360$ degrees---$2160$ arcs of $10^{\prime}$ each---in this great circle.  
Only the points in the two neighboring arcs are 
within $10^{\prime}$ of the first point, and so the probability sought is 
$2/2160$, or $9.3 \times 10^{-4}$.
This is many times larger than the first answer.
Bertrand suggested that both answers were equally valid, the original question being ill posed.
The concept of choosing points at random on a sphere was not, he said, sufficiently precise.

\begin{figure}[t]
   \begin{center}
      \includegraphics[width=3cm]{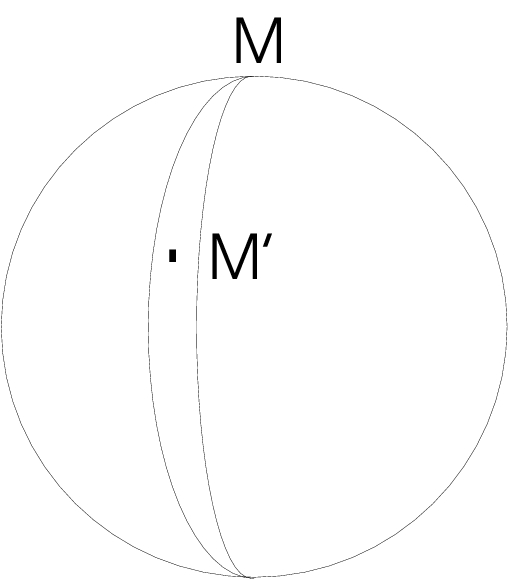}
      Borel's Figure 13.
    \end{center}
\end{figure}

In his own probability textbook, published in 1909\nocite{borel:1909b} (pp.~100--104),
Borel explained that Bertrand was mistaken.
Bertrand's first answer, obtained by assuming that equal areas on
the sphere have equal chances of containing the second point, is correct.
His second answer, obtained by assuming that equal arcs on a great 
circle have equal chances of containing it, is incorrect.  Writing M and M$^{\prime}$ for the
two points to be chosen at random on the sphere, Borel explained Bertrand's
mistake as follows:
\begin{quotation}
\noindent
\dots The error begins
when, after fixing the point M and the great circle, one assumes that the
probability of M$^{\prime}$ being on a given arc of the great circle is proportional
to the length of that arc.  If the arcs have no width, then in order to 
speak rigorously, we must assign the value zero to the probability that
M and M$^{\prime}$ are on the circle.  In order to avoid this factor of zero,
which makes any calculation impossible, one must consider a thin bundle of 
great circles all going through M, and then it is obvious that there is 
a greater probability for M$^{\prime}$ to be situated in a vicinity $90$ degrees
from M than in the vicinity of M itself (\emph{fig.}~13). 
\end{quotation}
To give this argument practical content, Borel discussed how one might  
measure the longitude of a point on the surface of the earth.
If we use astronomical observations, then we are measuring an 
angle, and errors in the measurement of the angle correspond to 
wider distances on the ground at the equator than at the poles.  
If we instead use geodesic measurements, say with a line
of markers on each of many meridians, then in order 
to keep the markers out of each other's way, we must 
make them thinner and thinner as we approach the poles.

\subsubsection{Appraisal}

Poincar\'e, Borel, and others who understood the principles of the classical theory
were able to resolve the paradoxes that Bertrand contrived.
Two principles emerge from the resolutions they offered:
\begin{itemize}
  \item
    The equally likely cases must be detailed enough
    to represent new information (e.g., we find
    a gold medal) in all relevant detail.  The remaining equally likely cases
    will then remain equally likely.
  \item  We may need to consider the real observed
    event of non-zero probability that is represented in an idealized way
    by an event of zero probability (e.g., a randomly chosen point
    falls on a particular meridian).
    We should pass to the limit only after absorbing the new information.
\end{itemize}
Not everyone found it easy to apply these principles, however, and the 
confusion surrounding the paradoxes was another source of dissatisfaction
with the classical theory.

Modern theories have tried to solve the problem by representing
explicitly the possibilities
for new information.  This Kolmogorov did this using a partition
(see p.~\pageref{p:decomposition} below).  Other authors have used
filtrations (Doob 1953\nocite{doob:1953}), event trees 
(Shafer 1996\nocite{shafer:1996}) and game protocols
(Shafer and Vovk 2001\nocite{shafer/vovk:2001}).  These devices
may be helpful, but they have not put the paradoxers
out of business.  Puzzles like Bertrand's paradox of the three jewelry 
boxes still flourish
(Bar-Hillel and Falk 1982\nocite{barhillel/falk:1982},
Shafer 1985\nocite{shafer:1985}, 
Barbeau 1993\nocite{barbeau:1993},
Halpern and Tuttle 1993\nocite{halpern/tuttle:1993}).

\section{Measure-theoretic probability before the \emph{Grundbegriffe}}
\label{sec:measuretheory}

A discussion of the relation between measure and probability 
in the first decades of the twentieth century must navigate many pitfalls, 
for measure theory itself evolved, beginning
as a theory about the measurability of sets of real numbers and then becoming
more general and abstract.  
Probability theory followed along, 
but since the meaning of \emph{measure} was changing, we can
easily misunderstand things said at the time about 
the relation between the two theories.

The development of theories of measure and integration during 
the late nineteenth and early twentieth centuries has been 
studied extensively (Hawkins 1975\nocite{hawkins:1975},
Pier 1994a\nocite{pier:1994a}).  Here we offer only 
a bare-bones sketch, beginning with Borel and Lebesgue (\S\ref{subsec:borleb})
and touching on
those steps that proved most significant for the foundations of 
probability.  We discuss the work of Carath\'eodory, Radon, Fr\'echet,
and Nikodym,
who made measure primary and the integral secondary (\S\ref{subsec:radsak}),
as well as the contrasting approach of Daniell, 
who took integration to be basic (\S\ref{subsec:danwien}).  

We dwell more on the relation of measure and integration to probability.
We discuss perceptions of the relevance of Fr\'echet's work to probability
(\S\ref{subsec:frechet}) before turning to Wiener's
theory of Brownian motion.  Then we discuss 
Borel's strong law of large numbers, which 
focused attention on measure rather than on integration
(\S\ref{subsec:entangle}).  After looking at
Steinhaus's axiomatization of Borel's denumerable probability
and its relation to the Polish
work on independent functions, we turn to Kolmogorov's use of measure theory
in probability in the 1920s.  Kolmogorov's work in probability began in
collaboration with Khinchin, who was using Steinhaus's picture to develop
limit theorems, but he quickly dropped Steinhaus's picture in 
favor of Fr\'echet's integral, which he brought to prominence in probability
theory with his 1931 article on Markov processes (\S\ref{subsec:kolenter}).

\subsection{The invention of measure theory by Borel and Lebesgue}
\label{subsec:borleb}

\'Emile Borel is usually considered the founder of measure theory.  
Whereas Peano and Jordan had extended the 
concept of length from intervals to a larger class of sets 
of real numbers by approximating the sets inside and out with finite unions of
intervals, Borel used countable unions.
His motivation came from complex analysis.
In his doctoral dissertation in 1894 (published in 1895\nocite{borel:1895}),
Borel studied certain series that were known to diverge
on a dense set of points on a closed curve
and hence, it was thought, could not be continued analytically
into the region bounded by the curve.
Roughly speaking, Borel discovered that the set of points
where divergence occurred, although dense, can be covered by a countable number of 
intervals with arbitrarily small total length.
Elsewhere on the curve---almost everywhere, 
we would say now---the series does converge, and so
analytic continuation is possible (Hawkins 1975\nocite{hawkins:1975}, \S4.2).
This discovery led Borel to a new theory of measurability
for subsets of $[0,1]$, which he published in 1898\nocite{borel:1898}.  

Borel's innovation was quickly seized upon by Henri Lebesgue,
who made it the basis for the powerful theory of 
integration that he first announced in 1901\nocite{lebesgue:1901}. 
We now speak of Lebesgue measure on the real numbers $R$ and on 
the $n$-dimensional space $R^n$, and of the Lebesgue integral 
in these spaces.

We need not review Lebesgue's theory, but we should 
mention one theorem, the precursor of the Radon-Nikodym 
theorem:  any countably additive and absolutely
continuous set function on the real numbers is an indefinite integral.  
This result first appeared in Lebesgue's 1904 book\nocite{lebesgue:1904} 
(Hawkins 1975\nocite{hawkins:1975}, p.~145; Pier 1994\nocite{pier:1994}, p.~524).
He generalized it to $R^n$ in 1910
(Hawkins 1975\nocite{hawkins:1975}, p.~186).

We should also mention a note published in 1918 by 
Wac{\l}aw Sierpi\'nski on the axiomatic treatment of Lebesgue measure.
In this note\nocite{sierpinski:1918},
important to us because of the use Hugo Steinhaus later made of it,
Sierpi\'nski characterized the class of Lebesgue measurable sets
as the smallest class $K$ of sets satisfying the following conditions:
\begin{enumerate}\renewcommand{\labelenumi}{\Roman{enumi}}
  \item
     For every set $E$ in $K$, there is a nonnegative 
     number $\mu(E)$ that will be its measure and will 
     satisfy conditions II, III, IV, and V.
  \item
     Every finite closed interval is in $K$ and has its length 
     as its measure.
  \item
     $K$ is closed under finite and countable unions of disjoint 
     elements, and $\mu$ is finitely and countably additive.
  \item
     If $E_1 \supset E_2$ and $E_1$ and $E_2$ are in $K$, then 
     $E_1 \setminus E_2$ is in $K$.
  \item
     If $E$ is in $K$ and $\mu(E) = 0$, then any subset of $E$ is in $K$.
\end{enumerate}
An arbitrary class $K$ satisfying these five conditions 
is not necessarily a field; there is no requirement that an intersection
of two of $K$'s elements also be in $K$.\label{p:field}%
\footnote{Recall that a \emph{field} of sets is a collection of sets 
          closed under relative complementation and finite union and 
          intersection.  A field of sets closed under denumerable
          union and intersection is a \emph{Borel} field.  A field that
          has a largest element is called an \emph{algebra}, and a Borel
          field that has a largest element is called 
          a \emph{$\sigma$-algebra}.
          Although \emph{algebra} and \emph{$\sigma$-algebra} are now
          predominant in probability theory, \emph{field} and \emph{Borel field}
          were more common in mathematical work before the second world war.}

\subsection{Abstract measure theory  from Radon to Saks}\label{subsec:radsak}

Abstract measure and integration began in the second decade of
the twentieth century and came into full flower in the fourth.

The first and most important step was taken by Johann Radon, in a 
celebrated article published in 1913\nocite{radon:1913}.
Radon unified
Lebesgue and Stieltjes integration by generalizing  
integration with respect to Lebesgue measure
to integration with respect to any countably additive 
set function (\emph{absolut additive Mengenfunktion})
on the Borel sets in $\bbbr^n$.  
The generalization included a version of the theorem
of Lebesgue's we just mentioned:  if a countably additive
set function $g$ on $R^n$ is absolutely continuous with respect 
to another countably additive set function $f$, then $g$ is
an indefinite integral with respect to $f$
(Hawkins 1975\nocite{hawkins:1975}, p.~189).

Constantin Carath\'eodory was also influential in drawing attention
to measures on Euclidean spaces other than Lebesgue measure.
In 1914\nocite{caratheodory:1914},
Carath\'eodory gave axioms for outer measure in a $q$-dimensional space,
derived the notion of measure,
and applied these ideas not only to Lebesgue measure
on Euclidean spaces but also to lower-dimensional measures on Euclidean space,
which assign lengths to curves, areas to surfaces, etc.\ \cite{hochkirchen:1999}.
Carath\'eodory also recast Lebesgue's theory of integration
to make measure even more fundamental;
in his 1918 textbook\nocite{caratheodory:1918} on real functions,
he defined the integral of a positive function on a subset of $\bbbr^n$
as the $(n+1)$-dimensional measure of the region
between the subset and the function's graph
(Bourbaki, 1994, p.~228\nocite{bourbaki:1994}).

It was Fr\'echet who first went beyond Euclidean space.  
In 1915\nocite{frechet:1915a,frechet:1915b}, Fr\'echet observed that much 
of Radon's reasoning does not depend on the assumption that
one is working in $\bbbr^n$.  One can reason in the same way in a much
larger space, such as a space of functions.
Any space will do, so long as the countably additive set function is defined
on a $\sigma$-field of its subsets, as Radon had required.
One thus arrives at the abstract theory
of integration on which Kolmogorov based probability.
As Kolmogorov put it in the preface to his 
\emph{Grundbegriffe},
\begin{quotation}
\noindent
  \dots After Lebesgue's investigations,
  the analogy between the measure of a set and the probability of 
  an event, as well as between the integral of a function and the mathematical
  expectation of a random variable, was clear.
  This analogy could be extended further; for example,
  many properties of independent random variables are
  completely analogous to corresponding properties of orthogonal functions.
  But in order to base probability theory on this analogy,
  one still needed to liberate the theory of measure and integration
  from the geometric elements still in the foreground with Lebesgue.
  This liberation was accomplished by Fr\'echet.
\end{quotation}
Fr\'echet did not, however, manage to generalize
Radon's theorem on absolute continuity to the fully
abstract framework.  This generalization, now called
the \emph{Radon-Nikodym theorem}, was obtained by 
Otton Nikodym in 1930\nocite{nikodym:1930}.  

It should not be inferred from Kolmogorov's words
that Fr\'echet used ``measure'' in the way we do today.  
In his 1915 articles and in his treatise on 
probability\nocite{frechet:1937/1938},
cited in the \emph{Grundbegriffe}
but not published until 1937--1938,
Fr\'echet used \emph{fonction additive d'ensembles} for what we now call a measure.
He makes this comment on p.~6 of the treatise:
\begin{quotation}
  We should note a tendency, in recent years,
  to model probability theory on measure theory.
  In fact, it is not the notion of the measure of a set,
  but rather the notion of an \emph{additive set function} that is appropriate,
  because of the theorem
  (or postulate) of total probability, for representing a probability,
  either continuous or discrete.
\end{quotation}
Kolmogorov was similarly old-fashioned in the text of the \emph{Grundbegriffe}, using
\emph{vollst\"andig additive Mengenfunktion}.  
Fr\'echet may have liberated the theory of measure and integration
from its geometric roots, but both Fr\'echet and Kolmogorov 
continued to reserve the word measure for geometric settings.
As Stanis{\l}aw Ulam explained in 1943, measure has two characteristic 
properties:  additivity for disjoint sets and equality of measure
for sets that are congruent or otherwise considered equivalent
(Ulam 1943\nocite{ulam:1943}, p.~601).
We do find early examples in which ``measure'' is used in reference to an 
additive set function on a set not necessarily endowed with a 
congruence or equivalence relation:  
Ulam himself in German (1930\nocite{ulam:1930}, 1932\nocite{ulam:1932})
and Eberhard Hopf in English (1934)\nocite{hopf:1934}.  
But the usage seems to have become standard only after the second world war.
Doob's example is instructive.  He began using 
``measure function'' in a completely abstract context 
in his articles in the late 
1930s\nocite{doob:1937,doob:1938,doob:1940a,doob:1940b,doob:1942,doob:1945,doob:1947,doob:1948}
and then abbreviated it, more and more often during the 1940s, 
to ``measure''.   The full phrase ``measure function'' still surfaces occasionally in 
Doob's 1953 book\nocite{doob:1953}, but by then the modern usage had been cemented
in place by his student Paul Halmos in \emph{Measure Theory}\nocite{halmos:1950},
published in 1950.

Nikodym's theorem was the beginning of a torrent of work on abstract
measure and integration in the 1930s.
We can gain some perspective on what happened by looking
at the two editions of Stanis{\l}aw Saks's textbook on integration.
The first\nocite{saks:1933}, which appeared in French almost simultaneously
with Kolmogorov's \emph{Grundbegriffe} (the preface is dated May 1933), 
discusses the Perron and Denjoy integrals as well as the
Lebesgue integral, but stays, throughout the eleven chapters of the
main text, within Euclidean space.  We find abstract integration only in
a fifteen-page appendix, entitled
``Int\'egrale de Lebesgue dans les espaces abstraits'', which
opens with a bow to Radon, Fr\'echet, the Radon-Nikodym theorem, and 
Ulam's 1932 announcement in 1932 concerning the construction of product measures\nocite{ulam:1932}.
In the second edition four years later, 
in 1937\nocite{saks:1937}, the abstract Lebesgue integral comes at the beginning,
as the topic of Chapter I, now with bows to Radon, Daniell, Nikodym, and Jessen.
There is again an appendix on the Lebesgue integral in abstract spaces,
but this one is written by Stefan Banach,
for whom an integral was a linear operator.

Banach's appendix was one of the early signs of a split between two schools of thought
concerning the integral,
which Bourbaki (1994\nocite{bourbaki:1994}, p.~228)
traces back to Carath\'eodory's 1918 textbook.
One school, following Carath\'eodory,
has made measure ever more abstract, axiomatized, and basic.  
The other, following Young, Daniell, and Banach,
takes integration as basic and puts more emphasis on the topological and geometric structures
that underlie most instances of integration.

Kolmogorov was a vigorous participant in the discussion of measure and integration
in the late 1920s and early 1930s.
In 1933, Saks cited three of Kolmogorov's articles,
including one in which Kolmogorov advanced a novel theory of integration
of his own (1930a)\nocite{kolmogorov:1930a}.
In 1937, Saks cited these same articles again
but took no notice of the \emph{Grundbegriffe}.

\subsection{Fr\'echet's integral}\label{subsec:frechet}

In an interview in 1984,
Fr\'echet's student Jean Ville recalled how Fr\'echet had wanted him to 
write a dissertation in analysis, not in probability.  In Paris in the 1930s, Ville
explained, probability was considered merely an honorable pastime for those who had
already distinguished themselves in pure mathematics (Cr\'epel 1984\nocite{crepel:1984}, p.~43).
Fr\'echet's own career, like that of his older colleagues Borel
and Castelnuovo, had followed this pattern.  His dissertation, 
completed in 1906 under Jacques Hadamard, can be said to have launched 
general topology.%
\footnote{The term ``general topology'' seems to have come into common use
          only after the second world war.
          Earlier names included ``point-set theory'', ``analysis situ'', and ``general analysis''.
          The topological spaces that interested Fr\'echet
          and his colleagues were spaces of functions,
          and the objects of greatest interest were real-valued functions on these spaces.
          Following a suggestion by Hadamard that has now been universally adopted, Fr\'echet
          called such functions ``fonctionnelles'',
          and he often called general topology ``le calcul fonctionnel'' (Taylor 1982, pp.~250--251).  
          Nowadays we speak of ``functional analysis'' or ``the theory of functions''.}
He continued to concentrate on general topology and linear functionals until 1928,
when, at the behest of Borel, he moved from Strasbourg to Paris
and turned his main attention to probability and statistics.
His stature as a mathematician assured him a leadership role.
In 1941, at the age of 63, he succeeded to Borel's chair
in Calculus of Probabilities and Mathematical Physics at the University of Paris.%
\footnote{In his 1956 autobiography\nocite{wiener:1956} (p.~50),
          Norbert Wiener recalled that in 1920 he would not have been surprised
          had Fr\'echet turned out ``to be the absolute leader
          of the mathematicians of his generation''.
          That things turned out differently was due, Wiener thought
          in hindsight, to the excessive abstract formalism of Fr\'echet's work.
          Others have pointed to the fact that Fr\'echet contributed, both
          in general topology and in probability, more definitions than theorems
          (Taylor 1982, 1985, 1987\nocite{taylor:1982,taylor:1985,taylor:1987}).
          
	  Borel and Hadamard first proposed Fr\'echet
          for the Acad\'emie des Sciences in 1934,
          but he was not elected until 1956, when he was 77.
          Harald Cram\'er included an uncharacteristically negative comment
          about Fr\'echet's work in probability and statistics
          in his own scientific memoir (Cram\'er 1976\nocite{cramer:1976}, p.~528).
          Cram\'er was a professor of actuarial science rather than mathematics,
          but he contributed more than Fr\'echet to mathematical statistics,
          and he may have been irritated that Fr\'echet had never invited him to Paris.}

When Fr\'echet generalized Radon's integral in 1915,
he was explicit about what he had in mind:
he wanted to integrate over function space.
In some sense, therefore, he was already thinking about probability.
An integral is a mean value.
In a Euclidean space this might be a mean value
with respect to a distribution of mass or electrical charge,
but we cannot distribute mass or charge over a space of functions.
The only thing we can imagine distributing over such a space
is probability or frequency.

Why did Fr\'echet fail at the time to elaborate his ideas on abstract integration,
connecting them explicitly with probability?
One reason was the war.  Mobilized on August 4, 1914, Fr\'echet was
at or near the front for about two and a half years, and he was still in uniform
in 1919.  Thirty years later, he still had the notes on abstract integration that 
he had prepared, in English, for the course he had planned to teach at the University
of Illinois in 1914--1915.  

We should also note that Fr\'echet was not always enthusiastic about
axioms.  In a lecture delivered in 1925
(Fr\'echet 1955\nocite{frechet:1955}, pp.~1--10), 
he argued that the best principles for purely
mathematical work are not necessarily best for practical science
and education.  He felt that too much emphasis had been put on axioms in 
geometry, mechanics, and probability; for some purposes these topics should
be de-axiomatized.

Regardless of Fr\'echet's opportunities and inclinations, 
the time was not yet ripe in 1915 for general theorizing
about probability in function space.  The problem, as the American
mathematician Theophil Hildebrandt pointed out in 1917\nocite{hildebrandt:1917} 
(p.~116), was the lack of interesting examples.
Fr\'echet's integral, he said,
\begin{quotation}
\noindent
\dots depends upon the existence, in a general class of elements, of an
absolutely additive function $v$ whose range is a class of subclasses of the
fundamental class, i.e., a function such that 
$$
    v(\Sigma_n E_n) = \Sigma_n v(E_n),
$$
the $E_n$ being mutually distinct and finite or denumerably infinite in number.
The examples of this which have been given for the general space are trivial
in that they reduce either to an infinite sum or an integral extended over a
field in a finite number of dimensions.  There is still lacking a really
effective and desirable absolutely additive function for the higher type
of spaces. \dots
\end{quotation}
In his presidential address to the American Mathematical Society on 
New Year's Day in 
1915\nocite{vanvleck:1915}, Edward Van Vleck had worried even about the 
examples that Hildebrandt dismissed as trivial.
According to Van Vleck, Poincar\'e, Borel, and 
Felix Bernstein had clarified the problem of mean motion by
showing that exceptional cases have only measure zero.  But 
\begin{quotation}
\noindent
   \dots care must be
   taken inasmuch as measure is not an invariant of analysis situ and hence may be 
   dependent on the parameters introduced.  This application of measure is 
   as yet prospective rather than actual. \dots (p.~337) 
\end{quotation}
The motions of classical physics are functions of time and hence
belong in function space, but the laws of classical physics are deterministic, and so we can 
put probability only on the initial conditions. 
In the examples under discussion, the initial conditions are finite-dimensional and can 
be parameterized in different ways.
The celebrated ergodic theorems of Birkhoff and von Neumann,
published in 1932, (Zund 2002)\nocite{zund:2002}, 
did provide a rationale for the choice of parameterization,
but they were still concerned
with Lebesgue measure on a finite-dimensional space of initial conditions.

The first nontrivial examples of probability in function space were provided 
by Daniell and Wiener.

\subsection{Daniell's integral and Wiener's differential space}\label{subsec:danwien}

Percy Daniell, an Englishman working at the Rice Institute 
in Houston, Texas,
introduced his integral in a series of 
articles in the \emph{Annals of Mathematics} from 1918 to  
1920\nocite{daniell:1918,daniell:1919a,daniell:1919b,daniell:1920}.  Although he
built on the work of Radon and Fr\'echet, his viewpoint owed more to 
earlier work by William~H. Young\nocite{young:1911,young:1915,hildebrandt:1917}.  

Like Fr\'echet, Daniell considered an abstract set $E$.  But instead of 
beginning with an additive set function on subsets of $E$, he began with 
what he called an integral on $E$---a linear operator on some class
$T_0$ of real-valued functions on $E$.  The class
$T_0$ might consist of all continuous functions
(if $E$ is endowed with a topology), or perhaps of all
step functions.  Applying Lebesgue's methods
in this general setting, Daniell extended the linear operator to a
wider class $T_1$ of functions on $E$, the summable functions.
In this way, the Riemann integral is extended to the
Lebesgue integral, the Stieltjes integral to the Radon integral,
and so on (Daniell 1918\nocite{daniell:1918}).  
Using ideas from Fr\'echet's dissertation, Daniell also
gave examples in infinite-dimensional spaces 
(Daniell 1919a,b\nocite{daniell:1919a,daniell:1919b}).

In a remarkable but unheralded 
1921 article\nocite{daniell:1921} in 
the \emph{American Journal of Mathematics}, Daniell used
his theory of integration to analyze 
the motion of a particle whose infinitesimal 
changes in position are independently and normally distributed.
Daniell said he was studying \emph{dynamic probability}.  We now 
speak of \emph{Brownian motion}, with reference to the botanist
Robert Brown, who described the erratic motion of pollen in 
the early nineteenth century (Brush 1968\nocite{brush:1968}).  Daniell cited
work in functional analysis by Vito Volterra and work in probability
by Poincar\'e and Pearson, but he appears to have been unaware
of the history of his problem, for he cited neither Brown, 
Poincar\'e's student Louis Bachelier 
(Bachelier 1900\nocite{bachelier:1900}, 
Courtault and Kabanov 2002\nocite{courtault/kabanov:2002}),
nor the physicists
Albert Einstein and Marian von Smoluchowski 
(Einstein 1905, 1906\nocite{einstein:1905,einstein:1906};
von Smoluchowski 1906\nocite{vonsmoluchowski:1906}).
In retrospect, we may say that Daniell's was the first rigorous
treatment, in terms of functional analysis, of the mathematical
model that had been studied less rigorously
by Bachelier, Einstein, and von Smoluchowski.
Daniell remained at the Rice Institute in Houston until 1924, when 
he returned to England to teach at Sheffield University.
Unaware of the earlier work on Brownian motion,
he seems to have made no effort to publicize his work on the topic,
and no one else seems to have taken notice of it until Stephen Stigler
spotted it in 1973 in the course of a systematic search for 
articles related to probability and statistics in the 
\emph{American Journal of Mathematics}
(Stigler 1973\nocite{stigler:1973}).

The American ex-child prodigy and polymath
Norbert Wiener, when he came upon Daniell's
1918 and July 1919 articles, was in a better position 
than Daniell himself to appreciate
and advertise their remarkable potential for probability
(Wiener 1956\nocite{wiener:1956}, Masani 1990\nocite{masani:1990},
Segal 1992\nocite{segal:1992}).  As a philosopher
(he had completed his Ph.D. in philosophy at Harvard before
studying with Bertrand Russell in Cambridge), Wiener was well
aware of the intellectual significance of Brownian motion
and of Einstein's mathematical model for it.
As a mathematician (his mathematical mentor was G.~H. 
Hardy at Cambridge), he knew the new functional analysis,
and his Cincinnati friend I. Alfred Barnett had
suggested that he use it to study Brownian motion
(Masani 1990, p.~77).

In November 1919, Wiener submitted his first article on Daniell's integral 
to the \emph{Annals of Mathematics}, the journal where Daniell's
four articles on it had appeared.  This article did not yet discuss
Brownian motion; it merely laid out a general method for setting
up a Daniell integral when
the underlying space $E$ is a function space.  But by August 1920, 
Wiener was in France to explain his ideas
on Brownian motion to Fr\'echet and L\'evy (Segal, 1992\nocite{segal:1992}, p.~397).
Fr\'echet, he later recalled,
did not appreciate its importance, but L\'evy
showed greater interest, once convinced that there was a difference
between Wiener's method of integration and G\^ateax's
(Wiener 1956, p.~64). 
Wiener followed up with a quick series of articles: 
a first exploration of
Brownian motion (1921a)\nocite{wiener:1921a}, 
an exploration of what later became known as the Ornstein-Uhlenbeck
model (1921b\nocite{wiener:1921b}, Doob 1966\nocite{doob:1966}),
a more thorough and 
later much celebrated article on Brownian motion (``Differential-Space'')
in 1923\nocite{wiener:1923}, and a final installment in 
1924\nocite{wiener:1924}.  

Because of his work on cybernetics after the second world war, 
Wiener is now the best known of the twentieth-century
mathematicians in our story---far
better known to the general intellectual public than Kolmogorov.
But he was not well known in the early 1920s, and though the 
most literary of mathematicians---he published fiction and social 
commentary---his mathematics was never easy to read, and the early articles attracted
hardly any immediate readers or followers.  Only after he became 
known for his work on Tauberian theorems in the later 1920s, and only after 
he returned to Brownian motion in collaboration with
C.~A.~B. Paley and Antoni Zygmund in the early 1930s
(Paley, Wiener, and Zygmund 1933\nocite{paley/wiener/zygmund:1933})
do we see the work recognized as central and seminal for the emerging theory
of continuous stochastic processes.  In 1934\nocite{levy:1934}, L\'evy 
finally demonstrated that he really understood what Wiener had done by publishing 
a generalization of Wiener's 1923 article.  

Wiener's basic idea was simple.
Suppose we want to formalize the notion of Brownian motion
for a finite time interval, say $0\le t \le 1$.
A realized path is a function on $[0,1]$.
We want to define mean values for certain functionals
(real-valued functions of the realized path).
To set up a Daniell integral that gives these mean values, 
Wiener took $T_0$ to consist of
functionals that depend only on the path's
values at a finite number of time points.  One can 
find the mean value of such a functional using 
Gaussian probabilities for the changes from each time point to the next.  
Extending this integral by Daniell's method, he succeeded in defining mean values 
for a wide class of functionals.
In particular, he obtained probabilities (mean values for indicator functions)
for certain sets of paths.
He showed that the set of continuous paths has probability one,
while the set of differentiable paths has probability zero.

It is now commonplace to translate this work into 
Kolmogorov's measure-theoretic framework.  
Kiyoshi It\^o, for example, in a commentary published 
along with Wiener's articles from this period in
Volume 1 of Wiener's collected works\nocite{wiener:1976}, 
writes as follows (p.~515) concerning 
Wiener's 1923 article\nocite{wiener:1923}:
\begin{quotation}
  Having investigated the differential space from various directions,
  Wiener defines the Wiener measure as a $\sigma$-additive probability 
  measure by means of Daniell's theory of integral.
\end{quotation}
It should not be thought, however, that Wiener defined 
a $\sigma$-additive probability measure and then found mean values as 
integrals with respect to that measure.  Rather, 
as we just explained, he started with mean values and 
used Daniell's theory to obtain more.
This Daniellian approach to probability, making mean
value basic and probability secondary, has
long taken a back seat to Kolmogorov's approach, but it still has its
supporters (Whittle 2000\nocite{whittle:2000}, 
Haberman 1996\nocite{haberman:1996}).

Today's students sometimes find it puzzling that Wiener could
provide a rigorous account of Brownian motion before Kolmogorov had
formulated his axioms---before probability itself had been made rigorous.
But Wiener did have a rigorous mathematical framework:  functional
analysis.  As Doob wrote in his commemoration of Wiener in 1966\nocite{doob:1966},
``He came into probability from analysis and made no concessions.''
Wiener set himself the task of using functional analysis to 
describe Brownian motion.  This meant using some method of integration
(he knew many, including Fr\'echet's and Daniell's)
to assign a ``mean'' or ``average value'' to certain functionals, and this 
included assigning ``a measure, a probability'' (Wiener 1924, p.~456)
to certain events.  It did not involve advancing an abstract
theory of probability, and it certainly did not involve divorcing the
idea of measure from geometry, which is deeply implicated
in Brownian motion.  From Wiener's point of view,
whether a particular number
merits being called a ``probability'' or a ``mean value'' in a particular context
is a practical question, not to be settled by abstract theory.

Abstract theory was never Wiener's predilection.
He preferred, as Irving Segal put it (1992, p.~422), 
``a concrete incision to an abstract envelopment''.
But we should mention the close relation between 
his work and L\'evy's general ideas about probability in abstract spaces.
In his 1920 paper on Daniell's integral\nocite{wiener:1920} 
(p.~66), Wiener pointed out that
we can use successive partitioning to set up such an integral in 
any space.  We finitely partition the 
space and assign probabilities (positive numbers adding to one)
to the elements of the partition,
then we further partition each element of the first partition, 
further distributing its probability, and so on.  
This is an algorithmic rather than an axiomatic picture, but it can
be taken as defining what should be meant by a probability measure
in an abstract space.  
L\'evy adopted this viewpoint in his 1925 book (p.~331), with due
acknowledgement to Wiener.  L\'evy later called a probability measure 
obtained by this process of successive partitioning a
\emph{true probability law} (1970\nocite{levy:1970}, pp.~65--66).

\subsection{Borel's denumerable probability}
\label{subsec:entangle}

Impressive as it was and still is, Wiener's work played little role in
the story leading to Kolmogorov's \emph{Grundbegriffe}.  The starring role 
was played instead by \'Emile Borel.

In retrospect, Borel's use of measure theory in complex analysis
in the 1890s already looks like probabilistic reasoning.
Especially striking in this respect is the argument Borel
gave in 1897 for his claim that a Taylor series will usually
diverge on the boundary of its circle of convergence\nocite{borel:1897}.
In general, he asserted, 
successive coefficients of the Taylor series, or at least successive
groups of coefficients, are independent.  He showed that each 
group of coefficients determines an arc on the circle,
that the sum of lengths of the arcs diverges, and 
that the Taylor series will diverge at a point on the circle
if it belongs to infinitely many of the arcs.  The arcs being 
independent, and the sum of their lengths being infinite, 
a given point must be in infinitely many of them.  To make 
sense of this argument, we must evidently take ``in general''
to mean that the coefficients are chosen at random and
``independent'' to mean probabilistically independent;
the conclusion then follows by what we now call the Borel-Cantelli Lemma.  
Borel himself used
probabilistic language when he reviewed this work in 1912\nocite{borel:1912} 
(Kahane 1994\nocite{kahane:1994}), and Steinhaus spelled 
the argument out in fully probabilistic terms in 1930\label{p.kreis}
(Steinhaus 1930a)\nocite{steinhaus:1930a}.
For Borel in the 1890s, however, complex analysis was not a domain for 
probability, which was concerned with events
in the real world.

In the new century, Borel did begin to explore the implications for probability 
of his and Lebesgue's work on measure and integration (Bru 2001\nocite{bru:2001}).  
His first comments came in an article in 1905\nocite{borel:1905}, where he
pointed out that the new theory justified
Poincar\'e's intuition that a point chosen at random from
a line segment would be incommensurable with probability one
and called attention to 
Anders Wiman's work on continued fractions 
(1900\nocite{wiman:1900}, 1901\nocite{wiman:1901}),
which had been inspired by
the question of the stability of planetary motions,
as an application of measure theory to probability.

Then, in 1909\nocite{borel:1909a}, Borel published a startling 
result---his strong law of large numbers (Borel 1909a).
This new result strengthened measure theory's connection
both with geometric probability and with the heart of 
classical probability theory---the concept of independent trials.
Considered as a statement in geometric probability, 
the law says that the fraction of ones in the
binary expansion of a real number chosen at random from 
$[0,1]$ converges to one-half with probability one.
Considered as a statement about independent trials
(we may use the language of coin tossing, though Borel did not), 
it says that the fraction of heads in a denumerable sequence of independent 
tosses of a fair coin converges to one-half with probability one.
Borel explained the geometric interpretation, and he asserted that
the result can be established using measure theory 
(\S I.8).  But he set measure theory aside for philosophical reasons
and provided an imperfect proof using denumerable 
versions of the rules of total and compound probability.
It was left to others, most immediately Faber (1910\nocite{faber:1910}) 
and Hausdorff (1914\nocite{hausdorff:1914}),
to give rigorous measure-theoretic proofs 
(Doob 1989\nocite{doob:1989}, 1994\nocite{doob:1994}; 
von Plato 1994\nocite{vonplato:1994}).

Borel's discomfort with a measure-theoretic treatment
can be attributed to his unwillingness to assume countable
additivity for probability
(Barone and Novikoff 1978\nocite{barone/novikoff:1978},
von Plato 1994\nocite{vonplato:1994}).  
He saw no logical absurdity in
a countably infinite number of zero probabilities adding to a
nonzero probability, and so instead of general appeals to 
countable additivity, he preferred arguments that derive probabilities 
as limits as the number of trials increases (1909a\nocite{borel:1909a}, \S I.4).  
Such arguments seemed to him stronger than formal appeals to 
countable additivity, for they exhibit the finitary pictures that are
idealized by the infinitary pictures.
But he saw even more fundamental problems in 
the idea that Lebesgue measure can model a random choice
(von Plato 1994\nocite{vonplato:1994}, pp.~36--56; Knobloch 2001\nocite{knobloch:2001}).
How can we choose a real number at random when most
real numbers are not even definable in any constructive sense?

Although Hausdorff did not hesitate to equate Lebesgue measure
with probability, his account of Borel's strong law, in 
his \emph{Grundz\"uge der Mengenlehre} in 1914 (pp.~419--421),
treated it as a theorem about real numbers:
the set of numbers in $[0,1]$ with binary expansions for which 
the proportion of ones converges to one-half has Lebesgue measure one.
But in 1916 and 1917\nocite{cantelli:1916a,cantelli:1916b,cantelli:1917}, 
Francesco Paolo Cantelli rediscovered the 
strong law (he neglected, in any case, to cite Borel)
and extended it to the more general result that the average of
bounded random variables will converge to their mean with arbitrarily
high probability.  Cantelli's work inspired
other authors to study the strong law and to sort out different concepts of 
probabilistic convergence.  

By the early 1920s, it seemed to some that there were two different 
versions of Borel's strong law---one concerned with real numbers and 
one concerned with probability.  
In 1923\nocite{steinhaus:1923}, Hugo Steinhaus proposed to clarify matters
by axiomatizing Borel's theory of denumerable probability along
the lines of Sierpi\'nski's axiomatization of Lebesgue measure.
Writing $A$ for the set of all infinite sequences of 
$\rho$s and $\eta$s ($\rho$ for ``rouge'' and $\eta$ for ``noir'';
now we are playing red or black rather than heads 
or tails), Steinhaus proposed the following axioms for 
a class $\KKK$ of subsets of $A$ and a real-valued function 
$\mu$ that gives probabilities for the elements of $\KKK$\label{p:stein}:
\begin{enumerate}\renewcommand{\labelenumi}{\Roman{enumi}}
  \item
     $\mu(E) \ge 0$ for all $E \in \KKK$.
  \item
     \begin{enumerate}\renewcommand{\labelenumii}{\arabic{enumii}}
       \item
          For any finite sequence $e$ of $\rho$s and $\eta$s, 
          the subset $E$ of $A$ consisting of all infinite sequences that begin with
          $e$ is in $\KKK$.  
       \item
          If two such sequences $e_1$ and $e_2$ differ in only one place,
          then $\mu(E_1) = \mu(E_2)$, where $E_1$ and $E_2$ are the corresponding
          sets.
       \item
          $\mu(A) =1$.
     \end{enumerate}
  \item
     $\KKK$ is closed under finite and countable unions of disjoint 
     elements, and $\mu$ is finitely and countably additive.
  \item
     If $E_1 \supset E_2$ and $E_1$ and $E_2$ are in $\KKK$, then 
     $E_1 \setminus E_2$ is in $\KKK$.
  \item
     If $E$ is in $\KKK$ and $\mu(E) = 0$, then any subset of $E$ is in $\KKK$.
\end{enumerate}
Sierpi\'nski's axioms for Lebesgue measure consisted of
I, III, IV, and V, together with an axiom that says that
the measure $\mu(J)$ of an interval $J$ is its length.
This last axiom being demonstrably equivalent to Steinhaus's axiom II,
Steinhaus concluded that the theory of probability for an infinite 
sequence of binary trials is isomorphic with the theory of Lebesgue measure.

In order to show that his axiom II is equivalent to setting the 
measures of intervals equal to their length, Steinhaus used the Rademacher
functions---the $n$th Rademacher function being the function that assigns 
a real number the value $1$ or $-1$ depending on whether the $n$th digit in
its dyadic expansion is $0$ or $1$.  He also used these functions,
which are independent random variables, in deriving Borel's strong law
and related results.  The work by Rademacher (1922)\nocite{rademacher:1922} 
and Steinhaus
marked the beginning of the Polish school of ``independent
functions'', which made important contributions 
to probability theory during the period between the wars 
(Holgate 1997\nocite{holgate:1997}).

Steinhaus cited Borel but not Cantelli.  The work of Borel and 
Cantelli was drawn together,
however, by the Russians, especially by Evgeny Slutsky in his wide-ranging
article in the Italian journal \emph{Metron} in 1925.  
Cantelli, it seems, was not 
aware of the extent of Borel's priority until he debated the matter
with Slutsky at the International Congress of Mathematicians at
Bologna in 1928 (Seneta 1992\nocite{seneta:1992}, Bru 2003a\nocite{bru:2003a}).

The name 
``strong law of large numbers'' was introduced by Khinchin
in 1928\nocite{khinchin:1928}.
Cantelli had used ``uniform'' instead of ``strong''.  
          The term ``law of large numbers'' had been introduced
          originally by Poisson (1837\nocite{poisson:1837})
          and had come to be used as 
          a name for Bernoulli's theorem (or for the conclusion, from 
          this theorem together with Cournot's principle, that the 
          frequency of an event will approximate its probability), although Poisson had
          thought he was naming a generalization 
          (Stigler 1986\nocite{stigler:1986}, p.~185).

\subsection{Kolmogorov enters the stage}\label{subsec:kolenter}

Although Steinhaus considered only binary trials in his 1923 article, 
his reference to Borel's more general concept of denumerable probability
pointed to generalizations.
We find such a generalization in Kolmogorov's first article on probability, 
co-authored by Khinchin (Khinchin and Kolmogorov 1925\nocite{khinchin/kolmogorov:1925}),
which showed that a series of discrete random variables $y_1 + y_2 + \cdots$
will converge with probability one when the series of means and the series of 
variances both converge.
The first section of the article, due to Khinchin,
spells out how to represent the random variables as functions on $[0,1]$:
divide the interval into segments with lengths equal to the 
probabilities for $y_1$'s possible values, then divide each of these 
segments into smaller segments with lengths proportional to the probabilities
for $y_2$'s possible values, and so on.
This, Khinchin notes with a nod to Rademacher and Steinhaus, reduces the problem 
to a problem about Lebesgue measure.  This reduction was
useful because the rules for working with Lebesgue measure were 
clear, while Borel's picture of denumerable probability remained murky.

Dissatisfaction with this detour into Lebesgue 
measure must have been one impetus for the \emph{Grundbegriffe}
(Doob 1989\nocite{doob:1989}, p.~818).  
Kolmogorov made no such detour in his next article on
the convergence of sums of independent random variables.  In this sole-authored
article, dated 24 December 1926 and published in 1928\nocite{kolmogorov:1928}, 
he took probabilities and expected values as his 
starting point.  But even then, he did not appeal to Fr\'echet's 
countably additive calculus.
Instead, he worked with finite additivity and then stated an explicit 
ad hoc definition when he passed to a limit.  For example, he 
defined the probability $P$ that the series $\sum_{n=1}^{\infty} y_n$
converges by the equation
$$
  P
  =
  \lim_{\eta \to 0} \lim_{n \to \infty} \lim_{N \to \infty}
  \mathfrak{W} \left[ \mathrm{Max} \bigg\vert \sum_{k=n}^p y_k \bigg\vert_{p=n}^N < \eta \right],
$$
where $\mathfrak{W}(E)$ denotes the probability of the event $E$.
(This formula does not appear in the Russian and English translations 
provided in Kolmogorov's collected works\nocite{kolmogorov:1986,kolmogorov:1992}; 
there the argument has been modernized so as to eliminate it.)  
This recalls the way Borel proceeded in 1909: think through each passage to the limit.

It was in his seminal article on Markov processes
(\"Uber die analytischen Methoden
in der Wahrscheinlichkeitsrechnung\nocite{kolmogorov:1931}, 
dated 26 July 1930 and published in 1931)
that Kolmogorov first explicitly and freely used Fr\'echet's calculus
as his framework for probability.
In this article, Kolmogorov considered a system with a set of states $\mathfrak{A}$.
For any two time points $t_1$ and $t_2$ ($t_1 < t_2$), any state $x\in\mathfrak{A}$,
and any element $\mathfrak{E}$ in a collection $\mathfrak{F}$ of subsets of $\mathfrak{A}$,
he wrote
\begin{equation}\label{eq:trans}
   P(t_1,x,t_2,\mathfrak{E})
\end{equation}
for the probability, when the system is in state $x$ at time $t_1$,
that it will be in a state in $\mathfrak{E}$ at time $t_2$.
Citing Fr\'echet, Kolmogorov assumed that 
$P$ is countably additive as a function of $\mathfrak{E}$ and that
$\mathfrak{F}$ is closed under differences and countable unions and contains the empty set,
all singletons, and $\mathfrak{A}$.
But the focus was not on Fr\'echet; it was on the equation that ties together
the transition probabilities~(\ref{eq:trans}), now called the Chapman-Kolmogorov
equation.  The article launched the study of this equation by purely analytical methods, 
a study that kept probabilists occupied for fifty years.  

As many commentators have noted, the 1931 article makes no reference to probabilities 
for trajectories.  There is no suggestion that 
such probabilities are needed in order for a stochastic process to be well defined.
Consistent transition probabilities, it seems, are enough.  Bachelier 
(1900\nocite{bachelier:1900}, 1910\nocite{bachelier:1910}, 1912\nocite{bachelier:1912})
is cited as the first to study continuous-time stochastic processes, 
but Wiener is not cited.  We are left wondering when Kolmogorov first became aware of Wiener's 
success in formulating probability statements concerning Brownian trajectories.
He certainly was aware of it in 1934, when he reviewed L\'evy's generalization
of Wiener's 1923 article in the \emph{Zentralblatt}\nocite{kolmogorov:1934b}.  
He then called that article well known, 
but so far as we are aware, he had not mentioned it in print earlier.

\section{Hilbert's sixth problem}
\label{sec:hilbert}

At the beginning of the twentieth century,
many mathematicians were dissatisfied with what they saw
as a lack of clarity and rigor in the probability 
calculus.  The whole calculus seemed to be concerned with 
concepts that lie outside mathematics:  event, trial, 
randomness, probability.
As Henri Poincar\'e wrote,
``one can hardly give a satisfactory definition of probability''
(1912\nocite{poincare:1912}, p.~24).

The most celebrated call for clarification came from David Hilbert.
The sixth of the twenty-three
open problems that Hilbert presented to the International
Congress of Mathematicians in Paris in 1900 was to treat axiomatically,
after the model of geometry,
those parts of physics in which mathematics already played an outstanding role,
especially probability and mechanics (Hilbert 1902\nocite{hilbert:1902}).  
To explain what he meant by 
axioms for probability, Hilbert cited Georg Bohlmann,
who had labeled the rules of total and compound probability 
axioms rather than theorems in his
lectures on the mathematics of life insurance (Bohlmann 1901\nocite{bohlmann:1901}).
In addition to a logical investigation of these axioms,
Hilbert called for a ``rigorous and satisfactory development
of the method of average values in mathematical physics, especially
in the kinetic theory of gases''.

Hilbert's call for a mathematical treatment of 
average values was answered in part by the work on integration 
that we discussed in the preceding
section, but his suggestion that the classical rules for
probability should be treated as axioms on the model of geometry 
was an additional challenge.
Among the early responses, we may mention the following:
\begin{itemize}
  \item
     In his dissertation, in 1904\nocite{laemmel:1904} in Z\"urich,
     Rudolf Laemmel discussed the rules of total and compound probability as axioms.
     But he stated the rule of compound probability only in the case of independence,
     a concept he did not explicate.  
     Schneider reprinted excerpts from this dissertation in 
     1988\nocite{schneider:1988} (pp.~359--366).
  \item
     In a 1907 dissertation\nocite{broggi:1907} in G\"ottingen, directed by Hilbert himself, 
     Ugo Broggi gave only two axioms:  an axiom stating that the sure event 
     has probability one, and the rule of total probability. 
     Following tradition, he then defined probability as a ratio 
     (a ratio of numbers of cases in the discrete setting; a ratio of 
     the Lebesgue measure of two sets in the geometric setting) and verified his axioms.
     He did not state an axiom corresponding to the classical rule
     of compound probability.
     Instead, he gave this name to a rule for calculating 
     the probability of a Cartesian product, which he derived from 
     the definition of geometric probability in terms of Lebesgue measure. 
     Again, see Schneider\nocite{schneider:1988} (pp.~367--377) for excerpts.
     Broggi mistakenly claimed that his axiom of total probability 
     (finite additivity) implied countable additivity; see Steinhaus 1923\nocite{steinhaus:1923}.
  \item
     In an article written in 1920, published in 1923\nocite{lomnicki:1923},
     and listed in the bibliography of the \emph{Grundbegriffe}, 
     Antoni {\L}omnicki proposed that
     probability should always be understood relative to a density $\phi$
     on a set $\mathcal{M}$ in $\bbbr^r$.
     {\L}omnicki defined this probability by combining two of Carath\'eodory's ideas:
     the idea of $p$-dimensional measure
     and the idea of defining the integral of a function on a set
     as the measure of the region between the set and the function's graph
     (see \S\ref{subsec:borleb} above). 
     The probability of a subset $m$ of $\mathcal{M}$, according to {\L}omnicki,
     is the ratio of the measure of the region between $m$ and $\phi$'s graph  
     to the measure of the region between $\mathcal{M}$ and this graph.
     If $\mathcal{M}$ is an $r$-dimensional subset of $\bbbr^r$, then the 
     measure being used is Lebesgue measure on $\bbbr^{r+1}$; if $\mathcal{M}$ 
     is a lower-dimensional subset of $\bbbr^r$, say $p$-dimensional,
     then the measure is the $(p+1)$-dimensional Carath\'eodory measure.
     This definition covers discrete as well as continuous probability;
     in the discrete case, $\mathcal{M}$ is a set of discrete points,
     the function $\phi$ assigns each point its probability, and the
     region between a subset $m$ and the graph of $\phi$ consists of a line
     segment for each point in $m$, whose Carath\'eodory measure is its
     length---i.e., the point's probability.
     The rule of total probability follows.  Like Broggi,
     {\L}omnicki treated the rule of compound probability as a rule 
     for relating probabilities on a Cartesian product to probabilities
     on its components.  He did not consider it an axiom, because it 
     holds only if the density itself is a product density.
\end{itemize}
Two general tendencies are notable here:  
an increasing emphasis on measure,
and an attendant decline in the role of compound probability.
These tendencies are also apparent in the 1923 article by 
Steinhaus that we have already discussed.  
Steinhaus did not mention compound probability.

We now discuss at greater length responses by Bernstein,
von Mises, Slutsky, Kolmogorov, and Cantelli.

\subsection{Bernstein's qualitative axioms}

In an article published in Russian in 1917\nocite{bernstein:1917}
and listed by Kolmogorov in the \emph{Grundbegriffe}'s bibliography, 
Sergei Bernstein showed that probability theory can be
founded on qualitative axioms for 
numerical coefficients that measure the probabilities of 
propositions.  

Bernstein's two most important axioms correspond
to the classical axioms of total and compound probability:
\begin{itemize}
  \item
    If $A$ and $A_1$ are equally likely, $B$ and $B_1$ are
    equally likely, $A$ and $B$ are incompatible, and $A_1$ and $B_1$
    are incompatible, then $(A \mbox{ or } B)$ and 
    $(A_1 \mbox{ or } B_1)$ are equally likely.
  \item
    If $A$ occurs, the new probability of 
    a particular occurrence $\alpha$ of $A$ is a function of 
    the initial probabilities of $\alpha$ and $A$.
\end{itemize}
Using the first axiom, Bernstein concluded that if $A$ is the
disjunction of $m$ out of $n$ equally likely and incompatible 
propositions, and $B$ is as well, then $A$ and $B$ must be 
equally likely.  It follows that the numerical probability of $A$ and $B$
is some function of the ratio $m/n$, and we may as well take 
that function to be the identity.  Using the second axiom, 
Bernstein then finds that the new probability of $\alpha$ when
$A$ occurs is the ratio of the initial probability of $\alpha$
to that of $A$.  

Bernstein also axiomatized the field of propositions
and extended his theory to the case where this field
is infinite.  
He exposited his qualitative axioms again in a probability textbook that he 
published in 1927\nocite{bernstein:1927}, but neither the article nor
the book were ever translated out of Russian into other languages.
John Maynard Keynes included
Bernstein's article in the bibliography of the 1921 book\nocite{keynes:1921}
where he developed his own system of qualitative probability.  
Subsequent writers on qualitative probability, most prominently
Bernard~O. Koopman in 1940\nocite{koopman:1940} and 
Richard~T. Cox in 1946\nocite{cox:1946},
acknowledged a debt to Keynes but not to Bernstein.
The first summary of Bernstein's ideas in English appeared only in 
1974, when Samuel Kotz published an English translation of 
Leonid~E. Maistrov's history of probability\nocite{maistrov:1974}.  

Unlike von Mises and Kolmogorov, Bernstein was not a frequentist.
He had earned his doctorate in Paris, and the philosophical 
views he expresses at the end of his 1917 article are in line
with those of Borel and L\'evy:  probability is essentially 
subjective and becomes objective only when there is sufficient
consensus or adequate confirmation of Cournot's principle.

\subsection{Von Mises's Kollektivs}\label{subsec:vonmises}
 
The concept of a Kollektiv was introduced into the German scientific
literature by Gustav Fechner in the 1870s and popularized in
his posthumous \emph{Kollektivmasslehre}, 
which appeared in 1897\nocite{fechner:1897} (Sheynin 2004)\nocite{sheynin:2004}.  The concept 
was quickly taken up by Georg Helm (1902\nocite{helm:1902})
and Heinrich Bruns (1906\nocite{bruns:1906}).

Fechner wrote about the concept of a Kollektiv-Gegenstand (collective object)
or a Kollektiv-Reihe (collective series).
It was only later,
in Meinong (1915\nocite{meinong:1915}) for example,
that we see these names abbreviated to Kollektiv.
As the name Kollektiv-Reihe indicates,
a Kollektiv is a population of individuals given in a certain order;
Fechner called the ordering the Urliste.
It was supposed to be irregular---random, we would say.
Fechner was a practical scientist,
not concerned with the theoretical notion of probability.
But as Helm and Bruns realized,
probability theory provides a framework for studying Kollektivs.

Richard von Mises was a well established applied mathematician when he
took up the concept of a Kollektiv in 1919\nocite{vonmises:1919,vonmises:1928,vonmises:1931}.
His contribution was to realize that the concept can be made
into a mathematical foundation for probability theory.  
As von Mises defined it, a Kollektiv is an infinite sequence of
outcomes satisfying two axioms:
\begin{enumerate}
  \item
     the relative frequency of each outcome converges 
     to a real number (the probability of the outcome) as 
     we look at longer and longer initial segments of the sequence, and
  \item
     the relative frequency converges to the same probability
     in any subsequence selected without knowledge of the future 
     (we may use knowledge of the outcomes so far in deciding whether 
     to include the next outcome in the subsequence).
\end{enumerate}
The second property says we cannot change the odds by selecting 
a subsequence of trials on which to bet; this is von Mises's version of
the ``hypothesis of the impossibility of a gambling system'', 
and it assures the irregularity of the Urliste.  

According to von Mises, the purpose of the probability calculus is to 
identify situations where Kollektivs exist and the probabilities
in them are known, and to derive from probabilities for other 
Kollektivs from these given probabilities.
He pointed to three domains where probabilities for Kollektivs
are know:
\begin{enumerate}
   \item
     Games of chance, where devices are carefully constructed
     so the axioms will be satisfied.
   \item
     Statistical phenomena, where the two axioms can sometimes be
     confirmed, to a reasonable degree.
   \item
     Theoretical physics, where the two axioms play the same hypothetical
     role as other theoretical assumptions.
\end{enumerate}
(See von Mises 1931\nocite{vonmises:1931}, pp.~25--27.)

Von Mises derived the classical rules of probability, such as 
the rules for adding and multiplying probabilities, from
rules for constructing new Kollektivs out of an initial one.
He had several laws of
large numbers.  The simplest was his definition of probability:
the probability of an event is the event's limiting frequency in a Kollektiv.
Others arose as one constructed further Kollektivs.

Von Mises's ideas were taken up by a number of mathematicians in the 1920s and 1930s.
Kolmogorov's bibliography includes an article
by Arthur Copeland (1932\nocite{copeland:1932},
based on his 1928\nocite{copeland:1928} article)
that proposed founding probability theory
on particular rules for selecting subsequences in von Mises's scheme, 
as well as articles by Karl D\"orge (1930\nocite{dorge:1930}), 
Hans Reichenbach (1932\nocite{reichenbach:1932}), and 
Erhard Tornier (1933\nocite{tornier:1933}) arguing for alternative schemes.
But the most prominent mathematicians of the time,
including the G\"ottingen mathematicians (MacLane 1995\nocite{maclane:1995}), 
the French probabilists, and the British statisticians,
were hostile or indifferent. 

After the publication of the \emph{Grundbegriffe}, 
Kollektivs were given a rigorous mathematical basis by 
Abraham Wald (1936, 1937, 1938)\nocite{wald:1936,wald:1937,wald:1938}
and Alonzo Church (1940)\nocite{church:1940}, 
but the claim that they provide a foundation for probability was refuted by Jean Ville
(1936\nocite{ville:1936}, 1939\nocite{ville:1939}).
Ville pointed out that whereas a Kollektiv
in von Mises's sense will not be vulnerable to a gambling system
that chooses a subsequence of trials on which to bet,
it may still be vulnerable to a more clever gambling system,
which also varies the amount of the bet and the outcome on which to bet.
Ville called such a system a ``martingale''.
Ville's work inspired Doob's measure-theoretic martingales.
Our own game-theoretic foundation for probability
(Shafer and Vovk 2001)\nocite{shafer/vovk:2001}
combines Ville's idea of a martingale with ideas about 
generalizing probability that go back to 
Robert Fortet (1951, pp.~46--47)\nocite{fortet:1951,bayer:1951}.

\subsection{Slutsky's calculus of valences}

In an article published in Russian in 1922\nocite{slutsky:1922},
Evgeny Slutsky presented a viewpoint that greatly influenced
Kolmogorov.  As Kolmogorov said on the occasion of Slutsky's death in 1948, 
Slutsky was ``the first to give the right picture 
of the purely mathematical content of probability theory''\nocite{kolmogorov:1948a}.

How do we make probability theory 
purely mathematical?  Markov had claimed to do this in his textbook,
but Slutsky did not think Markov had succeeded, for 
Markov had retained the subjective notion of
equipossibility, with all its subjectivity.  
The solution, Slutsky felt, was to remove both the word
``probability'' and the notion of equally likely cases 
from the theory.  Instead of beginning with equally likely cases, 
one should begin by assuming merely that numbers are assigned to 
cases, and that when a case assigned the number $\alpha$ is further subdivided, the numbers
assigned to the subcases should add to $\alpha$.  
The numbers assigned to cases might be equal, or they 
might not.  The addition and multiplication theorems would be
theorems in this abstract calculus.  But it should not be called the
probability calculus.  In place of ``probability'', he suggested the 
unfamiliar word валентность,
or ``valence''.  (He may have been following Laemmel, who had used the
German ``valenz'', which can be translated into English as ``weight''.)
Probability would be only one interpretation of the
calculus of valences, a calculus fully as abstract as group theory.

Slutsky listed three distinct interpretations of the calculus of valences:
\begin{enumerate}
  \item
     Classical probability (equally likely cases).
  \item
     Finite empirical sequences (frequencies).
  \item
     Limits of relative frequencies. (Slutsky remarks that this 
     interpretation is particularly popular with the English school.)
\end{enumerate}
Slutsky did not think probability could be reduced to limiting frequency,
because sequences of independent trials have properties that go beyond
their possessing limiting frequencies.
He gave two examples.
First, initial segments of the sequences have properties that are not imposed
by the eventual convergence of the relative frequency.
Second, the sequences must be irregular in a way
that resists the kind of selection discussed by von Mises
(he cites an author named Умов,
not von Mises).

In his 1925 article on limit theorems\nocite{slutsky:1925}, 
written in German,
Slutsky returned briefly to his calculus of valences, now saying
that it should follow the lines laid out by Bernstein and that it should 
be called \emph{Disjunktionsrechnung}, a name  
Ladislaus von Bortkiewicz had suggested to him in a letter.  
\emph{Disjunktionsrechnung} was to be distinguished from \emph{Stochastik},
the science concerned with the application of \emph{Disjunktionsrechnung}
to random phenomena.  

The word ``stochastic'' was first associated with probability 
by Jacob Bernoulli.  Writing in Latin (1713, p.~213), 
Bernoulli called the art of combining arguments and data 
``\emph{Ars Conjectandi} sive \emph{Stochastice}''---the art of conjecturing or guessing.
Greek in origin, ``stochastic''
may have seemed as erudite and mysterious to Bernoulli as it does to us.
It was not used widely until it was revived and promoted
by von Bortkiewicz in 1917\nocite{vonbortkiewicz:1917}.%
\footnote{On p.~3, von Bortkiewicz wrote:
  \begin{quotation}
   The consideration of empirical quantities, oriented to probability
   theory but based on the ``law of large numbers'', may be called 
   S\hspace{.04cm}t\hspace{.04cm}o\hspace{.04cm}c\hspace{.04cm}h\hspace{.04cm}a\hspace{.04cm}s\hspace{.04cm}t\hspace{.04cm}i\hspace{.04cm}k
   ($\sigma\tau o \chi\acute{\alpha}\zeta\epsilon\sigma\theta\alpha\iota$ =
   aim, conjecture).
   Stochastik is not merely probability theory, but rather 
   probability theory in its application \dots  ``Stochastik''
   also signified applied probability theory for Bernoulli; only he
   represented an antiquated viewpoint with respect to the applications
   considered.   
  \end{quotation}
   For one appearance of ``stochastic'' between Bernoulli 
   and von Bortkiewicz, see
   Prevost and Lhuilier (1799)\nocite{prevost/lhuilier:1799}.}
As von Bortkiewicz's saw it, the probability calculus is 
mathematics, and \emph{Stochastik} is the science
of applying it to the real world.
The term was adopted, with acknowledgement to von
Bortkiewicz, by Chuprov in 1922\nocite{chuprov:1922} and 
Du Pasquier in 1926\nocite{dupasquier:1926}.
Slutsky entitled his 1925 article 
``\"Uber stochastische Asymptoten und Grenzwerte'', 
stochastic limits being the limiting frequencies
and averages we extract from statistical data.
In 1923\nocite{mordukh:1923}, Chuprov's student Jacob Mordukh 
called exchangeability ``stochastic commutativity''.
In 1934\nocite{khinchin:1934}, Khinchin gave 
``stochastic process'' the meaning it has today.%
\footnote{In 1931\nocite{kolmogorov:1931},
          Kolmogorov had called what we now call 
          a ``Markov process'' a ``stochastically definite process''
          and had used ``stochastic process'' 
          casually.  In addition to fixing the precise
          meaning of ``stochastic process'', Khinchin's
          1934 article introduced ``Markov process''
          (Dynkin (1989\nocite{dynkin:1989}).  
          ``Markov chain'' was already in use
          in the 1920s (Bru 2003a\nocite{bru:2003a}).}
The prestige of the Moscow school assured the international
adoption of the term, but von Bortkiewicz's gloss on ``stochastic'' was lost.
In Stalin's time, the Russians were not interested in suggesting that 
the applications of probability require a separate science.
So stochastic became merely a mysterious but international synonym for
random, al\'eatoire, zuf\"allig, casuale, and случайный.

\subsection{Kolmogorov's general theory of measure}
\label{subsec:kllmm29}

As we have seen, Kolmogorov had used Fr\'echet's integral hesitantly 
in his article on the convergence of sums of random variables in 1926,
but forthrightly in his article on Markov processes in 1930\nocite{kolmogorov:1931}.
Between these two mathematical contributions, 
he wrote a thought piece about how probability 
should fit into a broader abstract theory.
Dated 8 January 1927 and published in 1929\nocite{kolmogorov:1929},
this piece was neither a survey (it did not cite Fr\'echet, Steinhaus, or Slutsky)
nor a report on research, and it included no proofs.
It was philosophical---a sketch of a program.

By all appearances, the piece was inspired by Slutsky.
But whereas Slutsky was a statistician,
Kolmogorov was already established as a mathematician's mathematician.
In 1922, at the age of 19, Kolmogorov had constructed an integrable function
whose Fourier series diverges almost everywhere (1923\nocite{kolmogorov:1923}).  
In 1925 he had published two notes on the theory of integration
in the \emph{Comptes rendus}\nocite{kolmogorov:1925a,kolmogorov:1925b}.
Slutsky had mentioned frequencies as an alternative
interpretation of a general calculus.  Kolmogorov pointed to
more mathematical examples:  the distribution of digits in the 
decimal expansions of irrationals, Lebesgue measure in an 
$n$-dimensional cube, and the density of a set $A$ of positive integers
(the limit as $n\to\infty$ of the fraction of the integers between $1$ and $n$
that are in $A$).  

The abstract theory Kolmogorov sketches is concerned with a function $M$ that assigns
a nonnegative number $M(E)$
to each element $E$ of class of subsets of a set $A$.
He calls $M(E)$ the measure
(мера)
of $E$,
and he calls $M$ a measure specification
(мероопределение).
So as to accommodate all the mathematical examples he has in mind, 
he assumes, in general, neither that $M$ is countably additive
nor that the class of subsets to which it assigns numbers is a field. 
Instead, he assumes only that when $E_1$ and $E_2$ are disjoint
and $M$ assigns a number to two of the three sets $E_1$, $E_2$, and $E_1 \cup E_2$,
it also assigns a number to the third, and that 
$$
   M(E_1 \cup E_2) = M(E_1) + M(E_2)
$$
then holds (cf.\ Steinhaus's Axioms III and IV).
In the case of probability, however, he does suggest (using different 
words) that $M$ should be countably additive and that 
class of subsets to which it assigns numbers should be a field,
for only then can we uniquely define probabilities for countable
unions and intersections, and this seems necessary to justify 
arguments involving events such as the convergence of random variables.

He defines the abstract Lebesgue integral
of a function $f$ on $A$, and
he comments that countable additivity is to be assumed whenever such
an integral is discussed.  He writes $M_{E_1}(E_2) = M(E_1 E_2)/M(E_1)$
``by analogy with the usual concept of relative probability''.
He defines independence for partitions, and he comments, no doubt in 
reference to Borel's strong law and other results in number theory,
that the notion of independence is responsible for the power of probabilistic methods within 
pure mathematics.

The mathematical core of the~\emph{Grundbegriffe} is already here.
Many years later, in his commentary in Volume II of his collected works
(p.~520 of the English edition),
Kolmogorov said that only the set-theoretic treatment of conditional probability
and the theory of distributions in infinite products was missing.
Also missing, though, is the bold rhetorical move that Kolmogorov made
in the \emph{Grundbegriffe}---giving the abstract theory the name probability.

\subsection{The axioms of Steinhaus and Ulam}\label{subsec:steinulam}

In the 1920s and 1930s, the city of Lw\'ow in Poland%
\footnote{Though it was in Poland between the two world wars,
          this city is now in Ukraine.
          Its name is spelled differently in different languages:
          Lw\'ow in Polish, Lviv in Ukrainian, and Lvov in Russian.
          When part of Austria-Hungary and, briefly, Germany, it was called Lemberg.
          Some articles in our bibliography refer to it as L\'eopol.}
was a vigorous center of mathematical research, led by Hugo Steinhaus.
In 1929 and 1930, Steinhaus's work on limit theorems intersected with Kolmogorov's, 
and his approach promoted the idea that probability should be axiomatized
in the style of measure theory.

As we saw in \S\ref{subsec:entangle}, Steinhaus had already, in 1923,
formulated axioms for heads and tails isomorphic to Sierpi\'nski's axioms for Lebesgue measure.
This isomorphism had more than a philosophical purpose;
Steinhaus used it to prove Borel's strong law.
In a pair of articles written in 1929 and published in 1930
(Steinhaus 1930a,b\nocite{steinhaus:1930a,steinhaus:1930b}),
Steinhaus extended his approach to limit theorems
involving an infinite sequence of independent draws $\theta_1,\theta_2,\dots$
from the interval $[0,1]$.
Here are the axioms he gave
(Steinhaus 1930b, pp.~23--24; we translate loosely from the French):
\begin{enumerate}
  \item
     Probability is a nonnegative finite number.
  \item
     The probability that $\theta_i \in \Theta_i$ for $i=1,\dots,n$,
     where the $\Theta_i$ are measurable subsets of $[0,1]$, is
     $$
       \left|\Theta_1\right| \cdot \left|\Theta_2\right| \cdots \left|\Theta_n\right|,
     $$
     where $\left|\Theta_i\right|$ is the Lebesgue measure of $\Theta_i$. 
  \item
     If $E_1,E_2,\dots$ is a sequence of disjoint subsets of $[0,1]^{\infty}$,
     and for every $k$ the probability $\mu(E_k)$ for the sequence $\theta_1,\theta_2,\dots$
     being in $E_k$ is determined, then the probability $\mu(\cup_{k=1}^{\infty} E_k)$
     is determined and is equal to $\sum_{k=1}^{\infty} \mu(E_k)$.  
  \item
     If $E_1 \supset E_2$ and $\mu(E_1)$ and $\mu(E_2)$ are determined,
     then $\mu(E_1 \setminus E_2)$ is determined.  
  \item
     If $\mu(E_1) = 0$ and $E_1 \supset E_2$, then $\mu(E_2)$ is determined.
  \item
     If $\mathfrak{K}^*$ is a class of $E$ such that it is possible to define 
     $\mu(E)$ for all the $E$ in $\mathfrak{K}$ while respecting postulates 1--5,
     and if $\mathfrak{K}$ is the intersection of all the $\mathfrak{K}^*$
     with this property, then $\mu(E)$ is defined only for $E$ in $\mathfrak{K}$.
\end{enumerate}
Except for the second one, these axioms are identical with Steinhaus's axioms 
for heads and tails (see p.~\pageref{p:stein}).  
In the case of heads and tails, the second  
axiom specified probabilities for each initial finite sequence
of heads and tails.  Here it specifies probabilities for 
$\theta_1,\theta_2,\dots,\theta_n$.

Steinhaus presented his axioms as a ``logical extrapolation''
of the classical axioms to the case of an infinite number of trials
(Steinhaus 1930b, p.~23).
They were more or less tacitly used, he asserted,
in all classical problems, such as the problem of the gambler's ruin,
where the game as a whole---not merely finitely many rounds---must 
be considered (Steinhaus 1930a, p.~409).  

As in the case of heads and tails, Steinhaus showed that there
are probabilities uniquely satisfying his axioms by setting up an isomorphism with 
Lebesgue measure on $[0,1]$, this time using a sort of Peano curve to map 
$[0,1]^{\infty}$ onto $[0,1]$.
He used the isomorphism to prove several limit theorems,
including one formalizing Borel's 1897 claim
concerning the circle of convergence of a Taylor's series
with randomly chosen coefficients.

Steinhaus's axioms were measure-theoretic, but they were not yet abstract.  
His words suggested that his ideas should apply to all sequences of random variables,
not merely ones uniformly distributed,
and he even considered the case where the variables were complex-valued rather than real-valued,
but he did not step outside the geometric context to consider probability on abstract spaces.
This step was taken by Stanis{\l}aw Ulam, one of Steinhaus's junior colleagues at Lw\'ow.
At the International Congress of Mathematicians
in Z\"urich in 1932, Ulam announced that he and another Lw\'ow mathematician,
Zbigniew {\L}omnicki (a nephew of Antoni {\L}omnicki), 
had shown that product measures can be constructed in
abstract spaces (Ulam 1932\nocite{ulam:1932}).

Ulam and {\L}omnicki's axioms for measure were simple.
They assumed that the class of $\mathfrak{M}$ of measurable sets on space $X$
satisfy four conditions:
\begin{enumerate}\renewcommand{\labelenumi}{\Roman{enumi}.}
   \item
     $X\in \mathfrak{M}$, and $\{x\}\in\mathfrak{M}$ for all $x\in \mathfrak{M}$.
   \item
     If $M_i \in \mathfrak{M}$ for $i=1,2,\dots$, then $\cup_{i=1}^{\infty} M_i \in \mathfrak{M}$.
   \item
     If $M,N \in \mathfrak{M}$, then $M \setminus N \in \mathfrak{M}$.
   \item
     If $M \in \mathfrak{M}$, $m(M) = 0$, and $N \subset M$, then $N \in \mathfrak{M}$.
\end{enumerate}
And he assumed that the measures $m(M)$ satisfy three conditions:
\begin{enumerate}
   \item
     $m(X) = 1$; $m(M) \ge 0$.
   \item
     $m(\cup_{i=1}^{\infty} M_i) = \sum_{i=1}^{\infty} m(M_i)$
     when $M_i \cap M_j = \emptyset$ for $i \ne j$.
   \item
     If $m(M) = 0$ and $N \subset M$, then $m(N) = 0$.
\end{enumerate}
In today's language, this says that $m$ is a probability measure on a $\sigma$-algebra that is
complete (includes all null sets) and contains all singletons.  
Ulam announced that from a countable sequence of spaces with such
probability measures, one can construct a probability measure satisfying the same 
conditions on the product space.

We do not know whether Kolmogorov knew about Ulam's announcement when he 
wrote the \emph{Grundbegriffe}.  Ulam's axioms would have held no novelty for him,
but he would presumably have found the result on product measures interesting.
{\L}omnicki and Ulam's article, which appeared 
only in 1934\nocite{lomnicki/ulam:1934}, cites the
\emph{Grundbegriffe}, not Ulam's earlier announcement, when it lists its axioms.
Kolmogorov cited the article in 1935\nocite{kolmogorov:1935}.

\subsection{Cantelli's abstract theory}

Like Borel, Castelnuovo, and Fr\'echet, Francesco Paolo Cantelli
turned to probability after distinguishing himself in other areas
of mathematics.  His first publications on probability, expositions
of the classical theory in a teaching journal, appeared in 1905 
and 1906\nocite{cantelli:1905,cantelli:1905/1906}, when he was 
about 40.  His work on the strong law of large numbers came 
ten years later.

It was only in the 1930s, about the same time as the \emph{Grundbegriffe}
appeared, that Cantelli introduced his own abstract theory of probability.
This theory, which has important affinities
with Kolmogorov's, is developed most clearly in 
``Una teoria astratta del calcolo delle probabilit\`a'',
published in 1932\nocite{cantelli:1932} and 
listed in the~\emph{Grundbegriffe}'s bibliography,
and in ``Consid\'erations sur la convergence dans le calcul des
probabilit\'es'', a lecture Cantelli
gave in 1933 at the Institut Henri Poincar\'e in Paris
and subsequently published in 1935\nocite{cantelli:1935}.

In the 1932 article, Cantelli argued for a theory
that makes no appeal to empirical notions such as possibility, event, probability,
or independence.
This abstract theory, he said, should begin with a set of points
having finite nonzero measure.  This could be any set for which measure is defined, 
perhaps a set of points on a surface.
He wrote $m(E)$ for the area of a subset $E$.
He noted that $m(E_1 \cup E_2) = m(E_1) + m(E_2)$
and $0 \le m(E_1 E_2)/m(E_i) \le 1$ for $i = 1,2$.
He called $E_1$ and $E_2$ \emph{multipliable} when
$m(E_1 E_2) = m(E_1) m(E_2)$.  Much of probability theory, he pointed out,
including Bernoulli's law of large numbers
and Khinchin's law of the iterated logarithm, could be carried
out at this abstract level.

In the 1935 article, Cantelli explains how his abstract theory
should be related to frequencies in the world.
The classical calculus of probability, he says,
should be developed for a particular class of events in the world
in three steps:
\begin{enumerate}
   \item
     Study experimentally the 
     equally likely cases (check that they happen equally frequently),
     thus justifying experimentally the rules of total and 
     compound probability.
   \item
     Develop an abstract theory based only on 
     the rules of total and compound probability, without 
     reference to their empirical justification.
   \item
     Deduce probabilities from the abstract theory, and use them
     to predict frequencies.
\end{enumerate}
His own abstract theory, Cantelli explains, is precisely
the theory one obtains in the second step.
The theory can begin with cases that are not equally likely.  
But application of the theory, in any case, involves
initial verification and subsequent prediction of frequencies.
Cantelli reviews earlier Italian discussion of 
the relation between probability and frequency and quotes with 
approval his friend and collaborator Guido Castelnuovo, 
who had explained that limiting frequency should be taken
not as a basis for the logical construction of the calculus of probabilities
but rather as a way of connecting the calculus to its applications\nocite{castelnuovo:1919}.

A fair evaluation of the importance of Cantelli's role is clouded
by the cultural differences that separated him from Kolmogorov,
who represented a younger generation,
and even from his contemporaries Bernstein, Fr\'echet,
and Slutsky (Benzi 1988\nocite{benzi:1988}, Bru 2003a\nocite{bru:2003a}).
Cantelli belonged to an older mathematical culture
that had not absorbed the new theory of functions,
and so from Fr\'echet and Kolmogorov's point of view,
one could question whether he understood even his own discoveries.
Yet he was quick to contrast his own mathematical depth
with the shallowness of others;
he offended Slutsky with his disdain for Chuprov (Seneta 1992\nocite{seneta:1992}).

We cannot really say that Cantelli's 1932 article and 1933 lecture
were sources for the \emph{Grundbegriffe}.
The theory in Kolmogorov's 1929 article
already went well beyond anything Cantelli did in 1932,
in both degree of abstraction (instead of developing an abstract measure theory,
Cantelli had simply identified events with subsets of a geometric space for which
measure was already defined) and mathematical clarity.
The 1933 lecture was more abstract but obviously came too late to influence the~\emph{Grundbegriffe}.
On the other hand, we may reasonably assume that when Cantelli prepared the 1932 article
he had not seen Kolmogorov's 1929 article,
which was published in Russian
at a time when Russian mathematicians published their best work in German, French, and Italian.
So we may say that Cantelli developed independently of Kolmogorov
the project of combining a frequentist interpretation of probability
with an abstract axiomatization
that stayed close to the classical rules of total and compound probability.
This project was in the air.

\section{The \emph{Grund\-begriffe}}
\label{sec:grundbegriffe}

When Kolmogorov sat down to write the~\emph{Grundbegriffe}, 
in a rented cottage on the Klyaz'ma River in November 1932, 
he was already being hailed as 
the Soviet Euler.  Only 29 years of age, his 
accomplishments within probability theory alone included
definitive versions of 
the law of the iterated logarithm
(Kolmogorov 1929b\nocite{kolmogorov:1929b})
and the strong law of large numbers 
(Kolmogorov 1930b\nocite{kolmogorov:1930b}), 
as well as his pathbreaking article on Markov processes\nocite{kolmogorov:1931}.
He was personally acquainted with many of the leading mathematicians
of his time; his recent trip abroad, from June 1930 to March 1931, had included
a stay at Hilbert's department in G\"ottingen and
extensive discussions with Fr\'echet and L\'evy in France
(Shiryaev 1989\nocite{shiryaev:1989}, 2000\nocite{shiryaev:2000},
2003a\nocite{shiryaev:2003a,shiryaev:2003b}).

The~\emph{Grundbegriffe} was an exposition, not another research contribution.
In his preface, after acknowledging that Fr\'echet had shown how to liberate
measure and integration from geometry, Kolmogorov said this:
\begin{quotation}
  In the pertinent mathematical circles it has been common for some time to
  construct probability theory in accordance with this general point of view.
  But a complete presentation of the whole system, free from superfluous 
  complications, has been missing (though a book by Fr\'echet, [2] in the
  bibliography, is in preparation).
\end{quotation}
Kolmogorov aimed to fill this gap,
and he did so brilliantly and concisely, in just 62 pages.
Fr\'echet's much longer book,
which finally appeared in two volumes in 1937 and 1938
(Fr\'echet 1937--1938\nocite{frechet:1937/1938}),
is remembered only as a footnote to Kolmogorov's achievement.

Fr\'echet's own evaluation of the \emph{Grundbegriffe}'s contribution, 
quoted at the beginning of this article, is correct so far as it goes.
Borel had introduced countable additivity into probability in 1909,
and in the following twenty years many authors, including Kolmogorov,
had explored its consequences.
The \emph{Grundbegriffe} merely rounded out the picture
by explaining that nothing more was needed:
the classical rules together with countable additivity
were a sufficient basis for what had been accomplished thus far
in mathematical probability.
But it was precisely Kolmogorov's mathematical achievement,
especially his definitive work on the classical limit theorems,
that had given him the grounds and the authority to say that nothing more was needed.

Moreover, Kolmogorov's appropriation of the name \emph{probability} 
was an important rhetorical achievement, with enduring implications.
Slutsky in 1922 and Kolmogorov himself in 1927 had proposed
a general theory of additive set functions but had relied on the classical
theory to say that probability should be a special case of this general 
theory.  Now Kolmogorov proposed axioms for probability.
The numbers in his abstract theory were probabilities, not merely
valences or меры.
As we explain in \S\ref{subsec:empirical} below, his philosophical justification
for proceeding in this way
so resembled the justification that Borel, 
Chuprov, and L\'evy had given for the classical theory that they could
hardly raise objections.

It was not really true that nothing
more was needed.  Those who studied Kolmogorov's formulation in detail
soon realized that his axioms and definitions were inadequate in a number
of ways.  Most saliently, his treatment of conditional probability 
was not adequate for the
burgeoning theory of Markov processes, to which he had just made so
important a contribution.  And there were other points in the monograph where
he could not obtain natural results at the abstract level
and had to fall back to the classical examples---discrete probabilities
and probabilities in Euclidean spaces.  
But these shortcomings only gave impetus to the new theory,
for the project of filling in the gaps 
provided exciting work for a new generation of probabilists.

In this section, we take a fresh look at the \emph{Grundbegriffe}.
In \S\ref{subsec:book} we review broadly the contents of 
the book and the circumstances of its publication.   
Then, in \S\ref{subsec:framework}, we review the basic 
framework---the six axioms---and two ideas that were novel at the time:
the idea of constructing probabilities on infinite-dimensional spaces
(which lead to his celebrated consistency theorem),
and the definition of conditional probability
using the Radon-Nikodym theorem.
Finally, in~\S\ref{subsec:empirical},
we look at the explicitly philosophical part of the monograph:
the two pages in Chapter~I where Kolmogorov explains the empirical origin of his axioms.

\subsection{An overview}\label{subsec:book}

In a letter to his close friend Pavel Sergeevich Aleksandrov,
dated November 8, 1932
(Shiryaev 2003b\nocite{shiryaev:2003b}, Vol.~2, pp.~456--458),
Kolmogorov explained that he planned to spend November, December, and January
writing the~\emph{Grundbegriffe}.
The preface is dated Easter, 1933,
and the monograph appeared that same year in a series published by Springer,
\emph{Ergebnisse der Mathematik und Ihrer Grenzgebiete}.

In 1936, when Communist leaders began to complain, in the context of
a campaign against Kolmogorov's teacher Nikolai Luzin,
about the Moscow mathematicians' publishing their best work in the West
(Demidov and Pevshin 1999\nocite{demidov/pevshin:1999};
Vucinich 2000\nocite{vucinich:2000}; Lorentz 2002\nocite{lorentz:2002}, p.~205), 
Kolmogorov had his student Grigory Bavli translate the \emph{Grundbegriffe} into Russian.
An English translation, by Nathan Morrison, was published in 1950,
with a second edition in 1956.  
The second Russian edition, published in 1974, modernized the exposition 
substantially and is therefore less useful for those interested
in Kolmogorov's viewpoint in 1933.  We have made our own translations from the
German original, modernizing the notation in only three minor respects:
$\cup$ instead of $+$, $\cap_n$ instead of $\mathfrak{D}_n$, and 
$\emptyset$ instead of $0$ for the empty set.

The monograph has six chapters:
\begin{enumerate}\renewcommand{\labelenumi}{\textbf{\Roman{enumi}.}}
  \item
     \textbf{The elementary probability calculus.}  This chapter deals with the 
     case where the sample space is finite.  Even in this elementary case, Kolmogorov
     treats conditional probabilities as random variables.
  \item
     \textbf{Infinite probability fields.}  Here Kolmogorov introduces countable
     additivity and discusses the two classical types of probability measures, as we now call
     them: discrete probability measures and probability measures on Euclidean 
     space specified by cumulative distribution functions.
  \item
     \textbf{Random variables.}  Here Kolmogorov proves his consistency theorem:
     Given a Cartesian product $R^M$, where $R$ is the real numbers and $M$ is an 
     arbitrary set, and given consistent distribution functions for all 
     the finite-dimensional marginals, the resulting set function on the 
     field in $R^M$ formed by finite-dimensional cylinder sets is countably additive
     and therefore extends, by Carath\'eodory's theorem, 
     to a probability measure on a Borel field.  Kolmogorov then
     discusses convergence in probability, a concept
     he attributes to Bernoulli and Slutsky.
  \item
     \textbf{Mathematical expectations.}  This chapter begins with a discussion, without 
     proofs, of the abstract Lebesgue integral.  After explaining that the integral
     of a random variable is called its mathematical expectation, Kolmogorov discusses
     Chebyshev's inequality and some criteria for convergence in probability.  In a final 
     section, he gives conditions for interchanging the expectation sign with
     differentiation or integration with respect to a parameter.  
  \item
     \textbf{Conditional probabilities and expectations.}  This is the 
     most novel chapter; here Kolmogorov defines
     conditional probability and conditional expectation using the
     Radon-Nikodym theorem.
  \item
     \textbf{Independence.  Law of large numbers.}  After defining 
     independence and correlation, Kolmogorov reviews, without proofs, his own
     and Khinchin's results on the law of large numbers.
\end{enumerate}
There is also an appendix on the zero-one law.

In his preface, after saying that he aimed to present
the abstract viewpoint without superfluous complications, 
Kolmogorov indicated that there was nevertheless some novelty in the book:
\begin{quotation}
  I would also like to call attention here to the points in the further presentation
  that fall outside the circle of ideas, familiar to specialists, that I just 
  mentioned.  These points are the following:  probability distributions in
  infinite-dimensional spaces (Chapter III, \S4), differentiation and integration
  of mathematical expectations with respect to a parameter (Chapter IV, \S5),
  and above all the theory of conditional probabilities and expectations
  (Chapter V) \dots
\end{quotation}
We will not discuss the differentiation and integration of expectations
with respect to a parameter,
for Kolmogorov's results here, even if novel to specialists in probability,
were neither mathematically surprising nor philosophically significant.
The other two results have subsequently received much more attention,
and we take a look at them in the next section.

\subsection{The mathematical framework}\label{subsec:framework}

Kolmogorov's six axioms for probability are so familiar that it seems superfluous
to repeat them but so concise that it is easy to do so.  We do repeat them
(\S\ref{subsubsec:six}), and then we discuss the two points just 
mentioned:  the
consistency theorem (\S\ref{subsubsec:consistency})
and the treatment of conditional probability and expectation~(\S\ref{subsubsec:radnik}).
As we will see, a significant amount of mathematics was involved in both 
cases, but most of it was due to earlier authors---Daniell in
the case of the consistency theorem and Nikodym in the case of conditional
probabilities and expectations.
Kolmogorov's contribution, more rhetorical and philosophical than
mathematical, was to bring this mathematics into his framework for probability.

\subsubsection{The six axioms}\label{subsubsec:six}

Kolmogorov began with five axioms concerning a set $E$
and a set $\FFF$ of subsets of $E$,
which he called \emph{random events}:
\begin{enumerate}\renewcommand{\labelenumi}{\Roman{enumi}}
   \item
      $\FFF$ is a field of sets.%
\footnote{For definitions of \emph{field of sets} and \emph{Borel field},
          see p.~\pageref{p:field}.}
   \item
      $\FFF$ contains the set $E$.
   \item
      To each set $A$ from $\FFF$ is assigned a nonnegative 
      real number $\mathsf{P}(A)$.  This number $\mathsf{P}(A)$
      is called the probability of the event $A$.
   \item
      $\mathsf{P}(E) = 1$.
   \item
      If $A$ and $B$ are disjoint, then 
      $$
             \mathsf{P}(A \cup B) = \mathsf{P}(A)+ \mathsf{P}(B).
      $$
\end{enumerate}
He then 
added a sixth axiom, redundant for finite $\FFF$ but independent of 
the first five axioms for infinite $\FFF$:
\begin{enumerate}\renewcommand{\labelenumi}{\Roman{enumi}}
  \setcounter{enumi}{5}
  \item
     If $A_1 \supseteq A_2 \supseteq \cdots$ is a decreasing sequence
     of events from $\FFF$ with $\bigcap_{n=1}^{\infty} A_n = \emptyset$,
     then $\lim_{n\to\infty}\mathsf{P} (A_n) = 0$.
\end{enumerate}
This is the \emph{axiom of continuity}.  Given the first five axioms,
it is equivalent to countable additivity.

The six axioms can be summarized by saying that $\mathsf{P}$
is a nonnegative additive set function in the sense of Fr\'echet with 
$\mathsf{P}(E) = 1$.

In contrast with Fr\'echet, who had debated countable additivity with 
de Finetti a few years before
(Fr\'echet 1930\nocite{frechet:1930},
de Finetti 1930\nocite{definetti:1930},
Cifarelli and Regazzini 1996\nocite{cifarelli/regazzini:1996}),
Kolmogorov did not try to make a substantive argument for it.  
Instead, he made this statement (p.~14):
\begin{quote}
   \dots  Since the new axiom is essential only for infinite fields of probability,
   it is hardly possible to explain its empirical meaning\dots\,.
   In describing any actual observable random process, we can obtain only
   finite fields of probability.  Infinite fields of probability occur only
   as idealized models of real random processes.  \emph{This understood,
   we limit ourselves arbitrarily to models that satisfy Axiom VI.}  
   So far this 
   limitation has been found expedient in the most diverse investigations.
\end{quote}
This echoes Borel, who 
adopted countable additivity not as a matter of principle but because 
he had not encountered circumstances where its rejection seemed
expedient (1909a\nocite{borel:1909a}, \S I.5).  But Kolmogorov was much clearer
than Borel about the purely instrumental significance of infinity.

\subsubsection{Probability distributions in infinite-dimensional spaces}
\label{subsubsec:consistency}

Suppose, using modern terminology, that 
$(E_1,\mathfrak{F}_1),(E_2,\mathfrak{F}_2),\dots$ 
is a sequence of measurable spaces.  For each finite set of 
indices, say $i_1,\dots,i_n$, write $\mathfrak{F}^{i_1,\dots,i_n}$  
for the induced $\sigma$-algebra in the product space $\prod_{j=1}^n E_{i_j}$.
Write $E$ for the product of all the $E_i$, and write 
$\mathfrak{F}$ for the algebra (not a $\sigma$-algebra) consisting of all the 
cylinder subsets of $E$ corresponding to elements of the various $\mathfrak{F}^{i_1,\dots,i_n}$.
Suppose we define consistent probability measures for all the 
marginal spaces $(\prod_{j=1}^n E_{i_j},\mathfrak{F}^{i_1,\dots,i_n})$.
This defines a set function on $(E,\mathfrak{F})$.  Is it 
countably additive?

In general, the answer is negative; a counterexample was 
given by Erik Sparre Andersen and B{\o}rge Jessen in 1948\nocite{andersen/jessen:1948}.
But as we noted in \S\ref{subsec:steinulam}, Ulam had announced in 1932\nocite{ulam:1932}
a positive answer for the case where all the
measures on the marginal spaces are product measures.
Kolmogorov's consistency theorem, in \S 4 of Chapter III of the
\emph{Grundbegriffe}, answered it affirmatively for another case,
where each $E_i$ is a copy of the real numbers and each $\mathfrak{F}_i$ consists
of the Borel sets.
(Formally, though, Kolmogorov had a slightly different starting point:
finite-dimensional distribution functions, not finite-dimensional measures.)

In his September 1919 article (Daniel 1919b)\nocite{daniell:1919b},
Daniell had proven a closely related theorem, using an infinite-dimensional
distribution function as the starting point.  In 
Bourbaki's judgment (1994\nocite{bourbaki:1994}, p.~243), 
the essential mathematical content of Kolmogorov's result is 
already in Daniell's.  
But Kolmogorov did not cite Daniell in the \emph{Grundbegriffe}, 
and even 15 years later, in 1948 (\S3.1),
Gnedenko and Kolmogorov\nocite{gnedenko/kolmogorov:1948}
ignored Daniell while claiming
the construction of probability measures on infinite products as a Soviet achievement.
What are we to make of this?

In a commemoration of Kolmogorov's early work, 
published in 1989\nocite{doob:1989}, Doob hazards the guess
that Kolmogorov was unaware of Daniell's 
result\nocite{daniell:1918,daniell:1919a,daniell:1919b,daniell:1920}
when he wrote the \emph{Grundbegriffe}.
This may be true.
He would not have been the first author to do repeat Daniell's result;
B{\o}rge Jessen presented the result
to the seventh Scandinavian mathematical conference in 1929
and became aware of Daniell's priority only
in time to acknowledge it in a footnote to the published version
(Jessen 1930)\nocite{jessen:1930}.
Saks did not even mention Daniell's integral in his first edition in 1933.
This integral became more prominent later that year and the next,
with the increased interest in Wiener's work on Brownian motion,%
\footnote{One small indication that Kolmogorov may have been aware of
  Wiener's and Daniell's articles before the end of 1933 is the fact
  that he submitted an article on Brownian motion to 
  the \emph{Annals of Mathematics}\nocite{kolmogorov:1934a}, the American
  journal where Daniell and Wiener had published, in 1933;
  it appeared in 1934 with the legend ``Received September 9, 1933''.}
but the articles by Daniell that Wiener had cited did not include
the September 1919 article.

In 1935, Wiener\nocite{wiener:1935} and Jessen\nocite{jessen:1935} both 
called attention to Daniell's priority.  So 
it seems implausible that Gnedenko and Kolmogorov would have remained 
unaware of Daniell's construction in 1948.
But their claim of priority for Kolmogorov may appear more reasonable
when we remember that the \emph{Grundbegriffe} was not meant as a contribution to pure mathematics.
Daniell's and Kolmogorov's theorems seem almost identical when they 
are assessed as mathematical discoveries, but they differed in context and
purpose.  Daniell was not thinking about probability, whereas the slightly
different theorem formulated by Kolmogorov was about probability. 
Neither Daniell nor Wiener undertook to make probability into a
conceptually independent branch of mathematics by establishing a general
method for representing it measure-theoretically.

Kolmogorov's theorem was more general than Daniell's in one respect---%
Kolmogorov considered an index set of arbitrary cardinality,
rather than merely a denumerable one as Daniell had.
This greater generality is merely formal, in two senses;
it involves no additional mathematical complications,
and it does not seem to have any practical application.  
The obvious use of a non-denumerable index would be to represent continuous time\label{p:trajectory},
and so we might conjecture that Kolmogorov is thinking of making probability statements
about trajectories in continuous time,
as Wiener had done in the 1920s but Kolmogorov had not done in 1931.
This might have been motivated not only by Bachelier's work on Brownian motion,
which Kolmogorov had certainly studied,
but also by the work on point processes by Filip Lundberg, Agnar Erlang,
and Ernest Rutherford and Hans Geiger in the first decade of the twentieth century
(Cram\'er 1976\nocite{cramer:1976}, p.~513).
But Kolmogorov's construction does not accomplish anything in this direction.  
The $\sigma$-algebra on the product
obtained by the construction contains too few sets; in the case of 
Brownian motion, it does not include the set of continuous trajectories.
It took some decades of further research were required 
in order to develop general methods of defining $\sigma$-algebras,
on suitable function spaces,
rich enough to include the infinitary events one typically wants 
to discuss (Doob 1953\nocite{doob:1953}; Schwartz 1973\nocite{schwartz:1973};
Dellacherie and Meyer 1975\nocite{dellacherie/meyer:1975};
Bourbaki 1994\nocite{bourbaki:1994}, pp.~243--245).
The generally topological character of these richer extensions
and the failure of the consistency theorem for 
arbitrary Cartesian products remain two important caveats to  
the \emph{Grundbegriffe}'s main thesis---that probability is 
adequately represented by the abstract notion of a probability measure.

It is by no means clear that Kolmogorov even did prefer
to start the study of stochastic processes with unconditional probabilities on trajectories.
Even in 1935, he recommended the opposite
(Kolmogorov 1935\nocite{kolmogorov:1935},
pp.~168--169 of the English translation).
In the preface to the second
Russian edition of the \emph{Grundbegriffe} (1974),
he acknowledged Doob's innovations with a rather neutral comment:  
``Nowadays people prefer to define conditional 
probability with respect to an arbitrary algebra 
$\mathcal{F}^{\prime} \subseteq \mathcal{F}$''\dots.

\subsubsection{Experiments and conditional probability}\label{subsubsec:radnik}

In the case where $A$ has nonzero probability, Kolmogorov defined 
$\mathsf{P}_A(B)$ in the usual way.  
He called it ``bedingte Wahrscheinlichkeit'', which translates into 
English as ``conditional probability''.

His general treatment of conditional probability and expectation 
was novel.  It began with a set-theoretic formalization of the 
concept of an \emph{experiment} (\emph{Versuch} in German).  
Here Kolmogorov had in mind a subexperiment of the grand 
experiment defined by the conditions $\mathfrak{S}$.  This 
subexperiment might give only limited information about the outcome
$\xi$ of the grand experiment.  If the subexperiment is 
well-defined, Kolmogorov reasoned, it should define
a partition $\mathfrak{A}$ of the sample space $E$ for the grand experiment:
its outcome should amount to specifying 
which element of $\mathfrak{A}$ contains $\xi$.
He formally identified
the subexperiment with $\mathfrak{A}$\label{p:decomposition}.
Then he introduced the idea of conditional probability relative
to $\mathfrak{A}$.
As he explained, first in the case of finite fields of probability
(Chapter I) and then in the general case (Chapter V),
this is a random variable, not a single number:
\begin{itemize}
  \item
    In the finite case, he wrote $\mathsf{P}_{\mathfrak{A}}(B)$ for the
    random variable whose value at each point $\xi$ of $E$ is $\mathsf{P}_A(B)$,
    where $A$ is the element of $\mathfrak{A}$ containing $\xi$, and 
    he called this random variable the ``conditional probability of $B$ after
    the experiment $\mathfrak{A}$'' (p.~12).  
    This random variable is well defined for all the $\xi$ 
    in elements $\mathfrak{A}$ that have positive probability, and 
    these $\xi$ form an event that has probability one.
  \item
    In the general case,
    he represented the partition $\mathfrak{A}$ by a
    function $u$ on $E$ that induces it, and he wrote $\mathsf{P}_u(B)$
    for any random variable that satisfies
\begin{equation}\label{eq:nikodym}
   \mathsf{P}_{\{u \subset A\}}(B) = \mathsf{E}_{\{u \subset A\}}  \mathsf{P}_u(B)
\end{equation}
   for every set $A$ of possible values of $u$ such that 
   the subset $\{\xi \st u(\xi) \in A\}$ of $E$ (this is what
   he meant by $\{u \subset A\}$) is measurable and has
   positive probability (p.~42).
   By the Radon-Nikodym theorem (only recently proven by Nikodym), 
   this random variable is unique up to
   a set of probability zero.  Kolmogorov called it 
   the ``conditional probability of $B$ with respect to (or knowing) $u$''.  
   He defined $\mathsf{E}_u(y)$, which he called ``the conditional expectation
   of the variable $y$ for a known value of $u$'', analogously (p.~46).
\end{itemize} 
Because many readers will be more familiar with Doob's slightly different 
definitions, it may be wise to add a few words of clarification.  
The mapping $u$ can be any
mapping from $E$ to another set, say $F$, that represents 
$\mathfrak{A}$ in the sense that it maps two elements of $E$ to the 
same element of $F$ if and only if they are in the same element of
$\mathfrak{A}$.  The choice of $u$ has no effect on the definition of
$\mathsf{P}_u(B)$, because~(\ref{eq:nikodym}) can be 
rewritten without reference to $u$; it says that the conditional
probability for $B$ given $\mathfrak{A}$ is any random variable $z$ such that 
$\mathsf{E}_C (z) = \mathsf{P}_C(B)$ for every $C$ in $\mathfrak{F}$
that is a union of elements of $\mathfrak{A}$ and has positive 
probability.  

Kolmogorov was doing no new mathematics here; the mathematics is
Nikodym's.  But Kolmogorov was the first to prescribe that Nikodym's result
be used to define conditional probability, and this involves a substantial
philosophical shift.  The rule of compound probability had encouraged 
mathematicians of the eighteenth and nineteenth centuries to see 
the probability of an event $A$ ``after $B$ has 
happened'' or ``knowing $B$'' as something to be sought directly in 
the world and then used together with $B$'s probability to construct
the joint probability of the two events.  Now the joint probability distribution
was consecrated as the mathematical starting point---what we abstract from
the real world before the mathematics begins.  Conditional probability
now appears as something defined within the mathematics, not something we 
bring directly from outside.%
\footnote{George Boole had emphasized the possibility of deriving
          conditional probabilities from unconditional probabilities, and 
          Freudenthal and Steiner (1966, p.~189)\nocite{freudenthal/steiner:1966} have pointed to 
          him as a precursor of Kolmogorov.}

\subsubsection{When is conditional probability meaningful?}

To illustrate his understanding of conditional probability,
Kolmogorov discussed Bertrand's paradox of the great circle,
which he called, with no specific reference,
a ``Borelian paradox''.  His explanation of the paradox was simple
but formal.  After noting
that the probability distribution for the second point conditional
on a particular great circle is not uniform, he said (p.~45):  
\begin{quotation}
This demonstrates the inadmissibility of the idea of 
conditional probability with respect to 
a given isolated hypothesis with probability zero.
One obtains a probability distribution for the latitude
on a given great circle only when that great circle is considered as 
an element of a partition of the entire surface of the sphere
into great circles with the given poles.
\end{quotation}
This explanation has become part of the culture of probability theory,
even though it cannot completely replace the
more substantive explanations given by Borel.

Borel insisted that we explain how the measurement on which we will 
condition is to be carried out.  This accords with Kolmogorov's insistence
that a partition be specified, for a procedure for  
measurement will determine such a partition.  Kolmogorov's explicitness on 
this point was a philosophical advance.  On the other hand, 
Borel demanded more than the specification of a partition. 
He demanded that the measurement be specified realistically enough that
we see real partitions into events with nonzero probabilities, 
not merely a theoretical limiting partition into events with zero probabilities.

Borel's demand that the theoretical partition into events of probability
zero be approximated by a 
partition into events of positive probability seems to be needed in 
order to rule out nonsense.  This is illustrated by an
example given by L\'evy in 1959\nocite{levy:1959}.  
Suppose the random variable $X$ is distributed uniformly on 
$[0,1]$.  For each $x \in [0,1]$, let $C(x)$ be the denumerable set given by
$$
   C(x) := \{ x^{\prime} \st x^{\prime}\in [0,1] \mbox{ and } 
              x-x^{\prime} \mbox{ is rational}\}.
$$
Let $\mathfrak{A}$ be the partition
$$
    \mathfrak{A} := \{C(x) \st x \in [0,1]\}.
$$
What is the distribution of $X$ conditional on, say, 
the element $C(\frac{\pi}{4})$
of the partition $\mathfrak{A}$?  We cannot respond that $X$ is 
now uniformly distributed over $C(\frac{\pi}{4})$,
because $C(\frac{\pi}{4})$ is a denumerable set.  
Kolmogorov's formal point of view gives an equally 
unsatisfactory response; it tells us that
conditionally on $C(\frac{\pi}{4})$, the random variable $X$
is still uniform on $[0,1]$, so that the set on which 
we have conditioned, $C(\frac{\pi}{4})$, has conditional probability zero
(see \S\ref{app:garb} below).
It makes more sense, surely, to take Borel's point of view, as 
L\'evy did, and reject the question.  It makes no practical sense 
to condition on the zero-probability
set $C(\frac{\pi}{4})$, because there is no real-world measurement that could
have it, even in an idealized way, as an outcome.

The condition that we consider idealized events of 
zero probability only when they approximate more down-to-earth
events of nonzero probability again brings topological
ideas into the picture.  
This issue was widely discussed in the 1940s and 1950s, 
often with reference to an example discussed by Jean
Dieudonn\'e (1948\nocite{dieudonne:1948}), in which 
the conditional probabilities defined by Kolmogorov do not even have versions 
(functions of $\xi$ for fixed $B$) that form 
probability measures (when considered as functions of $B$ for fixed $\xi$).
Gnedenko and Kolmogorov (1949\nocite{gnedenko/kolmogorov:1949})
and David Blackwell (1956\nocite{blackwell:1956}) addressed the issue
by formulating more or less topological conditions on measurable spaces or probability measures
that rule out pathologies such as those adduced by Dieudonn\'e and L\'evy.
For modern versions of these conditions,
see Rogers and Williams (2000)\nocite{rogers/williams:2000}.

\subsection{The empirical origin of the axioms}
\label{subsec:empirical}

Kolmogorov devoted about two pages of the \emph{Grundbegriffe}
to the relation between his axioms and the real world.
These two pages, the most thorough explanation of his frequentist philosophy
he ever provided, are so important to our story that we quote them in full.
We then discuss how this philosophy was related to the thinking of his predecessors.

\subsubsection{In Kolmogorov's own words}

Section 2 of Chapter I of the~\emph{Grund\-begriffe} is entitled
``Das Verh\"altnis zur Erfahrungswelt''.
It is only two pages in length.
This subsection consists of a translation of the section in its entirety.

\begin{center}
    \textbf{The relation to the world of experience}
\end{center}

\noindent
The theory of probability is applied to the real world of experience as follows:
\begin{enumerate}
   \item
     Suppose we have a certain system of conditions $\mathfrak{S}$, 
     capable of unlimited repetition.
   \item
     We study a fixed circle of phenomena that can arise
     when the conditions $\mathfrak{S}$ are realized.
     In general, these phenomena can
     come out in different ways in different 
     cases where the conditions are realized.
     Let $E$ be the set of the different possible variants $\xi_1,\xi_2,\dots$
     of the outcomes of the phenomena.
     Some of these variants might actually not occur.
     We include in the set $E$ all the variants we regard \emph{a priori}
     as possible.
   \item
     If the variant that actually appears
     when conditions $\mathfrak{S}$ are realized belongs to a set
     $A$ that we define in some way, then we say that the event $A$
     has taken place.
\end{enumerate}

Example.
The system of conditions $\mathfrak{S}$ consists of flipping a coin twice.
The circle of phenomena mentioned in point 2 consists of the 
appearance, on each flip, of head or tails.
It follows that there are four possible variants (\emph{elementary events}), namely
$$
  \mbox{heads---heads, heads---tails, tails---heads, tails---tails}.
$$
Consider the event $A$ that there is a repetition.
This event consists of the first and fourth elementary events.
Every event can similarly be regarded as a set of elementary events.
\begin{enumerate}
   \setcounter{enumi}{3}
   \item
     Under certain conditions, that we will not go into further here,
     we may assume that an event $A$ that does or does not 
     occur under conditions $\mathfrak{S}$ is assigned a real
     number $\mathsf{P} (A)$ with the following properties:
       \begin{enumerate}
       \renewcommand{\labelenumii}{\textbf{\Alph{enumii}.}}
         \item
            One can be practically certain that if the 
            system of conditions $\mathfrak{S}$ is repeated
            a large number of times, $n$, and the event $A$
            occurs $m$ times, then the 
            ratio $m/n$ will differ only slightly from 
            $\mathsf{P} (A)$.
         \item
            If $\mathsf{P} (A)$ is very small, then one can be practically
            certain that the event $A$ will not occur
            on a single realization of the conditions $\mathfrak{S}$.
       \end{enumerate}
\end{enumerate}

\emph{Empirical Deduction of the Axioms.}
Usually one can assume that the system $\FFF$ of events $A,B,C\dots$ 
that come into consideration and are assigned definite probabilities
form a field that contains $E$ (Axioms I and II
and the first half of Axiom III---the existence of the probabilities).
It is further evident that $0 \le m/n \le 1$ always holds,
so that the second half of Axiom III appears completely natural.
We always have $m=n$ for the event $E$,
so we naturally set $\mathsf{P}(E)=1$ (Axiom IV).
Finally, if $A$ and $B$ are mutually incompatible
(in other words, the sets $A$ and $B$ are disjoint),
then $m=m_1+m_2$, where $m$, $m_1$, and $m_2$ are the numbers
of experiments in which the events $A \cup B$, $A$, and $B$ happen, respectively.
It follows that
$$
  \frac{m}{n} =  \frac{m_1}{n}  +  \frac{m_2}{n}.
$$
So it appears appropriate to set
$\mathsf{P}(A \cup B) = \mathsf{P}(A) + \mathsf{P}(B)$.

Remark I.  If two assertions are both practically certain,
then the assertion that they are simultaneously correct is practically certain,
though with a little lower degree of certainty.
But if the number of assertions is very large,
we cannot draw any conclusion whatsoever
about the correctness of their simultaneous assertion
from the practical certainty of each of them individually.
So it in no way follows
from Principle \textbf{A} that $m/n$ will differ only a little from $\mathsf{P}(A)$ 
in every one of a very large number of series of experiments, 
where each series consists of $n$ experiments.

Remark II.  By our axioms, the impossible event (the empty set) 
has the probability $\mathsf{P}(\emptyset) =0$.   But the converse inference, from
$\mathsf{P}(A) =0$ to the impossibility of $A$, does not by any means follow.
By Principle \textbf{B}, the event $A$'s having probability zero implies only     
that it is practically impossible that it will happen on a particular 
unrepeated realization of the conditions $\mathfrak{S}$.  This by no means implies
that the event $A$ will not appear in the course of a sufficiently long 
series of experiments.  When $\mathsf{P}(A) =0$ and $n$ is very large, 
we can only say, by Principle \textbf{A}, that the quotient $m/n$ 
will be very small---it might, for example, be equal to $1/n$.

\subsubsection{The philosophical synthesis}

The philosophy set out in the two pages we have just translated is a 
synthesis, combining elements of the 
German and French traditions in a way not previously attempted.

By his own testimony, Kolmogorov drew first and foremost from von Mises.
In a footnote, he put the matter this way:
\begin{quotation}
\noindent
   \dots In laying out the assumptions needed to make probability theory 
   applicable to the world of real events, the author has followed in large
   measure the model provided by Mr.~von Mises \dots
\end{quotation}
The very title of this section of the \emph{Grundbegriffe},
``Das Verh\"altnis zur Erfahrungswelt'', echoes the title of the section of
von Mises's 1931 book that Kolmogorov cites, 
``Das Verh\"altnis der Theorie zur Erfahrungswelt''. 
But Kolmogorov does not discuss Kollektivs.
As he explained in his letter to Fr\'echet in 1939 (translated in Appendix \ref{app:kolfr}), 
he thought only a finitary
version of this concept would reflect experience truthfully, and a finitary version,
unlike the infinitary version, could not be 
made mathematically rigorous.  So for mathematics, one should adopt an
axiomatic theory, ``whose practical value can be deduced directly'' from a finitary
concept of Kollektivs.

Although Kollektivs are in the background, Kolmogorov starts in a 
way that echoes Chuprov more than von Mises.
He writes, as Chuprov did (1910, p.~149)\nocite{chuprov:1910}, of
a system of conditions (Komplex von Bedingungen in German,
комплекс условий
in Russian).
Probability is relative to a system of conditions
$\mathfrak{S}$, and yet further conditions must be satisfied in order
for events to be assigned a probability under
$\mathfrak{S}$.  Kolmogorov says nothing more about these conditions,
but we may conjecture that he 
was thinking of the three sources of probabilities mentioned by von Mises:  
chance devices designed for gambling, statistical phenomena, and physical theory.

Like the German objectivists, Kolmogorov did not think that every event has a probability.
Later, in his article on probability in the second edition of the
\emph{Great Soviet Encyclopedia}, published in 1951\nocite{kolmogorov:1951},
he was even more explicit about this:
\begin{quotation}
  Certainly not every event whose occurrence is not uniquely
  determined under given conditions has a definite probability
  under these conditions.  The assumption that a definite 
  probability (i.e. a completely defined fraction of the number
  of occurrences of an event if the conditions are repeated a large
  number of times) in fact exists for a given event under
  given conditions is a \emph{hypothesis} which must be verified
  or justified in each individual case.
\end{quotation}

Where do von Mises's two axioms---probability as a limit of 
relative frequency and its invariance under selection of subsequences---appear
in Kolmogorov's account?  Principle \textbf{A} 
is obviously a finitary version of von Mises's axiom
identifying probability as the limit of relative frequency.  
Principle \textbf{B}, on the other hand, is the
strong form of Cournot's principle (see \S\ref{subsubsec:twoforms} above).
Is it a finitary version of von Mises's principle of invariance under selection?
Perhaps so.
In a Kollektiv, von Mises says, we have no way of singling out an unusual infinite
subsequence.  One finitary version of this is that we have no way of 
singling out an unusual single trial.  So when we do select a single trial
(a single realization of the conditions
$\mathfrak{S}$, as Kolmogorov puts it), we should not expect anything unusual. 
In the special case where the probability is very small, the usual is that
the event will not happen.

Of course, Principle \textbf{B}, like Principle \textbf{A}, is only satisfied
when there is a Kollektiv---i.e., under certain conditions.  
Kolmogorov's insistence on this point 
is confirmed by his comments in 1956\nocite{kolmogorov:1956},
quoted on p.~\pageref{p:weak} above,
on the importance and nontriviality of the step
from ``usually'' to ``in this particular case''.

As Borel and L\'evy had explained so many times,
Principle \textbf{A} can be deduced from Principle \textbf{B}
together with Bernoulli's law of large numbers,
which is a consequence of the axioms
(see our discussion of Castelnuovo and Fr\'echet's use of
this idea on p.~\pageref{p:specialcase}).
But in the framework that Kolmogorov sets up, the deduction requires 
an additional assumption; we must assume that Principle \textbf{B} applies
not only to the probabilities specified for repetitions of conditions $\mathfrak{S}$
but also to the corresponding probabilities (obtaining by assuming independence)
for repetitions of $n$-fold repetitions of $\mathfrak{S}$.
It is not clear that this additional assumption is appropriate,
not only because we might hesitate about independence
(see Shiryaev's comments
in the third Russian edition of the \emph{Grundbegriffe}, p.~120),
but also because the enlargement of our model to $n$-fold repetitions might 
involve a deterioration in its empirical precision, to the extent that we are no longer justified
in treating its high-probability predictions as practically certain.
Perhaps these considerations justify Kolmogorov's presenting 
Principle \textbf{A} as an independent principle alongside 
Principle \textbf{B} rather than as a consequence of it.

Principle \textbf{A} has an independent status in Kolmogorov's story, however, 
even if we do regard it as a consequence of 
Principle \textbf{B} together with Bernoulli's law of large numbers, 
because it comes into play 
at a point that precedes the adoption of the axioms and hence precedes the 
derivation of the law of large numbers:  it is used 
to motivate (empirically deduce) the axioms 
(cf. Bartlett 1949\nocite{bartlett:1949}). 
The parallel to the thinking of  Hadamard and L\'evy is very striking.  In their picture,
the idea of equally likely cases motivated the rules (the 
rules of total and compound probability), while Cournot's principle
linked the resulting theory with reality.  The most important change Kolmogorov makes
in this picture is to replace
equally likely cases with frequency; frequency now motivates the axioms of 
the theory, but Cournot's principle is still the essential 
link with reality.

In spite of the obvious influence of Borel and L\'evy,
Kolmogorov cites only von Mises in this section of the \emph{Grundbegriffe}.
Philosophical works by Borel and L\'evy, along with those by Slutsky
and Cantelli, do appear in the \emph{Grundbegriffe}'s 
bibliography, but their appearance is explained
only by a sentence in the preface:  ``The bibliography gives some
recent works that should be of interest from a foundational viewpoint.'' 
The emphasis on von Mises may have been motivated in part
by political prudence.  Whereas Borel and L\'evy persisted in speaking of the subjective
side of probability, von Mises was an uncompromising frequentist.
Whereas Chuprov and Slutsky worked in economics and statistics, 
von Mises was an applied mathematician, concerned more with
aerodynamics than social science, and the 
relevance of his work on Kollektivs to physics 
had been established in the Soviet literature by Khinchin
in 1929\nocite{khinchin:1929} 
(see also Khinchin, 1961\nocite{khinchin:1961},
and Siegmund-Schultze, 2004\nocite{siegmund-schultze:2004b}).
(For more information on the political context, see
Blum and Mespoulet 2003\nocite{blum/mespoulet:2003};
Lorentz 2002\nocite{lorentz:2002}; Mazliak 2003\nocite{mazliak:2003}; Seneta 2003\nocite{seneta:2003}; 
Vucinich 1999, 2000, 2002\nocite{vucinich:1999,vucinich:2000,vucinich:2002}.)
Certainly, however, Kolmogorov's high regard for von Mises's theory of Kollektivs was sincere.
It emerged again in the 1960s, as a motivation for
Kolmogorov's thinking about the relation between algorithmic complexity and 
probability (Cover et al.\ 1989\nocite{cover/etal:1989},
Vovk and Shafer 2003\nocite{vovk/shafer:2003}).

The evident sincerity of Kolmogorov's frequentism,
and his persistence in it throughout his career,
gives us reason to think twice about the proposition,
which we discussed on p.~\pageref{p:trajectory} that he had in mind,
when he wrote the \emph{Grundbegriffe},
a picture in which a Markov process is represented not by transition probabilities,
as in his 1931 article,
but by a probability space in which one can make probability statements about
trajectories, as in Wiener's work.
The first clause in Kolmogorov's frequentist credo is
that the system of conditions $\mathfrak{S}$ should be capable of unlimited repetition.
If trajectories are very long, then this condition may be satisfied by transition probabilities
(it may be possible to repeatedly start a system in state $x$ and check where
it is after a certain length of time) without being satisfied to the same degree by
trajectories (it may not be practical to run through more than a single whole trajectory).
So transition probabilities might have made more sense to Kolmogorov the frequentist
than probabilities for trajectories.

\section{Reception}\label{sec:reception}

Mythology has it that the impact of the \emph{Grundbegriffe} was immediate and revolutionary.
The truth is more ordinary.
Like every important intellectual contribution, Kolmogorov's axiomatization was absorbed slowly.
Established scholars continued
to explore older approaches and engage in existing controversies,
while younger ones gradually made the new approach the basis for new work.
The pace of this process was influenced not only by the intellectual possibilities
but also by the turmoil of the times.
Every mathematical community in Europe, beginning with Germany's,
was severely damaged by Hitler's pogroms, Stalin's repression,
and the ravages of the second world war.%
\footnote{Many of the mathematicians in our story were direct victims
          of the Nazis.  Halbwachs, Antoni {\L}omnicki, and Saks were murdered.  
          Doeblin and Hausdorff committed suicide to escape the same fate.
          Bavli died at the 
          front.  Banach died from his privations just after the war ended.
          Lo\'eve, arrested and interned at Drancy, was saved only 
          by the liberation of Paris.
          Castelnuovo, L\'evy, Schwartz, and Steinhaus survived in hiding.  
          Others, including {\L}ukasiewicz, Nikodym, Sierpi\'nski, and Van Dantzig, 
          merely lost their jobs during the war. 
          Those who fled Europe included Felix Bernstein,
          Carnap, Einstein, Feller, Hadamard, Popper, Rademacher, Reichenbach,
          von Mises, and Wald.} 

We cannot examine these complexities
in depth, but we will sketch in broad strokes how
the \emph{Grundbegriffe} came to have the iconic status it now enjoys.
In \S\ref{subsec:acceptance} we look at how Kolmogorov's axiomatization became
the accepted framework for mathematical research in probability, and in 
\S\ref{subsec:disappear} we look at how the philosophical interpretation of
the framework evolved.

\subsection{Acceptance of the axioms}\label{subsec:acceptance}

Although the \emph{Grundbegriffe} was praised by reviewers, its abstract measure-theoretic
approach continued to co-exist, during the 1930s, with other approaches.
The controversy concerning von Mises's concept of Kollektivs continued, and although
most mathematicians weighed in against using Kollektivs as the foundation of probability,
Kolmogorov's axioms were not seen as the only alternative.  Only after the second 
world war did the role of the abstract measure-theoretic framework become dominant
in research and advanced teaching.

\subsubsection{Some misleading reminiscences}

Reminiscences of mathematicians active in the 1930s 
sometimes give the impression that the acceptance of Kolmogorov's framework
was immediate and dramatic, 
but these reminiscences must be taken with a grain of salt and a large dose of context.  

Particularly misleading is a 
passage (p.~67--68) in the memoir L\'evy published in 1970\nocite{levy:1970},
when he was in his eighties:
\begin{quotation}
  Starting in 1924, I gradually became accustomed to the idea that one 
  should not consider merely what I called the true probability laws.
  I tried to extend a true law.  I arrived at the idea, arbitrary as it might be, 
  of a law defined in a certain Borel field.  I did not think of saying to myself
  that this was the correct foundation for the probability calculus; I did not have
  the idea of publishing so simple an idea.  Then, one day, I received A.
  Kolmogorov's tract on the foundations of the probability calculus.  
  I realized what a chance I had lost.  But it was too late.  When would I ever
  be able to figure out which of my ideas merited being published?
\end{quotation}
In evaluating these words, one must take account 
of the state of mind L\'evy displays throughout the memoir. 
Nothing was discovered in probability in his lifetime, it seems,
that he had not thought of first.  

In fact, L\'evy was slow, in the 1930s, to mention or use the \emph{Grundbegriffe}.  
Barbut, Locker, and Mazliak (2004\nocite{barbut/locker/mazliak:2004}, pp.~155--156)
point out that he seems to have noticed its zero-one law only in early 1936.
Although he followed Fr\'echet in paying lip service to Kolmogorov's axiomatization
(L\'evy 1937, 1949\nocite{levy:1937,levy:1949})
and ruefully acknowledged the effectiveness of Doob's use of it (L\'evy 1954\nocite{levy:1954}),
he never used it in his own research (Schwartz 2001\nocite{schwartz:1997}, p.~82).

A frequently quoted recollection by Mark Kac about his work
with his teacher Hugo Steinhaus can also give
a misleading impression concerning how quickly the \emph{Grundbegriffe} was accepted:
\begin{quotation}
\noindent
   \dots Our work began at a time when probability theory was emerging
   from a century of neglect and was slowly gaining acceptance as a
   respectable branch of pure mathematics.
   This turnabout came as a result of a book by the great Soviet mathematician A.~N. Kolmogorov
   on foundations of probability theory, published in 1933.
   (Kac 1985\nocite{kac:1985}, pp.~48--49)
\end{quotation}
This passage is followed by a sentence that is quoted less frequently:
``It appeared to us awesomely abstract.''  
On another occasion, Kac wrote,
\begin{quotation}
\noindent
    \dots The appearance in 1933 of the book by Kolmogorov on foundations
    of probability theory was little noticed in Lw\'ow, where I was a student.
    I do remember looking at the book, but I was frightened away by its abstractness,
    and it did not seem to have much to do with what I had been working on. 
    I was wrong, of course \dots  (Kac 1982, p.~62)\nocite{kac:1982}\nocite{gani:1982}
\end{quotation}
It is hard to guess from Kac's recollections what his more mature colleagues, 
Steinhaus and Ulam, thought about the importance of the \emph{Grundbegriffe}
at the time.

\subsubsection{First reactions}\label{subsubsec:firstreactions}

The most enthusiastic early reviews of the \emph{Grundbegriffe} 
appeared in 1934:  one by Willy Feller, of Kiel, 
in the \emph{Zentralblatt}\nocite{feller:1934} and one
by Henry~L. Rietz\nocite{rietz:1934}, 
of Iowa,
in the \emph{Bulletin of the American Mathematical Society}.
Feller wrote:
\begin{quotation}
  The calculus of probabilities is constructed axiomatically, with no gaps  
  and in the greatest generality, and for the first time systematically integrated,
  fully and naturally, with abstract measure theory.  The axiom system is certainly
  the simplest imaginable\dots.  The great generality is noteworthy; probabilities
  in infinite dimensional spaces of arbitrary cardinality are dealt with\dots.
  The presentation 
  is very precise, but rather terse, directed to the reader who is not
  unfamiliar with the material.  Measure theory is assumed.
\end{quotation}
Rietz was more cautious, writing that 
``This little book seems to the reviewer to be an 
important contribution directed towards securing a more logical development 
of probability theory.''  

Other early reviews, by the German mathematicians Karl D\"orge (1933)\nocite{dorge:1933} and 
Gustav Doetsch (1935)\nocite{doetsch:1935} were more neutral.
Both D\"orge and Doetsch contrasted Kolmogorov's approach with Tornier's,
which stayed closer to the motivating idea of relative frequency
and so did not rely on the axiom of countable additivity.

Early authors to make explicit use of Kolmogorov's framework included 
the Polish mathematicians Zbigniew {\L}omnicki and Stanis{\l}aw Ulam, 
the Austrian-American mathematician Eberhard Hopf,
and the American mathematician Joseph L. Doob.
We have already mentioned
{\L}omnicki and Ulam's 1934 article on product measures\nocite{lomnicki/ulam:1934}.
Hopf used Kolmogorov's framework in his work on the method of arbitrary functions,
published in English in 1934\nocite{hopf:1934}.
Doob, in two articles in 1934, used the framework in an attempt to make
Ronald Fisher and Harold Hotelling's results
on maximum likelihood\nocite{doob:1934a,doob:1934b} rigorous.

The French probabilists did not take notice so quickly.  
The first acknowledgement may have been in comments by Fr\'echet
in a talk at the International Congress of Mathematicians in Oslo in 1936:
\begin{quotation}
  The foundations of probability theory have changed little.
  But they have been enriched by particulars about the additivity
  of the probabilities of denumerable sets of incompatible events.
  This revision was completely and systematically exposited
  for the first time by A.~Kolmogorov.
  But on this point we mention, on the one hand,
  that one would find the elements of such an exposition
  scattered about in the previous publications of many of the authors who have written
  about the calculus of probabilities in recent years.
  And on the other hand, if this revision was needed,
  it was because of the revolution brought about by \'Emile Borel,
  first in analysis by his notion of measure,
  and then in probability theory by his theory of denumerable probabilities.%
\footnote{These comments are on pp.~153--154 of the 
          \emph{Les math\'ematiques et le concret}, a collection of essays
          that Fr\'echet published in 1955.  The essay in which
          they appear is not in the proceedings of the Oslo meeting, but
          it is identified in the 1955 collection as having been presented
          at that meeting.}
\end{quotation}

\subsubsection{The situation in 1937}

In the years immediately following the publication of the \emph{Grundbegriffe},
the burning issue in the foundation of the probability was the 
soundness and value of von Mises's concept of a Kollektiv.  
Although von Mises had been forced to leave Berlin for Istanbul when 
Hitler took power in 1933, he had continued to argue that any satisfactory
mathematical foundation for probability would have to begin with Kollektivs.  In 1936, 
in the second edition of his \emph{Wahrscheinlichkeit, Statistik und Wahrheit},
he specifically criticized Kolmogorov's way of avoiding them (pp.~124ff).

The rejection of von Mises's viewpoint by most mathematicians, 
especially the French, became clear 
at the colloquium on mathematical probability held 
in Geneva in October 1937 (Wavre 1938--1939\nocite{wavre:1938}).%
\footnote{The colloquium was very broad, both in topics covered and 
          in nationalities represented.  But its location dictated 
          a French point of view.  Borel was the honorary chair.  
          Most of the written contributions, including those by Feller
          and Neyman, were in French.} 
Wald presented his demonstration of the existence of Kollektivs at this
colloquium, and it was accepted.  But the mathematicians
speaking against Kollektivs as a foundation for probability
included Cram\'er, de Finetti, Feller, Fr\'echet, and L\'evy.
No one, it seems, rose to defend von Mises's viewpoint.%
\footnote{Neither von Mises nor Kolmogorov were able to attend the Geneva meeting,
          but both made written contributions,
          including, in von Mises's case, comments on the foundational
          issues.}
Fr\'echet summed up, in a contribution rewritten after the colloquium,
with the words we quoted at the beginning of this article;
the classical axioms, updated by including countable additivity,
remained the best starting point for probability
theory, and Kolmogorov had provided the best exposition.

Willy Feller, now working with Harald Cram\'er in Stockholm, 
was the most enthusiastic adherent of
the \emph{Grundbegriffe} at Geneva.  Nearly 30 years
younger than Fr\'echet, even a few years younger than Kolmogorov himself,
Feller praised the \emph{Grundbegriffe} without stopping to discuss
how much of it had already been formulated by earlier authors.  
Kolmogorov's well known axiomatization was, Feller declared, 
the point of departure for most modern theoretical research in probability
(Feller 1938, p.~8).%
\footnote{Feller's contribution to the Geneva meeting\nocite{feller:1938}, 
          was mainly devoted to reconciling Tornier's theory,
          which followed von Mises in defining probability as a limiting
          relative frequency, with Kolmogorov's axiomatization.
          Feller had collaborated in Kiel with Tornier, who became
          an ardent Nazi (Segal 2003\nocite{segal:2003}, p.~150).
          According to Abraham Frankel (1967\nocite{fraenkel:1967}, p.~155), 
          Tornier forced Feller out of his job by revealing Feller's
          Jewish ancestry.  Doob (1972\nocite{doob:1972}), on the other hand, 
          reports that Feller lost his job by refusing to sign a loyalty oath.
          Feller spent about five
          years with Cram\'er in Stockholm before emigrating to the United
          States in 1939.}
In Feller's view, Kolmogorov had demonstrated the sufficiency of additivity
for probability theory, something that had been doubted not so long before.
General opinion held, Feller said, that Kolmogorov's axiomatization
covered too vast a field.  But could anyone demonstrate the existence of a relation 
between probabilities that the axiomatization did not 
take into account (p.~13)?

Others at the colloquium paid less attention to Kolmogorov's axiomatization.
Jerzy Neyman, for example, prefaced his contribution,
which was concerned with statistical estimation,
with an exposition of the modernized classical theory
that did not mention Kolmogorov\nocite{neyman:1938,wavre:1938}.
He provided his own resolution of the problem of defining
$\mathrm{P}(\mathrm{A}\given\mathrm{B})$\label{p:neyman}
when $\mathrm{B}$ has probability zero.
In a similar article read to the Royal Society in 1937 \citep{neyman:1937},
Neyman did mention the \emph{Grundbegriffe},
calling it
``a systematic outline of the theory of probability based on that of measure'',
but also cited Borel, L\'evy, and Fr\'echet, as well as {\L}omnicki and Ulam.
Rejection of Kollektivs did not, evidently, necessarily mean a thorough
acceptance of Kolmogorov's framework.
Many who rejected Kollektivs could agree
that Kolmogorov had modernized the classical theory
by showing how it becomes abstract measure theory
when countable additivity is added.
But no real use was yet being made of this level of abstraction.

The Italians, for whom Cantelli represented
the modernized classical alternative to von Mises,
paid even less attention to the \emph{Grundbegriffe}.
Throughout the 1930s,
Cantelli debated von Mises, developed his own abstract theory,
and engaged other Italian mathematicians on the issues it raised,
all without citing Kolmogorov
(Regazzini 1987b\nocite{regazzini:1987b}, pp.~15--16).

Especially striking is the absence of mention, even by enthusiasts
like Feller, of Kolmogorov's treatment of conditional probability.
No one, it appears, knew what to do with this idea in 1937.
Neyman did something simpler, and the French generally preferred to 
stick with the classical approach, taking the rule of compound
probability as an axiom (L\'evy 1937\nocite{levy:1937}, Ville 1939\nocite{ville:1939}).
Paul Halmos recalls that Kolmogorov's approach to conditional
probability was puzzling and unconvincing to him
(Halmos 1985\nocite{halmos:1985}, p.~65), and this was probably
the dominant attitude, even among the most able mathematicians.

\subsubsection{Triumph in the textbooks}

The first textbook author to sign on to Kolmogorov's viewpoint
was the Swedish statistician Harald Cram\'er, who was ten years older than
Kolmogorov.  Cram\'er explicitly subscribed to Kolmogorov's system in 
his \emph{Random Variables and 
Probability Distributions}\nocite{cramer:1937}, published in 1937, 
and his \emph{Mathematical Methods of Statistics}\nocite{cramer:1946}, 
written during the second world war and published in 1946.  
In these books, we find the now standard list of axioms for measure, 
expressed in terms of a
set function $P(S)$, and this can be contrasted with the old fashioned 
discussion of the classical rules for probabilities $\mathrm{Pr.}(E)$ that we still find 
in the second edition (1950) of Fr\'echet's treatise
and in the textbook\nocite{fortet:1950}, also published in 1950,
of his student Robert Fortet.
In substance, however, 
Cram\'er made no use of the novelties in the \emph{Grundbegriffe}.  
As a statistician, he had no need
for the strong laws, and so his books made no use of probability distributions 
on infinite-dimensional spaces.  He also had no need for Kolmogorov's treatment
of conditional probability; in 1937 he did not even mention conditional probability,
and in 1946 he treated it in an old-fashioned way.

In the Soviet Union, it was Boris Gnedenko, who became Kolmogorov's
student in 1934, who brought the spirit of the \emph{Grundbegriffe} to
an introductory textbook in probability and statistics
for advanced students.  Gnedenko's\nocite{gnedenko:1950}
Курс теории вероятностей,
published in 1950, treated the strong laws as well as statistics and
followed the \emph{Grundbegriffe} closer than Cram\'er had.
With the translation of the \emph{Grundbegriffe} itself into English in 1950,
the appearance of Lo\`eve's \emph{Probability Theory}\nocite{loeve:1955} in 1955, and 
the translation of Gnedenko's second edition into English in 1962,
Kolmogorov's axioms became ubiquitous in the advanced teaching of probability
in English.

\subsubsection{Making it work}

The irresistibility of an axiomatic presentation, even a superficial one,
for the advanced teaching of probability
was part and parcel of the continuing advance of the spirit of 
Hilbert, which called for the axiomatization of all mathematics
on the basis of set theory.  Driven by this same spirit, a younger generation
of mathematicians undertook, in the late 1930s and on through the war years,
to correct the shortcomings of Kolmogorov's framework.  

Bru (2002, pp.~25--26)\nocite{bru:2002,brissaud:2002} 
recounts the development of this enterprise in the Borel Seminar
in Paris in the late 1930s.  Though officially run 
by Borel, who retained his chair until 1941, the seminar was actually
initiated in 1937 by Ville, Borel's assistant, in the image of 
Karl Menger's seminar in Vienna, where
Ville had developed his ideas on martingales.
The initial theme of the seminar was the same as Menger's theme when
Ville had visited there in 1934--1935: the foundations of the calculus of probabilities.  
It was the younger people in the seminar, especially Ville, 
Wolfgang Doeblin, Robert Fortet, and Michel Lo\`eve, who argued, 
against Borel and L\'evy's resistance, for 
Kolmogorov's axiomatization.  In 1938, the Borel Seminar turned to the
fundamental mathematical issue, which also occupied Doob and
the Russian school:  how to extend the axiomatic theory
to accommodate continuous-time processes.

The French work on Kolmogorov's framework 
continued through and after the war, under the leadership
of Fortet, who was handicapped, 
as Bru explains, by the disdain of Bourbaki:
\begin{quotation}
\noindent
   \dots Bourbaki having decreed that the measurable structure---neither
   sufficiently algebraic nor sufficiently topological---was unacceptable,
   Kolmogorov's axiomatization was hardly taken seriously,
   and one knows how damaging this was for the development of 
   probability theory in France in the fifties and sixties.  Fortet
   felt a certain bitterness about this, though he formulated it
   discreetly and always with humor.
\end{quotation}
The approach favored by Bourbaki can be seen in the treatise
\emph{Radon Measures on Arbitrary Spaces and Cylindrical Measures}\nocite{schwartz:1973},
published in 1973 by Laurent Schwartz,
L\'evy's son-in-law and one of the members
of the Bourbaki group most engaged with probability 
(see also Weil 1940\nocite{weil:1940}, Schwartz 1997\nocite{schwartz:1997}).

It was Doob who finally provided the 
definitive treatment of stochastic processes within the
measure-theoretic framework, in his 
\emph{Stochastic Processes} (1953)\nocite{doob:1953}.  
Not only did Doob show how to construct satisfactory probability spaces
for stochastic processes; he also enriched the abstract framework with 
the concept of a filtration and with a measure-theoretic version of Ville's 
concept of a martingale, and he gave conditions for the existence of
regular conditional distributions.  It would go far beyond the scope of 
this article to analyze Doob's achievement.  Suffice it to say that although
Bourbaki was never persuaded, leaving probability still outside the 
mainstream of pure mathematics, 
L\'evy eventually acknowledged defeat
(in the second edition of his
\emph{Th\'eorie de l'addition des variables al\'eatoires}, 1954)\nocite{levy:1937}, 
and Doob's elaboration of Kolmogorov's framework
became the generally accepted mathematical foundation for probability theory.

We do not know what Kolmogorov thought of Doob's innovations.
In the preface to the second Russian edition of the \emph{Grundbegriffe} (1974),
he acknowledged Doob's innovation with a rather neutral comment:
``Nowadays people prefer to define conditional probability
with respect to an arbitrary $\sigma$-algebra $\mathcal{F}^{\prime}\subseteq\mathcal{F}$''\dots.

\subsection{The evolution of the philosophy of probability}\label{subsec:disappear}

Whereas Kolmogorov's axiomatization became standard,
his philosophy of probability showed little staying power.
We do find a version of his philosophy in a textbook by Cram\'er.
We also find Cournot's principle, in forms advocated by Borel, L\'evy, and Chuprov,
flourishing into the 1950s.
But aside from Cram\'er's imitation,
we find hardly any mention of Kolmogorov's philosophy in the subsequent literature,
and even Cournot's principle faded after the 1950s.

\subsubsection{Cram\'er's adaptation of Kolmogorov's philosophy}\label{subsubsec:risefall}

Harald Cram\'er, who felt fully in tune with Kolmogorov's frequentism,
repeated the key elements of his philosophy in his 1946 book (pp.~148--150).
Cram\'er expressed Kolmogorov's caution that 
the theory of probability applies only under certain conditions by saying that only
certain experiments are random.  In the context of 
a random experiment $\mathfrak{E}$, Cram\'er stated Kolmogorov's
Principle \textbf{A} in this way:
\begin{quotation}
   \emph{Whenever we say that the probability of an event $E$ with respect
   to an experiment $\mathfrak{E}$ is equal to $P$, the concrete meaning
   of this assertion will thus simply be the following:  In a long series
   of repetitions of $\mathfrak{E}$, it is practically certain that the
   frequency of $E$ will be approximately equal to $P$. --- This statement
   will be referred to as the frequency interpretation of the probability $P$.}
\end{quotation}
He stated Kolmogorov's Principle \textbf{B} as a principle applying
to an event whose probability is very small or zero:
\begin{quotation}
   \emph{If $E$ is an event of this type, and if the experiment $\mathfrak{E}$
   is performed one single time, it can thus be considered as practically 
   certain that $E$ will not occur.} --- This particular case of the frequency
   interpretation of a probability will often be applied in the sequel.
\end{quotation}
The final sentence of this passage shows Cram\'er to be a less careful
philosopher than Kolmogorov, for it suggests that 
Principle \textbf{B} is a particular case of Principle \textbf{A}, 
and this is not strictly true.  As we noted 
when discussing Castelnuovo's views in \S\ref{subsubsec:twoforms}, 
the weak form of Cournot's principle is indeed a special case of Principle \textbf{A}.  
But Principle \textbf{B} is the strong form of Cournot's principle, and this is not
merely a special case of Principle \textbf{A}.   

Cram\'er's textbook was influential across the world; it was even translated into 
Russian.  But its influence was not enough to keep Kolmogorov's 
Principles \textbf{A} and \textbf{B} alive.
Gnedenko's textbook emphasized that probability is relative
to a certain complex of conditions $\mathfrak{S}$, but it did not mention the two 
principles.  Aside from Cram\'er's book, 
we do not know of any other textbook or philosophical discussion where
Kolmogorov's formulation of these principles is imitated or discussed.

This disappearance of interest in Kolmogorov's philosophical formulation among
mathematicians working in probability theory 
may be due in part to a lack of interest in philosophical issues altogether.
Doob, for example, explained in a debate
with von Mises in 1941\nocite{doob:1941} that he rejected all philosophical
explication; the meaning of a probability statement should
be left to those who actually use probability in the real world.
It is also of significance, however, that Kolmogorov's philosophical formulation 
quickly encounters difficulties once we follow Doob in taking the
probability space $\Omega$ to be the set of trajectories for a 
stochastic process.  In many applications, contrary to the assumption with
which Kolmogorov begins, the experiment that produces the
trajectory may not be capable of repetition.

\subsubsection{Cournot's principle after the \emph{Grundbegriffe}}\label{subsubsec:cournotafter}

As we explained in \S\ref{subsubsec:french}, Borel and L\'evy had a simpler philosophy 
of probability than
Kolmogorov's; they used only Principle \textbf{B}, Cournot's principle.
Their viewpoint flourished in the first couple of decades after the 
\emph{Grundbegriffe}.

Borel, probably the most influential advocate of the principle,
sharpened his statement of it in the 1940s.
In earlier years, he wrote frequently about the practical meaning of probabilities
very close to zero or one,
but it is hard to discern in these writings the philosophical principle,
which we do find in Hadamard and L\'evy,
that interpreting a very small probability as impossibility
is the unique way of bringing probability theory into contact with the real world.
But in the 1940s, we find the principle articulated very clearly.
In his 1941 book,
\emph{Le jeu, la chance et les th\'eories scientifiques modernes}\nocite{borel:1941},
he calls it the ``fundamental law of chance'' (la loi fondamentale du hasard).
Then, in 1943, on the first page of the text of his ``Que sais-je?''\ volume,
\emph{Les probabilit\'es et la vie}\nocite{borel:1943},
he finally coined the name he used thereafter:
``the only law of chance'' (la loi unique du hasard).
This name appears again in the 1948 edition of \emph{Le Hasard}
and the 1950 edition of \emph{\'El\'ements de la th\'eorie des probabilit\'es}
(see also Borel 1950\nocite{borel:1950}).
It was also popularized by Robert Fortet, in his essay in Fran\c{c}ois Le Lionnais's 
\emph{Les grands courants de la pens\'ee math\'ematique}\nocite{lionnais:1948},
first published in 1948\nocite{fortet:1948,lionnais:1948}.

Borel's books from this period, as well as Le Lionnais's,
were translated into English.
They may have influenced Marshall Stone,
one of the few American mathematicians to echo the idea that probability theory
makes contact with reality only by the prediction that events of small probability
will not happen (Stone 1957)\nocite{stone:1957}.
Per Martin-L\"of has told us that he also learned the idea from reading Borel
(see Martin-L\"of 1969\nocite{martinlof:1969a}, p.~616;
1966--1967, p.~9\nocite{martinlof:1966};
1969--1970, p.~8\nocite{martinlof:1969}).

Borel never used the name ``Cournot's principle''; nor did L\'evy or Kolmogorov.
The name seems to be due to Fr\'echet,
but the reference to Cournot can be traced back to Chuprov,
who considered Cournot the founder of the philosophy of modern statistics
(Sheynin 1996\nocite{sheynin:1996}, p.~86).
As we mentioned earlier, Chuprov called the principle Cournot's ``lemma''.
This enthusiasm for Cournot and the principle was brought from Russian 
into German by Chuprov's student, Oskar Anderson, who spent the 1930s in Sofia and 
then moved to Munich in 1942.
Anderson called the principle the ``Cournotsche Lemma''
or the ``Cournotsche Br\"ucke''---a bridge between mathematics and the
world.  We find both phrases already in Anderson's 1935 book\nocite{anderson:1935,anderson:1954},
but the book may have been less influential than an article Anderson contributed to 
a special issue of the Swiss philosophy journal
\emph{Dialectica}\nocite{anderson:1949,borel:1949,levy:1949}
in 1949, alongside articles by Borel and L\'evy revisiting their versions of
Cournot's principle.  Fr\'echet took these
articles as one of his themes in his presiding 
report at the session on probability at the 
Congr\`es International de Philosophie des Sciences\nocite{bayer:1951} at Paris that same year
(Fr\'echet 1951\nocite{frechet:1951}), 
where he accepted the attribution of the principle to Cournot (``bien qu'il
semble avoir \'et\'e d\'ej\`a plus ou moins nettement indiqu\'e par d'Alembert'')
but suggested the appellation ``principe de Cournot'', 
reserving ``lemma'' as a purely mathematical term. 
It was normal for Fr\'echet to legislate on terminology; 
from 1944 to 1948 he had
led the effort by the Association Fran\c{c}aise de Normalisation
to standardize probability terminology and notation, 
putting in place appellations such as Borel-Cantelli, Kolmogorov-Smirnov, 
etc.\ (Bru 2003b\nocite{bru:2003b}, Pallez 1949\nocite{pallez:1949}).
Fr\'echet had second thoughts about giving so much credit to
Cournot; when he reprinted his 1949 report as a chapter in a book in
1955\nocite{frechet:1955}, he replaced
``principe de Cournot'' with ``principe de Buffon-Cournot''.  
But here no one else appears to have followed his example.  

Both Anderson and the 
Dutch mathematical statistician David Van Dantzig argued for using
Cournot's principle as the foundation for statistical testing:
Anderson in \emph{Dialectica} (Anderson 1949\nocite{anderson:1949}), 
and Van Dantzig at the meeting in Paris
(Van Dantzig 1951\nocite{vandantzig:1951}).
Neyman found this view of statistical testing incomprehensible; at the same meeting in Paris he 
said Anderson was the ``only contemporary author I know who seems 
to believe that the inversion of the
theorem of Bernoulli is possible'' (Neyman 1951, p.~90)\nocite{neyman:1951}.
The German mathematical statistician Hans Richter, also in Munich,
emphasized Cournot's principle 
in his own contributions to \emph{Dialectica} 
(Richter 1954; von Hirsch 1954)\nocite{richter:1954,vonhirsch:1954}
and in his probability textbook (Richter 1956)\nocite{richter:1956}, which
helped bring Kolmogorov's axioms to students in postwar Germany.
As a result of Richter's book, the name ``Cournotsche Prinzip'' is fairly
widely known among probabilists in Germany.

The name ``Cournot's principle'' was brought into English
by Bruno de Finetti (1951)\nocite{definetti:1951}, 
who ridiculed it in several different languages as 
``the so-called principle of Cournot''.  De Finetti participated in
the 1949 Paris conference, and he presumably remembered that
Fr\'echet had coined the phrase.  His disdain for the name
may have been influenced, however, by the similarity between the
names Cournot and Carnot.  When mentioning the
principle in his 1970 book\nocite{definetti:1970} (p.~181), 
de Finetti cites a note, dating back to 1930\nocite{borel:1930}, 
in which Borel mentions ``Carnot's principle''---the second law
of thermodynamics---as an example where a small probability should be
interpreted as an impossibility.  In spite of his tone, de Finetti
has his own way of making sense of the principle.
As he explained in a note written in 1951\nocite{definetti:1955} for 
Fr\'echet's \emph{Les math\'ematiques et le concret}, 
he did not really disagree with the statement that one should act as if an 
event with a very small probability should not happen.  Rather he took the
principle as a tautology, a consequence of the subjective definition of probability, not 
a principle standing outside probability theory and relating it to the real
world (de Finetti 1955, p.~235; see also Dawid 2004\nocite{Dawid:2004}).

\subsubsection{The new philosophy of probability}

Why did no one discuss Kolmogorov's philosophy after Cram\'er, and why did 
discussion even of Cournot's principle fade after the 1950s?

Those who do not consider Cournot's principle useful or coherent may consider
this explanation enough for its fading, but there are also sociological 
explanations.  The discussions of Cournot's principle that we have 
just reviewed were discussions by mathematicians.  But since the publication
of the \emph{Grundbegriffe}, philosophy has been more and more a matter
for professionals.  As we noted in \S\ref{subsubsec:britger}, this
was already true in Germany before the \emph{Grundbegriffe}.
German and Austrian philosophers had worked on probability for 
over 50 years, from the 1880s to the 1930s; they had endlessly debated
subjective and objective interpretations, the meaning of possibility, and the relation of 
probability to logic and to induction.  This tradition was rich 
and autonomous; it had left the mathematicians and their preoccupations 
behind.  In particular, since von Kries it had had little truck
with Cournot's principle.

Because of its mathematical austerity, the \emph{Grundbegriffe} was scarcely
noticed by philosophers when it appeared. 
According to Ville (1985)\nocite{ville:1985}, it had
received little attention 
at the Menger seminar in Vienna, the mathematical satellite of the
Vienna circle, where Popper and other philosophers rubbed shoulders with 
mathematicians such as 
Carath\'eodory, G\"odel, von Mises, and Wald (Hacohen 2000\nocite{hacohen:2000}).

When Kolmogorov finally is mentioned in the philosophical literature, 
in Ernest Nagel's \emph{Principles of the Theory of Probability} in 
1939\nocite{nagel:1939}, he is merely listed as one of many
mathematicians who have given axioms for probability.  No mention is made of the
philosophy he spelled out in the \emph{Grundbegriffe}.
For Nagel, Kolmogorov's work was pure mathematics, and it was his job to propose different 
interpretations of it---frequentist,
subjectivist, or whatever (Nagel 1939, pp.~40--41):
\begin{quotation}
\noindent
  \dots Abstract sets of postulates for probability have been given by
  Borel, Cantelli, Kolmogoroff, Popper, Reichenbach, and several other writers. \dots
  From an abstract mathematical point of view, the probability calculus is a chapter 
  in the general theory of measurable functions, so that the mathematical theory
  of probability is intimately allied with abstract point-set theory.  This aspect
  of the subject is under active investigation and has been especially cultivated
  by a large number of French, Italian, Polish, and Russian mathematicians. 
\end{quotation}
Nagel then listed nine possible interpretations.
Kolmogorov would have agreed, of course, that his axioms have many interpretations.
He said this on p.~1 of the \emph{Grundbegriffe}:
\begin{quotation}
  As we know, every axiomatic (abstract) theory can be interpreted in an unlimited
  number of ways.  The mathematical theory of probability, in particular, has numerous 
  interpretations in addition to those from which it grew.  Thus we find 
  applications of the mathematical theory of probability to areas of research that 
  have nothing to do with the ideas of chance and probability in their concrete senses.%
\footnote{The next sentence states that an appendix is devoted to such applications,
          but there is no such appendix, and the sentence was omitted from the 1936 translation
          into Russian and subsequently from the 1950 translation into English.}
 \end{quotation}
But for Kolmogorov, the concrete sense of chance and probability was frequentist.  
The other applications were about something else
(areas or volumes, for example), not probability.

The philosophers of probability most prominent in the United States
after the second world war, Carnap and Reichenbach,
were themselves refugees from Germany and Austria,
even more steeped in the German tradition.
They believed that they had new things to say,
but this was because they had new philosophical tools,
not because they had new things to learn from mathematicians.  
As a consequence, the philosophical tradition lost what the mathematicians of 
Kolmogorov's generation had learned about how to make sense of frequentism.  
They lost Borel's, L\'evy's
and Kolmogorov's understanding of Cournot's principle, and they also missed
Ville's insight into the concept of a Kollektiv.  

The one post-war philosopher who might have been expected to make 
Cournot's principle central to his philosophy of probability was Karl Popper.
Popper taught that in general scientific theories 
make contact with reality by providing opportunities for falsification.
Cournot's principle tells us how to find such opportunities with
a probabilistic theory:
single out an event of very small probability and see if it happens.

Popper is in fact one of the few English-speaking post-war philosophers who discussed 
Cournot's principle, and on the whole he seems to have subscribed to it.
This already seems clear, for example, in 
\S68 of his celebrated \emph{Logik der Forschung}\nocite{popper:1935} (1935). 
The picture is muddied, however, by Popper's own youthful and enduring ambition to 
make his mark by axiomatizing probability.  He had tried his hand at this
in 1938, before he was even aware of Kolmogorov's work
(Popper 1938)\nocite{popper:1938}.
By 1959, when he published the expanded English edition of \emph{Logik der Forschung},
under the title \emph{The Logic of Scientific Discovery},
he knew the importance that Kolmogorov's axioms had assumed,
and he included an extensive discussion of the axioms themselves,
finding reasons to disagree with them and propose alternatives.
This makes it difficult to keep track of the issue of interpretation
(see, however, the footnote on p.~191).  When he returns to the issue
in \emph{Realism and the Aim of Science}\nocite{popper:1983} in 1983,
he speaks explicitly about Cournot's principle,
but again his views are very complex (see especially p.~379).

\subsubsection{Kolmogorov's reticence}

The scant attention paid to Kolmogorov's own philosophy
is due in part to his own reticence.
Given his great prestige, people would have paid attention to
if he had written about it at length, but he did not do so.
In the Soviet Union, it would have been very dangerous to do so.

Kolmogorov's discussion of theoretical statistics at the Tashkent conference in 1948
(Kolmogorov 1948b\nocite{kolmogorov:1948b})
shows that he was interested in the controversies about
Bayesian, fiducial, and confidence methods
and that he wanted to develop a predictive approach.
He was careful, however, to treat statistical theory as a branch of mathematics
rather than a philosophical enterprise using mathematics as a tool,
and this did not permit him to engage in dialogue with the schools of thought that
were debating each other in the West.

Stalin's death led to only slightly greater breathing space for mathematical statistics.
The crucial step was a conference held jointly in 1954 by the Academy
of Sciences, the Central Statistical Bureau, and the Ministry of Higher Education.
At this meeting, Kolmogorov was still careful to criticize the notion that there
is a universal general theory of statistics; all that is needed, he said,
is mathematical statistics, together with some techniques for collecting and processing data.
But there are some problems, he said, in domains such as
physics, actuarial work, and sampling,
in which probability can legitimately be applied to statistics
(Kolmogorov 1954\nocite{kolmogorov:1954,anonymous:1954};
Kotz 1965; Lorentz 2002\nocite{kotz:1965,lorentz:2002};
Sheynin 1996, p.~102\nocite{sheynin:1996}).
This formulation permitted him to publish a still very concise exposition
of his frequentist viewpoint, which appeared in English,
as a chapter in \emph{Mathematics, Its Content, Methods, and Value} 
(1956)\nocite{kolmogorov:1956,aleksandrov/etal:1956}.

In the 1960s, when Kolmogorov did return to the foundations of probability
(Kolmogorov 1963\nocite{kolmogorov:1963}, 1965\nocite{kolmogorov:1965},
1968\nocite{kolmogorov:1968}),
he still emphasized topics far from politics and philosophy:  
he was applying information theory to linguistics and the analysis of literature.
The definition of randomness he produced used new ideas concerning complexity
but was inspired by the idea of a finitary Kollektiv,
which he had previously thought impossible to treat mathematically
(Li and Vit\'anyi 1997, 
Vovk and Shafer 2003)\nocite{li/vitanyi:1997,vovk/shafer:2003}.
It could be said that the new definition brought Cournot's principle
back in a new and perhaps more convincing form.
Instead of ruling out events of small probability 
just because they are singled out in advance,
we rule out sets of outcomes that have simple enough a description
that we can in fact single them out in advance.
But Kolmogorov said nothing this philosophical.
Vladimir Arnol'd tells us that like most people who had lived through Stalin's terror,
Kolmogorov was afraid of the authorities until his last day
(Arnol'd 1993, p.~92)\nocite{arnold:1993,shiryaev:1993,zdravkovska/duren:1993}.

Kolmogorov's reticence complemented Doob's explicit rejection of
philosophy, helping to paint for many 
in the West a picture of mathematical statistics free from 
philosophy's perils and
confusions.
In fact he did have a philosophy of probability---the one he presented in the 
\emph{Grundbegriffe}.

\section{Conclusion}

The great achievement of the \emph{Grundbegriffe} was to seize 
the notion of probability from the classical probabilists.
In doing so, Kolmogorov made space for a mathematical theory that
has flourished for seven decades, driving
ever more sophisticated statistical methodology and
feeding back into other areas of pure mathematics.
He also made space, within the Soviet Union, for a theory of 
mathematical statistics that could prosper without entanglement
in the vagaries of political or philosophical controversy.

Kolmogorov's way of relating his axioms to the world has received much 
less attention than the axioms themselves.  But Cournot's principles
has re-emerged in new guises, as the information-theoretic and 
game-theoretic aspects of probability have come into their own.
We will be able to see this evolution as growth---rather than mere
decay or foolishness---only if we can see the \emph{Grundbegriffe}
as a product of its own time.

\appendix

\section{Appendix}

\subsection{Ulam's contribution to the Z\"urich Congress}\label{app:ulam}

Stanis{\l}aw Ulam's contribution to the International Congress of Mathematicians
in Z\"urich (Ulam 1932)\nocite{ulam:1932}, was entitled 
``Zum Massbegriffe in Produktr\"aumen''.
We here provide a translation into English.
We have updated the notation slightly,
using $\{x\}$ instead of $(x)$ for the set containing a single element $x$,
$\cup$ instead of $\sum$ for set union,
$\emptyset$ instead of $0$ for the empty set,
and $x_1,x_2,\dots$ instead of $\{x_i\}$ for a sequence.

The article by {\L}omnicki and Ulam appeared in 1934,
in Volume 23 of \emph{Fundamenta Mathematicae}
rather than in Volume 20 as they had hoped.

\begin{center}
\textbf{On the idea of measure in product spaces}

By St. Ulam, Lw\'ow
\end{center}

The content of the report consists of the presentation of some results that 
were obtained in collaboration with Mr.~Z.~Lomnicki.

Let $X,Y$ be two abstract spaces,
on which a measure $m(M)$ is defined for certain sets $M$ (the measurable sets).
No topological or group-theoretic assumption is made
about the nature of the spaces.
Now a measure will be introduced on the space $X \times Y$
(i.e., the space of all ordered pairs $(x,y)$,
where $x \in X$ and $y \in Y$).
This corresponds to problems in probability,
which are concerned with finding probabilities of combined events
from probabilities of individual (independent) events.

We make the following assumptions about the measure on $X$ (and analogously on $Y$)
and the class $\mathfrak{M}$ of measurable sets there:
\begin{enumerate}\renewcommand{\labelenumi}{\Roman{enumi}.}
   \item
     The whole space $X$ is measurable: $X\in \mathfrak{M}$.
     Also the singleton set $\{x\}$ is in $\mathfrak{M}$ for each $\{x\} \in X$.
   \item
     From $M_i \in \mathfrak{M}$ for $i=1,2,\dots$,
     it follows that $\cup_{i=1}^{\infty} M_i \in \mathfrak{M}$.
   \item
     From $M,N \in \mathfrak{M}$, it follows that $M \setminus N \in \mathfrak{M}$.
   \item
     If $M \in \mathfrak{M}$, $m(M) = 0$, and $N \subset M$,
     then $N$ is also in $\mathfrak{M}$.
\end{enumerate}
\begin{enumerate}
   \item
     $m(X) = 1$; $m(M) \ge 0$.
   \item
     $m(\cup_{i=1}^{\infty} M_i) = \sum_{i=1}^{\infty} m(M_i)$
     when $M_i \cap M_j = \emptyset$ for $i \ne j$.
   \item
     From $m(M) = 0$ and $N \subset M$, it follows that $m(N) = 0$.\footnote%
       {This condition follows from the previous one and was omitted from the 1934 article.}
\end{enumerate}

With these assumptions, we obtain

\textit{Theorem.}  One can introduce a measure on the space $X \times Y$ such that 
all sets of the form $M \times N$ are measurable and have the measure $m(M) \cdot m(N)$
(here $M$ and $N$ denote measurable subsets of $X$ and $Y$, respectively) and 
all our assumptions remain satisfied.

An analogous theorem holds for a countable product---i.e., a set
$$
  \prod X_i = X_1 \times X_2 \times \cdots
$$
consisting of all sequences $x_1,x_2,\dots$, where $x_i \in X_i$.

Some of the properties of these measures will be studied
in connection with the formulation of questions in probability theory.
A detailed presentation will appear shortly in \emph{Fundamenta Mathematicae}
(presumably Vol.~20).

\subsection{A letter from Kolmogorov to Fr\'echet}\label{app:kolfr}

The Fr\'echet papers in the archives of the Academy of Sciences in Paris include a letter
in French to Fr\'echet, in which Kolmogorov elaborates briefly on his philosophy of probability.
Here is a translation.

\begin{center}
  Moscow 6, Staropimenovsky per.\ 8, flat 5

3 August 1939
\end{center}

\noindent
Dear Mr.~Fr\'echet,

I thank you sincerely for sending the proceedings of the Geneva Colloquium,
which arrived during my absence from Moscow in July.

The conclusions you express on pp.~51--54
are in full agreement with what I said in the introduction to my book:
\begin{quotation}
  In the pertinent mathematical circles it has been common for some time to 
  construct probability theory in accordance with this general point of view.
  But a complete presentation of the whole system, free from superfluous complications,
  has been missing\dots
\end{quotation}

You are also right to attribute to me (on p.~42) the opinion that the formal axiomatization
should be accompanied by an analysis of its real meaning.
Such an analysis is given, perhaps too briefly,
in the section ``The relation to the world of experience'' in my book.
Here I insist on the view, expressed by Mr.~von Mises himself
(Wahrscheinlickeitsrechnung 1931, pp.~21--26),
that ``collectives'' are \underline{finite} (though very large) in real practice.

One can therefore imagine three theories:
\begin{enumerate}\renewcommand{\labelenumi}{\Alph{enumi}}
   \item
     A theory based on the notions of ``very large'' finite ``collectives'', ``approximate''
     stability of frequencies, etc.  This theory uses ideas that cannot be defined in a
     purely formal (i.e., mathematical) manner,
     but it is the only one to reflect experience truthfully.
   \item
     A theory based on infinite collectives and limits of frequencies.
     After Mr.~Wald's work we know that this theory can be developed
     in a purely formal way without contradictions.
     But in this case its relation to experience cannot have any different nature than for 
     any other axiomatic theory.
     So in agreement with Mr.~von Mises, we should regard theory B
     as a certain ``mathematical idealization'' of theory A.
   \item
     An axiomatic theory of the sort proposed in my book.
     Its practical value can be deduced directly from the ``approximate'' theory A
     without appealing to theory B.
     This is the procedure that seems simplest to me.
\end{enumerate}

\noindent
Yours cordially,

\noindent
A. Kolmogoroff

\subsection{A closer look at L\'evy's example}\label{app:garb}

Corresponding to the partition $\mathfrak{A}$ considered by L\'evy is
the $\sigma$-algebra $\GGG$ consisting 
of the Borel sets $E$ in $[0,1]$ such that
\begin{equation}\label{eq:shift}
  \left.
    \begin{array}{r}
      x \in E\\
      x'-x \mbox{ is rational}
    \end{array}
  \right\}
  \Longrightarrow
  x' \in E.
\end{equation}
Each of L\'evy's $C(x)$ is in $\GGG$, and conditional expectation
with respect to the partition $\mathfrak{A}$ in Kolmogorov's sense
is the same as conditional expectation with respect to $\GGG$
in Doob's sense.

\begin{lemma}\label{lem:1}
  For any rational number $a=k/n$
  ($k$ and $n$ are integers, $k\ge0$, and $n\ge1$),
  $$
    \Prob([0,a] \given \GGG)
    =
    a
    \quad
    \mbox{a.s.},
  $$
  where $\Prob$ is uniform on $[0,1]$.
\end{lemma}
\begin{proof}
  We must show that
  \begin{equation}\label{eq:1}
    \int_E
      a
    dx
    =
    \Prob([0,a]\cap E)
  \end{equation}
  for all $E \in \GGG$.
  Consider the sets
  \begin{align*}
    E_1 &:= E \cap \left[ 0, \frac{1}{n}\right],\\
    E_2 &:= E \cap \left[ \frac{1}{n}, \frac{2}{n}\right],\\
        &\vdots\\
    E_n &:= E \cap \left[ \frac{n-1}{n}, 1\right].
  \end{align*}
  Since $E$ is invariant with respect to adding $\frac{1}{n}$ (modulo 1),
  these sets are congruent.
  So the left-hand side of~(\ref{eq:1}) is $an\Prob(E_1)$
  while the right-hand side is $k\Prob(E_1)$;
  since $k = an$, these are equal.
  \qedtext
\end{proof}

\begin{lemma}\label{lem:2}
 For any bounded Borel function $f:[0,1]\to\bbbr$,
  \begin{equation}\label{eq:2}
    \Expect(f\given\GGG)
    =
    \Expect(f)
    \quad
    \mbox{a.s.}
  \end{equation}
\end{lemma}
\begin{proof}
  Equation~(\ref{eq:2}) means that
  \begin{equation}\label{eq:3}
    \int_E f dx
    =
    \Expect(f) \Prob(E)
  \end{equation}
  for any $E\in\GGG$.
  By Luzin's theorem \citep{kolmogorov/fomin:1999},
  for any $\epsilon>0$ there exists a continuous function $g:[0,1]\to\bbbr$ such 
that
  $$
    \Prob\{x\st f(x)\ne g(x)\} < \epsilon;
  $$
  so it suffices to prove (\ref{eq:3}) for continuous functions.
  We can easily do this using Lemma~\ref{lem:1}.
  \qedtext
\end{proof}

As a special case of Lemma~\ref{lem:2}, we have
  $$
    \Prob(B\given\GGG)
    =
    \Prob(B)
    \quad
    \mbox{a.s.}
  $$
for any Borel set $B\subseteq[0,1]$.  In other words, we can take 
$X$'s conditional distribution in L\'evy's example to be uniform, just like
its unconditional distribution, no matter which $C(x)$ we condition
on.  This is exceedingly unnatural, because the uniform distribution gives
the set on which we are conditioning probability zero.

Theorem~II.(89.1) of Rogers and Williams (1995\nocite{rogers/williams:2000}, 
p.~219) tells us that if a $\sigma$-algebra $\mathfrak{G}$ is
countably generated, then conditional probabilities with respect
to $\GGG$ will almost surely 
give the event on which we are conditioning
probability one and will be equally well behaved in many other respects.
But the $\sigma$-algebra $\mathfrak{G}$ characterized 
by~(\ref{eq:shift}) is not countably generated.

\subsection*{Acknowledgments}

This article expands on part of Chapter~2 of Shafer and Vovk 2001\nocite{shafer/vovk:2001}.
Shafer's research was partially supported by NSF grant SES-9819116 to Rutgers University.
Vovk's research was partially supported by EPSRC grant~GR/R46670/01,
BBSRC grant 111/BIO14428, EU grant~IST-1999-10226,
and MRC grant S505/65 to Royal Holloway, University of London.

We have benefited from conversation 
and correspondence with Bernard Bru, Pierre Cr\'epel, Elyse Gustafson, Sam Kotz,
Steffen Lauritzen, Per Martin-L\"of, 
Thierry Martin, Laurent Mazliak, Paul Miranti, Julie Norton, Nell Painter,
Goran Peskir, Andrzej Ruszczynski, Oscar Sheynin, J.~Laurie Snell,
Stephen M. Stigler, and Jan von Plato.  Bernard Bru gave us an advance look at
Bru 2003a\nocite{bru:2003a}
and provided many helpful comments and insightful judgments, some
of which we may have echoed without adequate acknowledgement.
Oscar Sheynin was also exceptionally helpful, providing many useful
insights as well as direct access to his extensive translations,
which are not yet available in a United States library.

Several people have helped us locate sources.
Vladimir V'yugin helped us locate
the original text of Kolmogorov 1929\nocite{kolmogorov:1929},
and Aleksandr Shen' gave us a copy of the 1936 translation of
the \emph{Grundbegriffe} into Russian.  
Natalie Borisovets, at Rutgers's Dana Library, and Mitchell Brown, at Princeton's Fine Library,
have both been exceedingly helpful in locating other references.

We take full responsibility for our translations into English
of various passages from French, German, Italian, and Russian,
but in some cases we were able to consult previous translations.

Although this article is based almost entirely on published sources,
extensive archival material is available to 
those interested in further investigation.  
The Accademia dei Lincei at Rome has a Castelnuovo archive
(Gario 2001\nocite{gario:2001}).
There is an extensive Fr\'echet archive at 
the Acad\'emie des Sciences in Paris.
L\'evy's papers were lost when his Paris apartment was ransacked by the Nazis,
but his extant correspondence includes letters
exchanged with Fr\'echet (Barbut, Locker, and Mazliak 2004\nocite{barbut/locker/mazliak:2004}),
Kai Lai Chung (in Chung's possession),
and Michel Lo\`eve (in Bernard Bru's possession).
Additional correspondence of L\'evy's
is in the library of the University of Paris at Jussieu
and in the possession of his family.
The material at Kolmogorov and Aleksandrov's country home at Komarovka
is being cataloged under the direction of Albert N. Shiryaev.
Doob's papers, put into order by Doob himself,
are accessible to the public at the University of Illinois, Champagne-Urbana.

Volumes 1--195 of the \emph{Comptes rendus hebdomadaires des s\'eances de l'Acad\'emie des Sciences}
(1885--1932) are available free on-line at the Biblioth\`eque Nationale de 
France (http://gallica.bnf.fr/).

\subsection*{Lifespans}

\noindent
George Biddell Airy (1801--1892)

\noindent
\hangindent=\parindent
Aleksandr Danilovich Aleksandrov (1912--1999)
(Александр Данилович Александров)

\noindent
Pavel Sergeevich Aleksandrov (1896--1982)
(Павел Сергеевич Александров)

\noindent
Erik Sparre Andersen (1919--2003)

\noindent
Oskar Johann Victor Anderson (1887--1960)
(Оскар Николаевич Андерсон)

\noindent
Vladimir Igorevich Arnol'd (born 1937)
(Владимир Игоревич Арнольд)

\noindent
Louis Jean-Baptiste Alphonse Bachelier (1870--1946)

\noindent
Stefan Banach (1892--1945)

\noindent
Marc Barbut (born 1928)

\noindent
Maya Bar-Hillel

\noindent
I. Alfred Barnett (1894--1975)

\noindent
Jack Barone

\noindent
Maurice Stephenson Bartlett (1910--2002)

\noindent
Grigory Minkelevich Bavli (1908--1941)
(Григорий Минкелевич Бавли)

\noindent
Raymond Bayer (born 1898)

\noindent
Margherita Benzi (born 1957)

\noindent
Jacob Bernoulli (1654--1705) 

\noindent
Claude Bernard (1813--1878)

\noindent
Felix Bernstein (1878--1956)

\noindent
Sergei Natanovich Bernstein (1880--1968) 
(Сергей Натанович Бернштейн)

\noindent
Joseph Louis Fran\c{c}ois Bertrand (1822--1900)

\noindent
Nic H. Bingham

\noindent
George David Birkhoff (1884--1944)

\noindent
David Blackwell (born 1919)

\noindent
Alain Blum

\noindent
Georg Bohlmann (1869--1928) 

\noindent
Ludwig Eduard Boltzmann (1844--1906)

\noindent
George Boole (1815--1864)

\noindent
\'Emile F\'elix-Edouard-Justin Borel (1871--1956) 

\noindent
Ladislaus von Bortkiewicz (1868--1931)
(Владислав Иосифович Борткевич)

\noindent
Marcel Brissaud

\noindent
Ugo Broggi (1880--1965) 

\noindent
Robert Brown (1773--1858)

\noindent
Bernard Bru (born 1942)

\noindent
Ernst Heinrich Bruns (1848--1919)

\noindent
Stephen George Brush

\noindent
George-Louis Leclerc de Buffon (1707--1788) 

\noindent
Francesco Paolo Cantelli (1875--1966) 

\noindent
Constantin Carath\'eodory (1873--1950) 

\noindent
Rudolf Carnap (1891--1970)

\noindent
Nicolas L\'eonard Sadi Carnot (1796--1832)

\noindent
Guido Castelnuovo (1865--1952) 

\noindent
Carl Wilhelm Ludvig Charlier (1862--1934)

\noindent
Kai Lai Chung (born 1917)

\noindent
\hangindent=\parindent
Aleksandr Aleksandrovich Chuprov (1874--1926)
(Александр Александрович Чупров)

\noindent
Alonzo Church (1903--1995)

\noindent
Cifarelli, Donato Michele 

\noindent
Auguste Comte (1798--1857)

\noindent
Julian Lowell Coolidge (1873--1954)

\noindent
Arthur H. Copeland Sr. (1898--1970) 

\noindent
Antoine-Augustin Cournot (1801--1877) 

\noindent
Jean-Michel Courtault

\noindent
Thomas Merrill Cover (born 1938)

\noindent
Richard T. Cox (1898--1991) 

\noindent
Harald Cram\'er (1893--1985) 

\noindent
Pierre Cr\'epel (born 1947)

\noindent
Emanuel Czuber (1851--1925) 

\noindent
Jean-le-Rond D'Alembert (1717--1783)

\noindent
Percy John Daniell (1889--1946) 

\noindent
Lorraine Daston (born 1951)

\noindent
Bruno {d}e Finetti (1906--1985) 

\noindent
Claude Dellacherie (born 1943)

\noindent
Sergei S. Demidov (born 1942)
(Сергей Демидов)

\noindent
Abraham De Moivre (1667--1754) 

\noindent
Augustus De Morgan (1806--1871)

\noindent
Jean Alexandre Eug\`ene Dieudonn\'e (1906--1992)

\noindent
Wolfgang Doeblin (1915--1940)

\noindent
Joseph L. Doob (1910--2004)

\noindent
Karl D\"orge (1899--1977) 

\noindent
Louis-Gustave Du Pasquier (1876--1957)

\noindent
Peter Larkin Duren (born 1935)

\noindent
Evgeny Borisovich Dynkin (born 1924)
(Евгений Борисович Дынкин)

\noindent
Francis Ysidro Edgeworth (1845--1926)

\noindent
Albert Einstein (1879--1955) 

\noindent
Robert Leslie Ellis (1817--1859) 

\noindent
Agner Krarup Erlang (1878--1929)

\noindent
Georg Faber (1877--1966) 

\noindent
Ruma Falk

\noindent
Gustav Theodor Fechner (1801--1887) 

\noindent
Willy Feller (1906--1970) (William after his immigration to the U.S.)

\noindent
Arne Fisher (1887--1944)

\noindent
Ronald Aylmer Fisher (1890--1962) 

\noindent
Sergei Vasil'evich Fomin (1917--1975)
(Сергей Васильевич Фомин)

\noindent
Robert M. Fortet (1912--1998)

\noindent
Jean Baptiste Joseph Fourier (1768--1830)

\noindent 
Abraham Adolf Fraenkel (1891--1965)

\noindent
R\'en\'e-Maurice Fr\'echet (1878--1973) 

\noindent
John Ernst Freund (born 1921)

\noindent
Hans Freudenthal (1905--1990)

\noindent
Thornton Carl Fry (1892--1991) 

\noindent
Peter G\'acs (born 1947)

\noindent
Joseph Mark Gani (born 1924)

\noindent
R\'en\'e G\^ateaux (1889--1914)

\noindent
Hans Geiger (1882--1945)

\noindent
Valery Ivanovich Glivenko (1897--1940)
(Валерий Иванович Гливенко)

\noindent
\hangindent=\parindent
Boris Vladimirovich Gnedenko (1912--1995)
(Борис Владимирович Гнеденко)

\noindent
\hangindent=\parindent
Vasily Leonidovich Goncharov (1896--1955)
(Василий Леонидович Гончаров)

\noindent
William Sealy Gossett (1876--1937)

\noindent
Jorgen Pederson Gram (1850--1916)

\noindent
Robert Molten Gray (born 1943)

\noindent
Shelby Joel Haberman (born 1947)

\noindent
Ian Hacking (born 1936)

\noindent
Malachi Haim Hacohen

\noindent
Jacques Hadamard (1865-1963) 

\noindent
Maurice Halbwachs (1877--1945)

\noindent
Anders Hald (born 1913) 

\noindent
Paul Richard Halmos (born 1916)

\noindent
Joseph Y. Halpern

\noindent
Godfrey H. Hardy (1877--1947)

\noindent
Felix Hausdorff (1868--1942) 

\noindent
Thomas Hawkins (born 1938) 

\noindent
Georg Helm (1851--1923)

\noindent
Christopher Charles Heyde (born 1939)

\noindent
David Hilbert (1862--1943) 

\noindent
Theophil Henry Hildebrandt (born 1888)

\noindent
Eberhard Hopf (1902--1983)

\noindent
Philip Holgate (1934--1993) 

\noindent
Harold Hotelling (1895--1973)

\noindent
Kiyosi It\^o (born 1915)

\noindent
Harold Jeffreys (1891--1989) 

\noindent
B{\o}rge Jessen (1907--1993)

\noindent
Normal Lloyd Johnson (born 1917)

\noindent
Fran\c{c}ois Jongmans

\noindent
Camille Jordan (1838--1922) 

\noindent
Youri Kabanov
(Юрий Кабанов)

\noindent
Mark Kac (1914--1984)

\noindent
Jean-Pierre Kahane (born 1926) 

\noindent
Immanuel Kant (1724-1804)

\noindent
John Maynard Keynes (1883--1946) 

\noindent
\hangindent=\parindent
Aleksandr Yakovlevich Khinchin (1894--1959)
(Александр Яковлевич Хинчин)

\noindent
Eberhard Knobloch

\noindent
\hangindent=\parindent
Andrei Nikolaevich Kolmogorov (1903--1987)
(Андрей Николаевич Колмогоров)

\noindent
Bernard O. Koopman (1900--1981)

\noindent
Samuel Borisovich Kotz (born 1930)

\noindent
Ulrich Krengel (born 1937)

\noindent
Sylvestre-Fran\c{c}ois Lacroix (1765--1843) 

\noindent
Rudolf Laemmel (1879--1972) 

\noindent
Pierre Simon de Laplace (1749--1827) 

\noindent
Steffen Lauritzen (born 1947)

\noindent
\hangindent=\parindent
Mikhail Alekseevich Lavrent'ev (1900--1980)
(Михаил Алексеевич Лаврентьев)

\noindent
Henri Lebesgue (1875--1941) 

\noindent
\hangindent=\parindent
Mikhail Aleksandrovich Leontovich (1903--1981)
(Михаил Александрович Леонтович)

\noindent
Paul Pierre L\'evy (1886--1971) 

\noindent
Simon Antoine Jean Lhuilier (1750--1840)

\noindent
Ming Li (born 1955)

\noindent
Jean-Baptiste-Joseph Liagre (1815--1891)

\noindent
G. Lindquist

\noindent
Michel Lo\`eve (1907--1979)

\noindent
Bernard Locker (born 1947)

\noindent
Antoni Marjan {\L}omnicki (1881--1941) 

\noindent
Zbigniew {\L}omnicki (born 1904) 

\noindent
Jerzy {\L}o\'s (born 1920)

\noindent
J. Loveland

\noindent
Jan {\L}ukasiewicz (1878--1956)

\noindent
Ernest Filip Oskar Lundberg (1876--1965)

\noindent
Nikolai Nikolaevich Luzin (1883--1950)
(Николай Николаевич Лузин)

\noindent
Leonid Efimovich Maistrov (born 1920)
(Леонид Ефимович Майстров)

\noindent
Hugh MacColl (1837--1909) 

\noindent
\hangindent=\parindent
Andrei Aleksandrovich Markov (1856--1922)
(Андрей Александрович Марков)

\noindent
Thierry Martin (born 1950)

\noindent
Per Martin-L\"of (born 1942) 

\noindent
Pesi R. Masani

\noindent
Laurent Mazliak (born 1968)

\noindent
Paolo Medolaghi (1873--1950)

\noindent
Alexius Meinong (1853--1920)

\noindent
Karl Menger (1902--1985)

\noindent
Martine Mespoulet 

\noindent
Paul-Andr\'e Meyer (1934--2003)

\noindent
Philip Mirowski (born 1951)

\noindent
Edward Charles Molina (born 1877)

\noindent
Jacob Mordukh (born 1895)

\noindent
Ernest Nagel (1901--1985)

\noindent
Jerzy Neyman (1894--1981) 

\noindent
Otton Nikodym (1889--1974) 

\noindent
Albert Novikoff

\noindent
Octav Onicescu (1892--1983)

\noindent
Kheimerool O. Ondar (born 1936)
(Хеймероол Опанович Ондар)

\noindent
Egon Sharpe Pearson (1895--1980)

\noindent
Karl Pearson (1857--1936)

\noindent
Charles Saunders Peirce (1839--1914) 

\noindent
Ivan V. Pesin
Иван Песин

\noindent
B.~V. Pevshin
(Б. В. Певшин)

\noindent
Jean-Paul Pier

\noindent
Jules Henri Poincar\'e (1854--1912) 

\noindent
Sim\'eon-Denis Poisson (1781--1840)

\noindent
Karl R. Popper (1902--1994) 

\noindent
Pierre Prevost (1751--1839)

\noindent
Johann Radon (1887--1956) 

\noindent
Hans Rademacher (1892--1969)

\noindent
Frank Plumpton Ramsey (1903--1930)

\noindent
Eugenio Regazzini (born 1946)

\noindent
Hans Reichenbach (1891--1953) 

\noindent
Alfr\'ed R\'enyi (1921--1970)

\noindent
Hans Richter (born 1912)

\noindent
Henry Lewis Rietz (1875--1943)

\noindent
Leonard Christopher Gordon Rogers 

\noindent
Bertrand Russell(1872--1970)

\noindent
Ernest Rutherford (1871--1937)

\noindent
Stanis{\l}aw Saks (1847--1942)

\noindent
Leonard Jimmie Savage (1917--1971)

\noindent
Ivo Schneider (born 1938)

\noindent
Laurent Schwartz (1915--2002)

\noindent
Irving Ezra Segal (1918--1998)

\noindent
Stanford L. Segal (born 1937)

\noindent
Eugene Seneta (born 1941)

\noindent
Oscar Sheynin (born 1925)\nocite{sheynin:1998a,sheynin:1998b,sheynin:1999a,sheynin:1999b,sheynin:2000}
(Оскар Борисович Шейнин)

\noindent
Albert Nikolaevich Shiryaev (born 1934)
(Альберт Николаевич Ширяев)

\noindent
Wac{\l}aw  Sierpi\'nski (1882--1969) 

\noindent
Evgeny Slutsky (1880--1948)
(Евгений Слуцкий)

\noindent
Aleksei Bronislavovich Sossinsky (born 1937)

\noindent
Charles M. Stein (born 1920)

\noindent
Margaret Stein

\noindent
Hans-Georg Steiner (born 1928)

\noindent
Hugo Dyonizy Steinhaus (1887--1972) 

\noindent
Stephen Mack Stigler (born 1941)

\noindent
Erling Sverdrup (1917--1994)

\noindent
Angus E. Taylor (1911--1999)

\noindent
Thorvald Nicolai Thiele (1838--1910)

\noindent
Erhard Tornier (1894--1982) 

\noindent
Mark R. Tuttle

\noindent
Stanis{\l}aw Ulam (1909--1984)

\noindent
Friedrich Maria Urban (1878--1964)

\noindent
James Victor Uspensky (1883--1947)

\noindent
David Van Dantzig (1900--1959)

\noindent
Edward Burr Van Vleck (1863--1943)

\noindent
John Venn (1834--1923) 

\noindent
Jean-Andr\'e Ville (1910--1988) 

\noindent
Paul Vit\'anyi (born 1944)

\noindent
V. M. Volosov
(В. М. Волосов)

\noindent
Johannes von Kries (1853--1928)

\noindent
Richard Martin Edler von Mises (1883--1953) 

\noindent
John von Neumann (1903--1957)\nocite{vonneumann:1932}

\noindent
Raymond Edward Alan Christopher Paley (1907--1933)

\noindent
Jan von Plato (born 1951) 

\noindent
Marian von Smoluchowski (1872--1917) 

\noindent
Vito Volterra (1860--1940)

\noindent
Alexander Vucinich (1914--2002)

\noindent
Abraham Wald (1902--1950) 

\noindent
Rolin Wavre (1896--1949)

\noindent
Harald Ludvig Westergaard (1853--1936)

\noindent
Peter Whittle (born 1927)

\noindent
William Allen Whitworth (1840--1905)

\noindent
Norbert Wiener (1894--1964) 

\noindent
David Williams (born 1938)

\noindent
Anders Wiman (1865--1959) 

\noindent
William Henry Young (1863--1942)

\noindent
Smilka Zdravkovska

\noindent
Joseph David Zund (born 1939)

\noindent
Antoni Zygmund (1900--1992)

\end{document}